\documentclass{amsart}

\usepackage{amsmath,amsfonts,amssymb,amsthm,bm,mathrsfs}
\usepackage{xparse}

\usepackage[usenames,dvipsnames]{color}
\usepackage{xcolor}

\usepackage{tikz-cd}
\usepackage{spath3}
\usetikzlibrary{knots, braids}

\usepackage{graphicx}
\usepackage{import}

\usepackage[colorlinks=true, linktocpage=true]{hyperref}
\usepackage{cleveref}

\newtheorem{proposition}{Proposition}[section]
\newtheorem{bigtheorem}{Theorem}
\newtheorem{theorem}[proposition]{Theorem}
\newtheorem{lemma}[proposition]{Lemma}
\newtheorem{corollary}[proposition]{Corollary}

\newtheorem{conjecture}{Conjecture}

\theoremstyle{definition}
\newtheorem{definition}[proposition]{Definition}

\newtheorem{remark}[proposition]{Remark}
\newtheorem{example}[proposition]{Example}
\newtheorem{conventions}[proposition]{Conventions}

\newtheoremstyle{TheoremNum}
  {\topsep}{\topsep}  %
  {\itshape}          %
  {}                  %
  {\bfseries}         %
  {.}                 %
  { }                  %
  {\thmname{#1}\thmnote{ \bfseries #3}}
\theoremstyle{TheoremNum}
\newtheorem{theoremn}{Theorem}

\renewcommand{\epsilon}{\varepsilon}
\newcommand{\ts}{\otimes}
\DeclareMathOperator{\im}{Im}
\DeclareMathOperator{\End}{End}

\DeclareMathOperator{\tr}{tr}
\DeclareMathOperator{\id}{id}

\renewcommand{\ker}{\operatorname{Ker}}
\renewcommand{\hom}{\operatorname{Hom}}

\DeclareMathOperator{\spec}{Spec}

\newcommand{\iso}{\cong}
\newcommand{\defeq}{:=}

\renewcommand{\epsilon}{\varepsilon}
\renewcommand{\phi}{\varphi}

\makeatletter
\newcommand{\extp}{\@ifnextchar
^\@extp{\@extp^{\,}}}
\def\@extp^#1{\mathop{\bigwedge\nolimits^{\!#1}}}
\makeatother

\newcommand{\C}{\mathbb{C}}
\newcommand{\Q}{\mathbb{Q}}
\newcommand{\Z}{\mathbb{Z}}

\newcommand{\lie}[1]{\mathfrak{#1}}

\newcommand{\GL}{\operatorname{GL}}

\newcommand{\slgroup}{\operatorname{SL}_2(\C)}
\newcommand{\slgroupdual}{\slgroup^*}
\newcommand{\repvar}[1]{\mathfrak{R}(#1)}

\newcommand{\rmat}{\mathcal{R}}
\newcommand{\algbraid}{\check{\mathcal{R}}}
\newcommand{\algbraidd}{\check{\overline{\mathcal{R}}}}

\newcommand{\braid}{\mathbb{B}}
\NewDocumentCommand{\braidf}{D[]{\slgroup{}} }{\braid{}(#1^*)}
\NewDocumentCommand{\braidfh}{D[]{\slgroup{}} }{\hat{\braid}{}(#1^*)}

\newcommand{\gl}{\operatorname{GL}}

\newcommand{\burau}{\mathcal{B}}
\newcommand{\redburau}{\overline{\burau}}
\NewDocumentCommand{\homol}{D[]{1} m D[]{} }{\operatorname{H}_#1(#2)^{#3}}
\newcommand{\lf}{\mathrm{lf}}
\newcommand{\vect}[1]{\mathsf{Vect}(#1)}

\newcommand{\U}{\mathcal{U}}

\newcommand{\irrmod}[1]{V(#1)}
\NewDocumentCommand{\dirrmod}{D[]{} m}{W_{#1}(#2)}
\newcommand{\adjmod}[1]{\U/\ker{#1}}

\newcommand{\Ucenter}{\mathcal{Z}}
\newcommand{\Ucentersmall}{\mathcal{Z}_0}

\newcommand{\opspace}{\mathfrak{H}}
\newcommand{\clifford}{\mathfrak{C}}
\newcommand{\dopspace}[1]{\mathscr{H}_{#1}}
\newcommand{\dclifford}[1]{\mathscr{C}_{#1}}

\newcommand{\F}{\tilde{F}}
\newcommand{\ket}[1]{\left| #1 \right \rangle}
\newcommand{\bra}[1]{\left\langle #1 \right|}

\newcommand{\ev}[1]{\operatorname{ev}_{#1}}
\newcommand{\evbar}[1]{\overline{\operatorname{ev}}_{#1}}
\newcommand{\coev}[1]{\operatorname{coev}_{#1}}
\newcommand{\coevbar}[1]{\overline{\operatorname{coev}}_{#1}}

\newcommand{\op}{\mathrm{op}}
\newcommand{\cop}{\mathrm{cop}}

\newcommand{\mainfunctor}{\mathcal{F}}
\newcommand{\doubledfunctor}{\mathcal{T}}
\newcommand{\scalarfunctor}{\mathcal{K}}
\newcommand{\algfunctor}{\mathcal{A}}
\newcommand{\modc}[1]{#1\mathsf{-Mod}}
\newcommand{\modcat}{\mathcal{C}}
\newcommand{\doubledcat}{\mathcal{D}}
\newcommand{\maininv}{\nabla}
\newcommand{\casimirsystem}{\bm \mu}

\newcommand{\tsunit}{\mathbf{1}}
\newcommand{\parmod}{\Pi}
\newcommand{\proj}{\operatorname{Proj}}
\newcommand{\cat}{\mathscr{C}}

\newcommand{\rentr}{{\mathbf{t}}}
\newcommand{\rendim}{{\mathbf{d}}}
\DeclareMathOperator{\str}{str}

\crefname{theoremn}{Theorem}{Theorems}
\crefname{bigtheorem}{Theorem}{Theorems}
\crefname{proposition}{Proposition}{Propositions}
\crefname{theorem}{Theorem}{Theorems}
\crefname{lemma}{Lemma}{Lemmas}
\crefname{corollary}{Corollary}{Corollaries}
\crefname{conjecture}{Conjecture}{Conjectures}
\crefname{definition}{Definition}{Definitions}
\crefname{remark}{Remark}{Remarks}
\crefname{example}{Example}{Examples}
\crefname{conventions}{Conventions}{Conventions}
\crefname{figure}{Figure}{Figures}
\crefname{appendix}{Appendix}{Appendices}

\usepackage[style=numeric]{biblatex}
\title{Holonomy invariants of links and nonabelian Reidemeister torsion}
\author{Calvin McPhail-Snyder}
\address{Department of Mathematics, University of California, Berkeley, California 94720-3840}
\email{cmcs@math.berkeley.edu}

\begin{document}
\begin{abstract}
  We show that the reduced $\operatorname{SL}_2(\mathbb{C})$-twisted Burau representation can be obtained from the quantum group $\mathcal{U}_q(\mathfrak{sl}_2)$ for $q = i$ a fourth root of unity and that representations of $\mathcal{U}_q(\mathfrak{sl}_2)$ satisfy a type of Schur-Weyl duality with the Burau representation.
  As a consequence, the $\operatorname{SL}_2(\mathbb{C})$-twisted Reidemeister torsion of links can be obtained as a quantum invariant.
  Our construction is closely related to the quantum holonomy invariant of Blanchet, Geer, Patureau-Mirand, and Reshetikhin \cite{Blanchet2018}, and we interpret their invariant as a twisted Conway potential.
\end{abstract}
\maketitle
\tableofcontents

\section{Introduction}
Let $X$ be a space and $G$ a Lie group.
We can capture geometric information about $X$ by equipping it with a representation $\rho : \pi_1(X) \to G$, considered up to conjugation.\footnote{This data is equivalently described by a $G$-local system on $X$ or a gauge class of flat $\lie g$-connections.}
In this paper we consider the case of $X = S^3 \setminus L$ a link complement and $G = \slgroup$.
We call the pair $(L, \rho)$ of the link $L$ and representation $\rho : \pi_L \to \slgroup$ a \emph{$\slgroup$-link,} where $\pi_L \defeq \pi_1(S^3 \setminus L)$ is the fundamental group of the complement.

To extend the representation of links as braid closures to this context, we use the idea of a \emph{colored braid}. 
Express the link $L$ as the closure of a braid $\beta$ on $n$ strands.
Topologically, we can think of $\beta$ as an element of the mapping class group of an $n$-punctured disc $D_n$.
Because $\pi_1(D_n)$ is a free group, we can equip the disc $D_n$ with a representation $\rho : \pi_1(D_n) \to \slgroup$ by picking \emph{colors} $g_i \in \slgroup$, with $g_i$ giving the holonomy of a path going around the $i$th puncture.

The braid $\beta$ acts on the colors by mapping $\rho$ to the representation $\rho \beta^{-1}$.
If $L$ is the closure of $\beta$, the representation $\rho$ extends to a representation of the complement of $L$ exactly when $\rho = \rho \beta^{-1}$.
This perspective is one way to obtain invariants of $G$-links.
The braid group (as the mapping class group of $D_n$) acts on the $\rho$-twisted homology of $D_n$.
In particular, its action on $\homol{D_n; \rho}$ is the twisted Burau representation, which can be used to define the twisted Reidemeister torsion of $(L, \rho)$.

In this paper, we connect this story to the representation theory of the quantum group $\U = \U_q(\lie{sl}_2)$ at $q = i$ a fourth root of unity.
When $q$ is a root of unity $\U$ acquires a large central subalgebra $\Ucentersmall \subset \U$.
Previous work \cite{Kashaev2004,Kashaev2005,Weinstein1992} has shown that the variety of $\slgroup$-representations of $D_n$ is birationally equivalent to $\spec \Ucentersmall^{\otimes n}$.

\subsection{Schur-Weyl duality for the Burau representation}
We briefly describe this correspondence in order to state our first main result.
For more details, see~\S\ref{sec:colored-braid-groupoid}.
By work of \citeauthor{Kashaev2004} \cite{Kashaev2004} and \citeauthor{Blanchet2018} \cite[\S6]{Blanchet2018}, generic (in our terminology, \emph{admissible}) representations $\rho : \pi_1(D_n) \to \slgroup$ correspond to closed points of $\spec \Ucentersmall^{\otimes n}$, that is homomorphisms $\Ucentersmall^{\otimes n} \to \C$.
Any such homomorphism is of the form $\chi_1 \otimes \cdots \otimes \chi_n$ where $\chi_i : \Ucentersmall \to \C$.
The braid group acts on $\rho$ (via homemorphisms of $D_n$) and on $\chi_1 \otimes \cdots \otimes \chi_n$ (via an automorphism $\algbraid$ related to conjugation by the $R$-matrix) and these actions are compatible with this correspondence.

Our first major result is the extension of this relationship to Burau representations.
We summarize as follows:
\begin{theoremn}[\ref{thm:schur-weyl}]
  Let $\rho : \pi_1(D_n) \to \slgroup$ be a representation and $\beta$ a braid on $n$ strands that is \emph{nonsingular} and \emph{admissible}, and let $\rho' = \rho \beta^{-1}$ be the image of the representation under the action of $\beta$ on $D_n$.
  Then, for each $n \ge 2$:
  \begin{enumerate}
    \item 
      There exists a subspace $\opspace_n \subset \U[\Omega^{-1}]^{\otimes n}$ and a family of injective linear maps $\phi_\rho$ such that the diagram commutes:
      \[
        \begin{tikzcd}
          \homol{D_n; \rho}[\lf] \arrow[r, "\burau(\beta)"] \arrow[d, "\phi_\rho"] & \homol{D_n; \rho}[\lf] \arrow[d, "\phi_{\rho'}"] \\
          \opspace_n/\ker(\chi_1 \otimes \cdots \otimes \chi_n) \arrow[r, "\algfunctor(\beta)"] & \opspace_n/\ker(\chi_1' \otimes \cdots \otimes \chi_n')
        \end{tikzcd}
      \]
    \item The subspace $\opspace_n$ generates a Clifford algebra $\clifford_n$ inside $\U^{\otimes n}$ which super-commutes with $\Delta(\U)$, the image of $\U$ in $\U^{\otimes n}$ under the coproduct.
  \end{enumerate}
\end{theoremn}
In this theorem, $\burau(\beta)$ is the \emph{Burau representation}, the braid action on homology coming from the braid action on $D_n$ by homeomorphisms (\S\ref{subsec:burau-rep}), while $\algfunctor(\beta)$ is the braid action on $\U^{\otimes n}$ coming from the braiding $\algbraid$ on $\U$ (\S\ref{subsec:holonomy-braiding}).
We are specifically interested in the locally-finite (Borel-Moore) homology, denoted $\homol{D_n; \rho}[\lf]$ above.

We say that $\rho$ is \emph{nonsingular} if the holonomy $\rho(x_i)$ around a puncture never has $1$ as an eigenvalue.
This is a geometrically natural condition: it ensures that the torsion is nonzero (Proposition \ref{prop:torsion-computation}) and that the Casimir element $\Omega \in \U$ acts invertibly (Proposition \ref{rem:casimirs}).

The condition that $\rho$, $\beta$, and $\rho'$ be \emph{admissible} is related to the fact that the representation variety is only birationally equivalent to $\spec \Ucentersmall^{\otimes n}$ (see \S\ref{subsec:factorized-groups} for details).

Because of (2) we interpret \cref{thm:schur-weyl} as a Schur-Weyl duality between $\U_i(\lie{sl}_2)$ and the reduced $\slgroup$-twisted Burau representation.
This extends a similar result for $\U_q(\lie{gl}(1|1))$ and abelian $\slgroup$-representations due to \citeauthor{Reshetikhin2019} \cite{Reshetikhin2019}.

Usually Schur-Weyl duality is interpreted as a statement about \emph{modules}, as in \cref{cor:schur-weyl-modules}.
However, constructing the right modules for this to hold is somewhat delicate: we need a $G$-graded version of the quantum double.
For this reason we delay it to \S\ref{sec:quantum-double} and \S\ref{sec:proof-of-thm-2}.
These issues are discussed in more detail in \S\ref{subsec:strategy}.

\subsection{Gauge invariance}
In general, we are only interested in the representation $\rho : \pi_1(S^3 \setminus L) \to \slgroup$ up to conjugation.%
\footnote{Changing the basepoint of $S^3 \setminus L$ or changing basis in the space on which $\slgroup$ acts should not change the geometry of $S^3 \setminus L$, which is what $\rho$ is capturing. Similarly, if we obtain $\rho$ as the holonomy of a flat connection it only defined up to conjugation.}
We call conjugation $\rho \mapsto g \rho g^{-1}$ a \emph{gauge transformation} and say that $\rho$ and $g \rho g^{-1}$ are \emph{gauge equivalent}.
A well-behaved invariant $F$ of $\slgroup$-links should be \emph{gauge invariant}, in the sense that
\[
  F(L, \rho) = F(L, g \rho g^{-1})
\]
for any $g \in \slgroup$.
(This terminology comes from thinking of $\rho$ as the holonomy of a flat $\lie{sl}_2$ connection.)
The quantum holonomy invariant of \citeauthor{Blanchet2018} is gauge-invariant, as is the torsion.

Gauge-invariance lets us deal with the admissiblity hypothesis in Theorem \ref{thm:schur-weyl}: by Proposition \ref{prop:admissible-exist}, every  $\slgroup$-link is gauge-equivalent to a link with admissible representation.
We can therefore conjugate away from the singular, inadmissible representations that do not admit a description in the coordinates coming from~$\U_i(\lie{sl}_2)$.

\subsection{The nonabelian torsion is a quantum invariant}
As a consequence of the duality of Theorem \ref{thm:schur-weyl} we show that the $\slgroup$-twisted torsion of a link can be obtained as a \emph{quantum holonomy invariant}.
To say what this means, we first recall one definition of quantum invariant.

For $H$ a quasitriangular Hopf algebra and $V$ an $H$-module, the Reshetikhin-Turaev construction \cite{Reshetikhin1990} produces a functor
\[
  \mathcal{F} : \braid \to \modc{H}
\]
where we think of the disjoint union of the braid groups $\braid = \braid_1 \cup \braid_2 \cup \cdots$ as a category with objects $1, 2, \dots$.
The construction also gives a famliy of \emph{quantum traces} $\tr_q : \End_H(V^{\otimes n}) \to \C$.
If a link $L$ is the closure of a braid $\beta$, then the scalar
\[
  \mathcal{F}(L) \defeq \tr_q(\mathcal{F}(\beta))
\]
is an invariant of $L$ (ignoring technicalities like orientations, framings, etc.)

A quantum holonomy invariant of $G$-links is obtained from a $G$-graded version of this construction, namely a functor
\[
  \mathcal F : \braid(G) \to \modc{H},
\]
where $\braid(G)$ is the \emph{$G$-colored braid groupoid}, a variant of the braid group that keeps track of the representations $\pi_1(D_n) \to G$.
Similarly, $H$ needs to be appropriately $G$-graded; in our case, this will come from the central subalgebra $\Ucentersmall$ of $\U_i(\lie{sl}_2)$.

If $\beta$ is a colored braid whose closure is the link $(L, \rho)$, then the quantum trace $\tr_q(\mathcal{F}(\beta))$ will again be an invariant of $(L, \rho)$.
Actually, this is not quite true: in general, $\tr_q(\mathcal{F}(\beta))$ might depend on the writhe of $\beta$, that is the framing of its closue.

\citeauthor{Blanchet2018} \cite{Blanchet2018} constructed a nontrivial family of holonomy invariants for $G = \slgroup$ by using the representation theory of $\U_\xi(\lie{sl}_2)$ for $\xi$ a root of unity.
We denote the $\xi = i$ case of their invariant by $\mainfunctor$.
In \S\ref{sec:quantum-double}, we define a quantum holonomy invariant $\doubledfunctor$ which is roughly the ``norm-square'' or ``quantum double'' of $\mainfunctor$, up to a change in normalization.
To define $\mainfunctor$ and $\doubledfunctor$ we need to make an extra choice of square roots  ${\bm \mu} = \{\mu_i\}$ of the eigenvalues of the meridians of $L$.

\begin{theoremn}[\ref{thm:T-is-torsion}]
  Let $(L, \rho)$ be an $\slgroup$-link and $\rho$ an admissible representation with $\det ( 1- \rho(x) ) \ne 0$ for every meridian  $x$ of $L$.
  Choose square roots $\casimirsystem = \{\mu_i\}$ of the eigenvalues of the meridian of each component $L_i$ of $L$.
  Then
  \[
    \doubledfunctor(L, \rho, \casimirsystem) = \tau(L, \rho),
  \]
  where $\tau(L, \rho)$ is the $\slgroup$-twisted Reidemeister torsion of $S^3 \setminus L$.
\end{theoremn}
This theorem is a direct consequence of Schur-Weyl duality for the Burau representation (Theorem \ref{thm:schur-weyl}) and the definition of $\doubledfunctor$.
Up to sign $\tau(L, \rho)$ does not depend on the extra choices $\casimirsystem$ or on the framing of $L$, so neither does $\doubledfunctor(L, \rho, \casimirsystem)$.

There are two technical hypotheses in Theorem \ref{thm:T-is-torsion}: the colored link $L$ must admit a presentation as the closure of an admissible braid (see \S\ref{sec:colored-braid-groupoid}) and $\rho(x)$ cannot have $1$ as an eigenvalue for any meridian $x$ of $L$.
The first, which is related to the fact that $\slgroupdual$ is only birationally equivalent to $\slgroup$, is not particularly important, because every $(L, \rho)$  is gauge-equivalent (conjugate) to one with an admissible braid presentation.
The second condition is expected, because the torsion can be ill-defined when $\det(1 - \rho(x)) = 0$ for meridians $x$ of $L$.

\subsection{The relationship between \texorpdfstring{$\mainfunctor$}{F} and \texorpdfstring{$\doubledfunctor$}{T}}
In \S\ref{subsec:mirror-image} we define $\overline{\mainfunctor}$ as a dual version of $\mainfunctor$, and it is immediate from the definition that
\[
  \overline{\mainfunctor}(L, \rho, {\bm \mu}) = \mainfunctor(\overline{L}, \rho, {\bm \mu} ),
\]
where $\overline{L}$ is the mirror image of the link $L$.

We would like to say that $ \doubledfunctor(L, \rho, {\bm \mu}) = \mainfunctor(L, \rho, {\bm \mu}) \overline{\mainfunctor}(L, \rho, {\bm \mu})$, but unfortunately this is not true.
For technical reasons detailed in \S\ref{subsec:module-braiding}, the $R$-matrix  of quantum $\lie{sl}_2$ only defines a \emph{projective} braid action on $\U_i$-modules.
To define link invariants we need to lift this to a genuine representation, but doing so is a rather difficult technical problem.

\citeauthor{Blanchet2018} \cite{Blanchet2018} partially solve this problem and show that the scalar ambiguity can at least be reduced to a fourth root of unity (in the case $\xi = i$ we consider in this paper).
However, to obtain the relationship with the torsion, we need to use a different normalization.
This change in normalization can be captured by an invariant we denote $\scalarfunctor$ which (\cref{thm:scalar-invariant}) satisfies
\begin{equation}
  \label{eq:scalar-inv-def}
  \mainfunctor(L, \rho, {\bm \mu}) 
  \overline{\mainfunctor}(L, \rho, {\bm \mu}) 
  \scalarfunctor(L, \rho)
  =
  \doubledfunctor(L, \rho, {\bm \mu})
\end{equation}
up to a power of $i$.
We can think of $\scalarfunctor$ as an anomaly, and it comes from a \emph{scalar} representation
\[
  \scalarfunctor : \braid(\slgroup) \to \operatorname{GL}_1(\C)
\]
as opposed to $\mainfunctor$ and $\doubledfunctor$, which we can think of as taking values in $\operatorname{GL}_N(\C)$ for $N > 1$.

We do not have a good characterization of $\scalarfunctor$ other than \eqref{eq:scalar-inv-def}.
However, since it is already rather difficult to compute the value of $\mainfunctor$ (other than numerically) an independent definition of $\scalarfunctor$ is not particularly useful.
We expect (\cref{rem:future-work}) that the results of \cite{McPhailSnyderUnpub1} will clarify this situation and allow us to choose a new normalization $\mainfunctor'$ of $\mainfunctor$ that resolves some of these issues.
We discuss these issues in more detail in \S\ref{subsec:scalar-invariant}.

\subsection{Future directions}
The results of this paper are mostly algebraic, not topological: we show how to reproduce a known invariant, the torsion, in terms of quantum groups.
However, we hope that future work in this direction could relate geometric invariants like the torsion with quantum invariants like the colored Jones polynomial, with potentially significant topological consequences.

Many conjectures in this direction (such as the volume conjecture) concern the asymptotic behvaior of quantum invariants as $m \to \infty$, where $m$ is the order of the root of unity.
It would be quite useful in this context to extend our results to other roots of unity than $i$.

\begin{conjecture}
  Let $\xi$ be a primitive $4m$th root of $1$.
  There is a Schur-Weyl duality akin to \cref{thm:schur-weyl} between $\U_\xi(\lie{sl}_2)$ and the $m$th twisted Lawrence-Krammer-Bigelow representation.
\end{conjecture}

The twisted Burau representation of \cref{thm:schur-weyl} comes from the braid action on the twisted homology of the punctured disc $D_n$.
It is generalized by the \emph{Lawrence-Krammer-Bigelow} representations \cite{Lawrence1990, Anghel2017}, which replace $D_n$ with the configuration space of $m$ points in $D_n$.
The case $m = 1$ recovers the Burau representation, and the $m=1$ case of our conjecture is \cref{thm:schur-weyl}.

\subsection{Torsions of links}
The untwisted Reidemeister torsion $\tau(L)$ of a link complement $S^3 \setminus L$ (which is essentially the Alexander polynomial of $L$) is defined using the representation $\rho : \pi_L \to \operatorname{GL}_1(\Q(t))$ sending each meridian $x$ of $\pi_L$ to $t$.
More generally one can send all meridians in component $i$ to a variable $t_i$, which gives the multivariate Alexander polynomial.

The torsion is defined using the $\rho$-twisted homology of $S^3 \setminus L$ and still makes sense for $\rho$ a representation into any matrix group $\operatorname{GL}_n(k)$ for $k$ a field.
(For the torsion to be nonzero $\rho$ needs to be sufficiently far from the trivial representation.)
When the image of $\rho$ is nonabelian, $\tau(L, \rho)$ is usually called the \emph{twisted} torsion.
We prefer to call the two cases \emph{abelian} and \emph{nonabelian} torsion, since a twisted chain complex occurs in both.
Recently there has been considerable interest in nonabelian torsions of links; one overview is \cite{Friedl2009}.

It is known \cite{Reshetikhin2019} that the abelian torsion can also be obtained from the quantum group $\mathcal U_q(\lie{gl}(1|1))$ and in an essentially equivalent way from a certain quotient of $\U_i(\lie{sl}_2)$ \cite{Murakami1993}; see \cite{Viro2002} for a comparision of these approaches.
Our work extends this construction to the case of nonabelian $\slgroup$ torsions.

\subsection{The Conway potential as a square root of the torsion}
\label{subsec:conway-potential}
We explain the interpretation of $\maininv$ as a nonabelian Conway potential.
The classical Reidemeister torsion $\tau(L)$ is only defined up to an overall power of $t$.
It is possible to refine the torsion to a rational function $\nabla(L, t^{1/2})$ of $t^{1/2}$, the \emph{Conway potential,} which is defined up to an overall sign.
In fact, for a knot $K$, the invariant $\nabla$ is always of the form 
\[
  \nabla(K, t^{1/2}) = \frac{\Delta_L(t)}{t^{1/2} - t^{-1/2}}
\]
where $\Delta_K(t)$ is the Alexander polynomial of $K$, normalized so that it is symmetric under $t \to t^{-1}$.
Up to the denominator (which arises naturally in the definition of~$\nabla$) we can think of $\nabla$ as a symmetrized version of $\Delta$.
A similar formula holds \cite[Corollary 19.6]{Turaev_2001} for links, with a slightly different denominator.

One way to construct the Conway potential $\nabla$ is as follows:
Instead of sending each meridian to $t$, consider the representation $\alpha$ into $\operatorname{SL}_2(\Q(t))$ sending the meridians to
\[
  \begin{pmatrix}
    t & 0\\
    0 & t^{-1}
  \end{pmatrix}
\]
Then the Reidemeister torsion $\tau(L, \alpha)$ is defined up to $\pm \det \alpha = \pm 1$.
Furthermore, (up to some constants depending on the whether $L$ is a knot) it always factors as a product
\[
  \tau(L, \alpha) = \nabla(L, {t}^{1/2}) \nabla(L, -{t}^{1/2}).
\]
Theorem \ref{thm:T-is-torsion} says that the invariant $\mainfunctor(L, \rho, {\bm \mu})$ is analogous for the nonabelian case, with the choice of square roots ${\bm \mu}$ generalizing the choice of square root ${t}^{1/2}$.

Another perspective \cite[\S 19]{Turaev_2001} on the Conway potential is that it is a sign-refined version of the torsion, because for an \emph{oriented} link $L$ the sign of $\nabla(L, {t}^{1/2})$ is fixed, unlike the sign of $\tau(L, \alpha)$.%
\footnote{Picking an orientation of $L$ gives an orientation on its meridians, so one can distinguish between $t$ and $t^{-1}$, hence fix the sign of $t^{1/2} - t^{-1/2}$. See also Remark \ref{remark:orientations}.}
Our extension $\mainfunctor(L, \rho, {\bm \mu})$ does not satisfy this property, since even with an orientation of $L$ it is only defined up to a fourth root of unity.
We expect that future work \cite{McPhailSnyderUnpub1} will allow us to choose a definite phase of $\mainfunctor$ and fix this deficiency.

\subsection{The quantum double and holonomy invariants}
Theorem \ref{thm:T-is-torsion} involves two related functors%
\footnote{Strictly speaking there is a scalar ambiguity in $\mainfunctor$, so the codomain should really be $\modcat / \left\langle i \right\rangle$. See \cref{prop:mainfunctor}.}
\begin{equation*}
  \mainfunctor : \hat \braid(\slgroup)^* \to \modcat \text{ and }
  \doubledfunctor : \hat \braid(\slgroup)^* \to \doubledcat.
\end{equation*}
We explain the notation and how to interpret $\doubledfunctor$ as the quantum double of $\mainfunctor$.

In the above $\braidfh$ is a variant of the $\slgroup$-colored braid groupoid $\braid(\slgroup)$ (see \S\ref{subsec:factorized-groups} and \cref{def:extended-chars}), while $\modcat$ is the category of weight modules for the quantum group $\U = \U_i(\lie{sl}_2)$.
Here a $\U$-weight module is one on which the center of $\U$ (in particular, the central subalgebra $\Ucentersmall$) acts diagonalizably.

In particular, for any simple $\U$-weight module $V$ the action of $\Ucentersmall$ is given by a character $\chi : \Ucentersmall \to \C$, that is a point of $\spec \Ucentersmall$.
Since $\spec \Ucentersmall$ is birationally equivalent to $\slgroup$, the category $\modcat$ of weight modules is $\slgroup$-graded.
(More accurately, it is $\slgroupdual$-graded, where $\slgroupdual \iso \spec \Ucentersmall$ is the Poisson dual group of $\slgroup$ in \cref{def:poisson-dual-sl2}.)

The new ingredient above is the category $\doubledcat$, the \emph{double} of $\modcat$.
Concretely, $\doubledcat$ is the category of $\U \otimes \U^{\cop}$-weight modules that are \emph{locally homogeneous}:
modules $W$ such that for any central $z \in \U$ and $w \in W$,
\[
  (z \otimes 1) \cdot w =  (1 \otimes S(z)) \cdot w,
\]
where $S$ is the antipode of $\U$.

Because the antipode defines the inverse of the algebraic group $\spec \Ucentersmall$, we can informally say that a locally homogeneous module is one that has degree $g$ for $\U \otimes 1$ and $g^{-1}$ for $1 \otimes \U^{\cop}$.
A typical locally homogeneous module is (a direct sum of) modules of the form $V \otimes_\C V^*$, where $V$ is a simple $\U$-module.
Later we will denote these modules by $V \boxtimes V^*$.

We think of $\doubledcat$ as the tensor product $\modcat \boxtimes \overline{\modcat}$, where $\overline{\modcat}$ is a ``mirror'' version of~$\modcat$ associated to $\U^{\cop}$.
This is a special case of the Deligne tensor product of categories, hence the notation~$\boxtimes$.
Similarly, we think of $\doubledfunctor = \mainfunctor \boxtimes \overline{\mainfunctor}$ as the tensor product of two group(oid) representations.

\subsection{Quantum doubles and the \texorpdfstring{$G$}{G}-center}
For the reader familiar with algebraic TQFT, the following discussion may help motivate the previous section.
The invariant $\mainfunctor$ of $\slgroup$-links in $S^3$ constructed from $\modcat$ is a \emph{surgery} or \emph{Reshetikhin-Turaev} invariant of link complements.
This theory is \emph{anomalous} because the representations involved are projective.%
\footnote{Usually, the theory for link complements is not anomalous; the anomaly instead appears for general manifolds resulting from surgery. In the holonomy case, the anomalies show up earlier, because $\U_i$ is no longer quasitriangular.}
However, in the doubled theory $\doubledcat$, the anomalies from $\modcat$ and $\overline{\modcat}$ cancel, and the corresponding invariant $\doubledfunctor$ is defined unambiguously.

One could think of the invariant from $\doubledcat$ as being the \emph{state-sum} or \emph{Turaev-Viro} invariant associated to $\modcat$.
For the non-graded case, it is well-known that the state-sum theory on a fusion category $\modcat$ agrees with the surgery theory on the Drinfeld center $\mathcal{Z}(\modcat)$.
For more details, see the book \cite{Turaev2016} by Turaev and the series of papers \cite{Balsam1,Balsam2,Balsam3} by Balsam and Kirilov Jr.
If $\modcat$ is modular (in particular, if it has a braiding) then there is an equivalence of categories $\modcat \boxtimes \overline{\modcat} \equiv \mathcal{Z}(\modcat)$ \cite{Mueger2003}, so we can compute the value of the state-sum theory from $\modcat$ by using the surgery theory from $\modcat \boxtimes \overline{\modcat}$.

In the $G$-graded case, \citeauthor{Turaev2019}  \cite{Turaev2019} define notions of state-sum and surgery \emph{homotopy quantum field theory} (a.k.a.~$G$-graded TQFT) and show that the state-sum theory from $\modcat$ is equivalent to the surgery theory from $\mathcal{Z}_G(\modcat)$, where $\mathcal{Z}_G(\modcat)$ is a graded version of the Drinfeld center of $\modcat$.
\begin{conjecture}
  As in the non-graded case, there is an equivalence%
  \footnote{Technically speaking $\modcat$ is $\slgroup^*$-graded, not $\slgroup$-graded (see \S\ref{subsec:weight-module-cat}) so $\slgroup^*$ here is correct.}
\[
  \mathcal{Z}_{\slgroup^*}(\modcat) \iso \doubledcat,
\]
so that we can interpret our surgery invariant from $\doubledcat$ as the state-sum invariant from $\modcat$.
\end{conjecture}
In the context of this conjecture it would be useful to directly relate our construction of $\doubledcat$ to the more abstract construction of the $G$-center $\mathcal{Z}_G(\modcat)$.
Objects of $\doubledcat$ are of the form $V \boxtimes V^*$ for $V$ an object of $\modcat$, while objects of $\mathcal{Z}_G(\modcat)$ are pairs $(V, \sigma_V)$ with $\sigma_V$ a half-braiding relative to the identity-graded component $\modcat_1$ of $\modcat$.

One difficulty in understanding this relationship is that the category $\modcat$ is not semisimple, and the non-semisimplicity  is concentrated in $\modcat_1$ (see Proposition \ref{prop:semisimplicity}).
We expect that an appropriate semisimplification of $\modcat_1$ will allow an application of the theory of \citeauthor{Turaev2019} to the construction of $\doubledcat$.
(A less serious issue is that the braid action on the gradings of $\modcat$ is not simply conjugation, as it is in~\cite{Turaev2019}.)

In the non-graded case, it is well-known that the Drinfeld center corresponds to the Drinfeld double, in the sense that there is an equivalence of braided categories
\[
  \mathcal{Z}(\operatorname{Rep}(H)) \iso \operatorname{Rep}(D(H))
\]
where $H$ is a (not necessarily quasitriangular) Hopf algebra.
\begin{conjecture}
There is a $G$-graded version $D_G$ of the Drinfeld double construction such that
\[
  \doubledcat \iso \operatorname{Rep}(D_{\slgroup^*}(\U_i(\lie{sl}_2))).
\]
\end{conjecture}
It is likely that this construction is related to the work of \citeauthor{Zunino2004a} \cite{Zunino2004a, Zunino2004b} on crossed quantum doubles.

\subsection{Overview of the paper}
\begin{description}
  \item[Section 2] We fix conventions on colored braids and discuss the factorization structure used to relate $\slgroup$ and $\slgroupdual$-colorings.
    We also introduce the algebra $\U = \U_i(\lie{sl}_2)$ and its relationship to colored braids.
  \item[Section 3] We define the twisted Burau representations and the twisted Reidemeister torsion.
  \item[Section 4] We state and prove \cref{thm:schur-weyl} and discuss how to use it to prove \cref{thm:T-is-torsion}.
  \item[Section 5] We summarize the construction \cite{Blanchet2018} of the BGPR invariant $\mainfunctor$ in our notation.
  \item[Section 6] We construct the quantum double $\doubledfunctor$ and discuss how it relates to $\mainfunctor$.
  \item[Section 7] We prove a version (\cref{thm:schur-weyl-double}) of \cref{thm:schur-weyl} for modules, which gives \cref{thm:T-is-torsion} as a corollary.
  \item[Appendix A] We give some results on $\U$-modules (in particular, on the projective cover of the trivial module) used in Section 7 and Appendix B.
  \item[Appendix B] We apply the of the modified traces of \citeauthor{Geer2018} \cite{Geer2018} to the category $\modcat$ of $\U$-weight modules and its quantum double $\doubledcat$.
  \item[Appendix C] We prove \cref{lemma:invariant-vector}, which is used in the definition of $\doubledfunctor$.
\end{description}

\section*{Acknowledgements}
I would like to thank Nicolai Reshetikhin for introducing me to holonomy invariants and suggesting a relationship to nonabelian torsions, and for many helpful discussions.
In addtion, I want to thank:
\begin{itemize}
  \item Noah Snyder for an enlightening conversation that lead me to the correct definition of the doubled representation $\doubledfunctor$,
  \item Christian Blanchet, Hoel Queffelec, and N.R.~for sharing some unpublished notes \cite{BQRUnpub} on holonomy $R$-matrices,
  \item Bertrand Patureau-Mirand for finding the right argument for (3) of Proposition \ref{prop:torsion-computation}, 
  \item Nathan Geer for his talk\footnote{At the conference \emph{New Developments in Quantum Topology} at UC Berkeley in June 2019.} where I learned the theory of \cref{appendix:modified-traces}, and finally
  \item the anonymous referees whose feedback substantially improved the organization of this article and clarified some technical points about twisted homology.
\end{itemize}

While none of the computations in this paper require computer verification, the computer algebra system SageMath and programming language Julia were very helpful in intermediate work, and I thank the developers and maintainers of this software for their work.

During the final preparation of this article, I was saddened to learn of the passing of John H.~Conway.
He made remarkable contributions to many areas of mathematics, and I particularly admire his work in knot theory on link potential functions and algebraic tangles.

\section{Representations of link complements and colored braids}
\label{sec:colored-braid-groupoid}
A holonomy invariant of links depends both on the link $L$ and a representation $\rho : \pi_1(S^3 \setminus L) \to \slgroup$, that is a point of the \emph{$\slgroup$-representation variety} 
\[
  \repvar{L} = \hom(\pi_1(S^3 \setminus L) , \slgroup).
\]
In this section, we describe some coordinate systems on $\repvar{L}$, emphasizing those coming from a presentation of $L$ as a braid closure.\footnote{For the more general case of tangle diagrams, see \cite{Blanchet2018}.}
By doing this, we reduce the problem to
\begin{enumerate}
  \item describing the representation variety $\repvar{D_n}$ of a punctured disc, then
  \item understanding the action of braids on our description.
\end{enumerate}
This perspective motivates us to define variants of the braid group we call \emph{colored braid groupoids}.

In particular, we describe a correspondence between  the $\slgroup$-representation variety $\repvar{D_n}$ of a punctured disc and certain central characters of $\U_i(\lie{sl}_2)$ due to \citeauthor{Kashaev2004} \cite{Kashaev2004}; the correspondence extends to the braid actions defined by topology and by the $R$-matrix.
Because of the structure of the quantum group $\U_i(\lie{sl}_2)$ it only gives coordinates on a large (in the sense of Zarsiki open and dense) subset of the representation variety.
We call representations lying in this open set \emph{admissible}, and every representation is conjugate to an admissible one (Proposition \ref{prop:admissible-exist}).

\begin{conventions}
  The braid group $\braid_n$ on $n$ strands has generators $\sigma_1, \dots, \sigma_{n-1}$, with $\sigma_i$ given by braiding strand $i$ over strand $i+1$.
  Braids are drawn and composed left-to-right.
  For example, Figure \ref{fig:braid-example} depicts the braid $\sigma_1 \sigma_2^{-1} \sigma_1$ on $3$ strands.
\end{conventions}
\begin{figure}
  \begin{center}
    \includegraphics{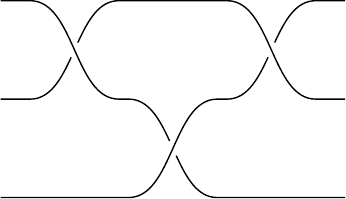}
  \end{center}
  \caption{The braid $\sigma_1 \sigma_2^{-1} \sigma_1$.}
  \label{fig:braid-example}
\end{figure}

\subsection{Colored braid groupoids}
Let $L$ be a link in $S^3$.
Given a diagram of $L$, we obtain the \emph{Wirtinger presentation} of the group $\pi_L = \pi_1(S^3 \setminus L)$.
(See Figure \ref{fig:wirtinger-example}.)
This presentation assigns one generator $x_i$ to each arc (unbroken curve) and one conjugation relation to each crossing.
\begin{figure}
  \begin{center}
    \includegraphics{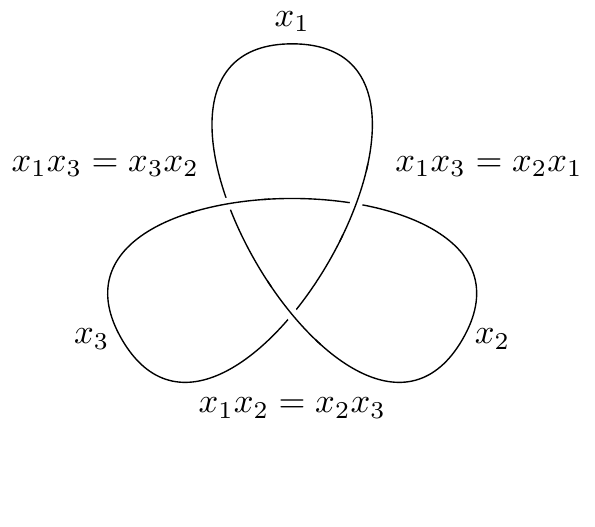}
    \vspace*{-0.5in}
  \end{center}
  \caption{Wirtinger generators of the trefoil group and the relation at each crossing.}
  \label{fig:wirtinger-example}
\end{figure}
If we represent $L$ as the closure of a braid $\beta$ on $n$ strands, we can examine the interaction between this presentation and the braid group.

Specifically, the Wirtinger presentation gives an action of the braid group $\braid_n$ on the free group $F_n$.
We can think of putting free generators $x_1, \dots x_n \in F_n $ on the $n$ strands on the left and acting on them by the braid to get words on the right.
Concretely, the generators act by
\begin{equation}
  \label{eq:braid-group-on-free-group}
  x_j \cdot \sigma_i =
  \begin{cases}
    x_i^{-1} x_{i+1} x_i & j = i, \\
    x_i & j = i+1, \\
    x_j & \text{ otherwise,}
  \end{cases}
\end{equation}
as in Figure \ref{fig:braid-action-free-group}.
Here the braid action on  the free group $F_n$ is written on the right, to match left-to-right composition of braids.
\begin{figure}
  \begin{center}
    \includegraphics{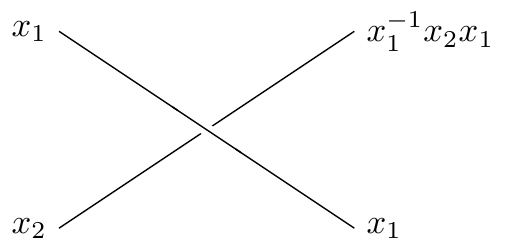}
  \end{center}
  \caption{Braid action on the free group.}
  \label{fig:braid-action-free-group}
\end{figure}

It follows that for any braid $\beta$ with closure $L$,
\[
  \pi_L = \left \langle x_1, \dots, x_n | x_i = x_i\cdot \beta\right\rangle
\]
gives a presentation of the fundamental group of $S^3 \setminus L$.
In particular,  a choice of representation $\rho : \pi_L \to G$ of the complement of the closure $L$ is equivalent to a choice of group elements $\rho(x_i)$ such that $\rho(x_i) = \rho( x_i \cdot \beta)$ for each $i$.

\begin{definition}
  A \emph{$G$-colored braid} is a braid $\beta$ on $n$ strands and a tuple $(g_1, \dots, g_n)$ of elements of $G$.%
  \footnote{More generally, one could define the groupoid of braids colored by any quandle or biquandle.}
  The \emph{$G$-colored braid groupoid} is the category $\braid(G)$ whose objects are tuples $(g_1, \dots, g_n)$ and whose morphisms are braids
  \[
    (g_1, \dots, g_n) \xrightarrow{\beta} (\rho(x_1 \cdot \beta), \dots, \rho(x_n \cdot \beta))
  \]
  where $\rho : F_n \to G$ is defined by $\rho(x_i) = g_i$.
  In particular, braid generators act by
  \[
    (g_1, g_2) \xrightarrow{\sigma_1} (g_1^{-1} g_2 g_1, g_1)
  \]
\end{definition}
One can think of the union $\braid \defeq \bigcup_n \braid_n$ of the braid groups as a category with objects $\{1, 2, \dots\}$, and links can be represented as closures of endomorphisms of $\braid$.
Similarly, links with a representation $\rho: \pi_L \to G$ can be represented as closures of endomorphisms of $\braid(G)$.
$\braid(G)$ is a monoidal category in the usual way: the product of objects is their concatenation, and the product of morphisms is obtained by placing them in parallel.
In our conventions, this monoidal product is vertical composition.

\begin{remark}
  \label{remark:orientations}
  The presentation of a $G$-link $(L, \rho)$ as the closure of a braid $\beta \in \braid(G)$ implicitly requires a choice of orientation.
  There are distinguished meridians $x_i$ around the base of the braid, but choosing between $\rho(x_i) = g_i$ and $\rho(x_i) = g_i^{-1}$ requires an orientation of the meridian $x_i$.

  The usual way to do this is to orient $L$ and use this to obtain an orientation of the meridian.
  For example, consider the result that the Conway potential is a sign-refined version of the Alexander polynomial defined for oriented links.

  We will usually leave this choice implicit going forward, but it will come up again when we discuss the mirrored invariants $\overline{\mainfunctor}$ in \S\ref{subsec:mirror-image}.
\end{remark}

\subsection{Factorized groups}
\label{subsec:factorized-groups}
To deal with the fact that the central subalgebra $\Ucentersmall$ of $\U_i(\lie{sl}_2)$ is not (the algebra of functions on) $\slgroup$ but on its Poisson dual group $\slgroupdual$ we need to use a slightly different description of $\slgroup$-links.

\begin{definition}
  A \emph{group factorization} is a triple $(G, \overline G, G^*)$ of groups, with $G$ a normal subgroup of $\overline G$, along with maps $\phi_+, \phi_- : G^* \to \overline G$ such that the map $\psi : G^* \to \overline G$
  \[
    \psi(a) = \phi^+(a) \phi^-(a)^{-1}
  \]
  restricts to a bijecton $G^* \to G$.
\end{definition}

\begin{example}
  Set $\overline{G} = \GL_2(\C)$,
  \begin{align*}
    S
    &\defeq
    \left\{
      \begin{pmatrix}
        \kappa & -\epsilon \\
        0 & \kappa^{-1}
      \end{pmatrix}
    \right\} \subseteq \GL_2(\C),
    \\
    S^*
    &\defeq
    \left\{
      \left(
      \begin{pmatrix}
        1 & 0 \\
        0 & \kappa
      \end{pmatrix}
      ,
      \begin{pmatrix}
        \kappa & \epsilon \\
        0 & 1
      \end{pmatrix}
      \right)
    \right\} \subseteq \GL_2(\C) \times \GL_2(\C),
  \end{align*}
  and let $\phi^+, \phi^- : G^* \to \overline{G}$ be the inclusions of the first and second factors, respectively.
  Then $(S, S^*, \overline{G})$ is a group factorization, and the map $\psi$ acts by
  \[
    \psi:
      \left(
      \begin{pmatrix}
        1 & 0 \\
        0 & \kappa
      \end{pmatrix}
      ,
      \begin{pmatrix}
        \kappa & \epsilon \\
        0 & 1
      \end{pmatrix}
      \right)
      \mapsto
      \begin{pmatrix}
        \kappa & -\epsilon \\
        0 & \kappa^{-1}
      \end{pmatrix}.
  \]
\end{example}
In general geometrically interesting representations into $\slgroup$ are irreducible, so their image does not lie in the subgroup $S \subseteq \slgroup$ above.
To include these representations we must consider a slightly more general notion:

\begin{definition}
  A \emph{generic group factorization} is a triple $(G, \overline G, G^*)$ of groups, with $G$ a normal subgroup of $\overline G$, along with maps $\phi_+, \phi_- : G^* \to \overline G$ such that the map $\psi : G^* \to \overline G$
  \[
    \psi(a) = \phi^+(a) \phi^-(a)^{-1}
  \]
  restricts to a bijection $G^* \to U$, where $U$ is a Zariski open dense subset of $G$.
\end{definition}
That is, instead of requiring $\psi$ to be a bijection, we simply require it to be a birational map.
Our motivating example is:
\begin{definition}
  \label{def:poisson-dual-sl2}
  The \emph{Poisson dual group}\footnote{$\slgroup$ is a Poisson-Lie group, so its Lie algebra $\lie{sl}_2$ is a Poisson-Lie bialgebra. %
  There is a dual Poisson-Lie bialgebra $\lie{sl}_2^*$, and its associated Lie group is $\slgroupdual$.}
    of $\slgroup$ is
  \[
    \slgroupdual \defeq \left\{ \left(
      \begin{pmatrix}
        \kappa & 0 \\
        \phi & 1
      \end{pmatrix}, %
      \begin{pmatrix}
        1 & \epsilon \\
        0 & \kappa
      \end{pmatrix} %
  \right) \right\} \subseteq \operatorname{GL}_2(\C) \times \operatorname{GL}_2(\C) .
  \]
  Set $G = \slgroup$ and $\overline G = \operatorname{GL}_2(\C) \times \operatorname{GL}_2(\C)$, and let $\phi^+, \phi^- : G^* \to \overline G $ be the inclusions of the first and second factors.
  Then the map $\psi$ acts by
  \begin{equation}
    \label{eq:sl2factorizaton}
    \psi :
    \left(
      \begin{pmatrix}
        \kappa & 0 \\
        \phi & 1
      \end{pmatrix}, %
      \begin{pmatrix}
        1 & \epsilon \\
        0 & \kappa
      \end{pmatrix} %
    \right)
    \mapsto
    \begin{pmatrix}
      \kappa & - \epsilon \\
      \phi & (1- \epsilon\phi)/\kappa
    \end{pmatrix}
  \end{equation}
  The image of $\psi$ is the set $U$ of matrices with $1,1$ entry nonzero, which is a Zariski open dense subset of $\slgroup$.
\end{definition}
We will show that every link admits a presentation with holonomies lying in $U$ so that we can use $\slgroupdual$ colorings instead of $\slgroup$ colorings and thus use the braiding on $\U_i(\lie{sl}_2)$.

We first describe how to use the group factorization to associate a tuple 
\[
  (g_1, \dots, g_n) \in \slgroup \times \cdots \times \slgroup
\] 
of $\slgroup$ elements to a tuple 
\[
  (a_1, \dots, a_n) \in \slgroupdual \times \cdots \times \slgroupdual
\]
of $\slgroupdual$ elements, in a way respecting the braiding action on the colors.
This is a special case of the \emph{biquandle factorization} defined in \cite{Blanchet2018}.

Write $a^\pm$ for $\phi^\pm(a)$.
Then $\psi$ extends to a map on tuples $\psi = (\psi_1, \cdots, \psi_n)$ with
\begin{equation}
  \label{eq:coordinate-change}
  \psi_i(a_1, \dots, a_n) = ( a_1^+ \cdots a_{i-1}^+ ) \psi(a_i) (a_1^+ \cdots a_{i-1}^+)^{-1}
\end{equation}
where $\psi(a_i) = a_i^+ (a_i^-)^{-1}$.
The formula is somewhat nicer in terms of the products $g_i \cdots g_1$: 
\[
  (\psi_i \cdots \psi_1)(a_1, \dots, a_n) = a_1^+ \cdots a_i^+ (a_1^- \cdots a_i^-)^{-1}, \quad i = 1, \dots, n
\]

\begin{figure}
  \centering
  \includegraphics{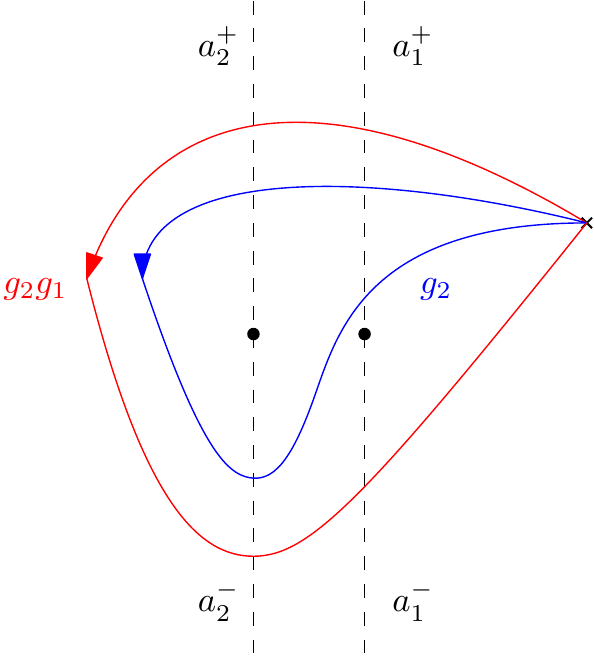}
  \caption{Derivation of $g_2 = a_1^+  a_2^+  (a_2^-)^{-1}  (a_1^+)^{-1}$ ({\color{blue} blue} path) and $g_2 g_1 = a_1^+ a_2^+ (a_2^-)^{-1} (a_1^-)^{-1}$ ({\color{red} red} path.)}
  \label{fig:holonomies}
\end{figure}

This is best-understood graphically.
For example, the blue path in Figure \ref{fig:holonomies} corresponds to the image $g_2$ of the generator $x_2$ of the fundamental group of the twice-punctured disc.
As it crosses the dashed line above the first point from left to right, it picks up a factor of $a_1^+$, then $a_2^+$ for the next dashed line.
When crossing the line below, we get a factor for $(a_2^-)^{-1}$ because we are crossing right to left, and similarly for $(a_1^+)^{-1}$.
We have derived the relation
\[
  g_2 = \psi_2(a_1, a_2) = a_1^+  a_2^+  (a_2^-)^{-1}  (a_1^+)^{-1}.
\]

We can think of the $a_i$ as local coordinates and the $g_i$ as global coordinates.
As an explicit example, if
\[
  a_i = \left(
      \begin{pmatrix}
        \kappa_i & 0 \\
        \phi_i & 1
      \end{pmatrix}, %
      \begin{pmatrix}
        1 & \epsilon_i \\
        0 & \kappa_i
      \end{pmatrix} %
  \right)
\]
for $i = 1, 2$, then the expressions for the images
\begin{align*}
  g_1 &= \psi_1(a_1, a_2) = a_1^+ (a_1^-)^{-1} \\
  g_2 &= \psi_2(a_1, a_2) = a_1^+  a_2^+  (a_2^-)^{-1}  (a_1^+)^{-1}
\end{align*}
of the Wirtinger generators are somewhat complicated, while the expressions for their products
\begin{align*}
  g_1 &= a_1^+ (a_1^-)^{-1} = 
  \begin{pmatrix}
    \kappa_{1} & -\epsilon_{1} \\
    \phi_{1} & \frac{ 1- \epsilon_{1} \phi_{1}}{\kappa_{1}} 
  \end{pmatrix} \\
  g_2g_1 &= a_1^+ a_2^+ (a_1^- a_2^-)^{-1} = 
  \begin{pmatrix}
    \kappa_{1} \kappa_{2} & -\epsilon_{1} \kappa_{2} - \epsilon_{2} \\
    \kappa_{2} \phi_{1} + \phi_{2} & \frac{{ 1 - \left(\epsilon_{1} \kappa_{2} + \epsilon_{2}\right)} {\left(\kappa_{2} \phi_{1} + \phi_{2}\right)}}{\kappa_{1} \kappa_{2}}
  \end{pmatrix}
\end{align*}
are simpler.

If $\sigma : (a_1, a_2) \to (a_4, a_3)$ is a generator, the image colors are the unique solutions to the equations 
\begin{equation}
  \label{eq:biquandlequations}
  a_1^+ a_2^+ = a_4^+ a_3^+, \quad a_1^- a_2^- = a_4^- a_3^-, \quad a_1^- a_2^+ = a_4^+ a_3^-
\end{equation}
which we can  read off by thinking about paths above, below, and between the strands.
For example, $a_1^- a_2^+ = a_4^+ a_3^-$ follows from comparing the red (left) and blue (right) paths in Figure \ref{fig:biquandle-relations}.
\begin{figure}
  \includegraphics{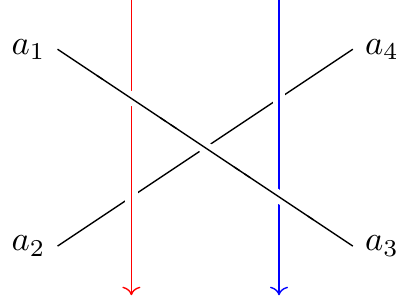}
  \caption{Derivation of the biquandle relation $a_1^- a_2^+ = a_4^+ a_3^-$ in (\ref{eq:biquandlequations}).}
  \label{fig:biquandle-relations}
\end{figure}

\begin{definition}
  \label{def:sl2star-biquandle}
  Let $a_1, a_2 \in \slgroupdual$.
  When they exist, let $a_4, a_3$ be the unique solutions of (\ref{eq:biquandlequations}) and set $B(a_1, a_2) = (a_4, a_3)$.
  We say that $B$ and $\slgroupdual$ form a \emph{generic biquandle}.
\end{definition}
For a general definition, see \cite[\S3]{Blanchet2018}.
It is possible for the equations (\ref{eq:biquandlequations}) to not have a solution, so this is only a partially defined or \emph{generic} biquandle.
The general theory of this is dealt with in \cite[\S5]{Blanchet2018}.
We will simply restrict to colorings for which the map $B$ is defined.

\begin{definition}
  $\braidf$ is the category whose objects are tuples $(a_1, \dots, a_n)$ of elements of $\slgroupdual$ and whose morphisms are admissible colored braids between them, with the action on colors given by the map $B$.
  A braid generator $\sigma : (a_1, a_2) \to B(a_1, a_2)$ is \emph{admissible} if $B(a_1, a_2)$ is defined (i.e.~if the equations (\ref{eq:biquandlequations}) have a solution) and a colored braid is admissible if it can be expressed as a product of admissible generators.

  We refer to morphisms of $\braidf$ as \emph{$\slgroupdual$-colored braids}.
  $\braidf$ becomes a monoidal category in the usual way, with the product of objects given by concatenation and the product of braids given by vertical stacking.
\end{definition}

\begin{proposition}
  \label{prop:factorization-functor}
  The map (\ref{eq:coordinate-change}) extends to a functor $\Psi : \braid(\slgroup)^* \to \braid(\slgroup)$.
\end{proposition}
\begin{proof}
  This is a special case of \cite[Theorem 3.9]{Blanchet2018}. 
\end{proof}

\begin{definition}
  \label{def:admissible}
  We say that objects and morphisms of $\braid(\slgroup)$ are \emph{admissible} when they lie in the image of $\Psi$. 
  More concretely, let $(g_1, \dots, g_n)$ be an object of the $\slgroup$-colored braid groupoid, that is a tuple of elements of $\slgroup$.
  We say it is \emph{admissible} if for $i = 1, \dots, n$ the element
  \[
    g_i \cdots g_1
  \]
  has a nonzero $1,1$-entry, so that the factorization map is well-defined.

  Similarly, a braid generator $\sigma_i: (g_1, \cdots, g_n) \to (g_1, \cdots, g_n) $ is admissible if its source and target are admissible, and a $\slgroup$-colored braid $\beta$ is admissible if it can be expressed as a product of admissible generators.
\end{definition}

The functor $\Psi$ is not an equivalence of categories because it is not onto, but it can be shown to be a generic equivalence, in a sense made precise in \cite[\S5]{Blanchet2018}.
In particular:
\begin{proposition}
  \label{prop:admissible-exist}
  Every $\slgroup$-link $L$ is gauge-equivalent to one admitting a presentation as the closure of an admissible $\slgroup$-colored braid $\beta$, hence as the closure of a $\slgroupdual$-colored braid.
\end{proposition}
  Here two $\slgroup$-links $(L, \rho)$ and $(L, \rho')$ are \emph{gauge-equivalent} if there exists an element $g \in \slgroup$ such that $\rho' = g \rho g^{-1}$.
\begin{proof}
  $L$ is clearly the closure of some $\slgroup$-colored braid $\beta : (g_1, \dots, g_n) \to (g_1, \dots, g_n)$.
  (It is closure of a braid $\beta$, and the representation $\rho$ makes $\beta$ a colored braid.)
  If $(g_1, \dots, g_n)$ is not admissible, we can conjugate $\rho$ to obtain an admissible object.
  Now by \cite[Theorem 5.5]{Blanchet2018} $\beta$ can be written as an admissible product of generators, hence is admissible.
\end{proof}
Later, we will construct functors of the form $\mathcal{F} : \braidf \to \modcat$, for $\modcat$ a pivotal category.
This means that we can take traces of endomorphisms of $\modcat$, and if the traces are appropriately gauge-invariant, we can use $\mathcal{F}$ to obtain invariants of (framed) $\slgroup$-links by the following process:
\begin{enumerate}
  \item Gauge transform the $\slgroup$-link $(L, \rho)$ to a link $(L, \rho')$ that is the closure of an admissible $\slgroup$-braid $\beta_0$.
  \item Because $\beta_0$ is admissible, it can be pulled back along $\Psi$ to a $\slgroupdual$-braid $\beta$.
  \item Take the trace of the image of $\beta$ under $\mathcal{F}$ to obtain an invariant of $L$:
    \[
      \mathcal{F}(L, \rho) \defeq \tr \mathcal{F}(\beta).
    \]
\end{enumerate}
For example, the BGPR invariant of \cref{thm:BGPR-inv} and \cite{Blanchet2018} is of this type.
As is usual for the RT construction, the invariants can depend on the framing of our link $L$, because the colored Reidemeister I move (\cref{fig:RI-move}) may not hold.

\begin{figure}
  \centering
\begingroup%
  \makeatletter%
  \providecommand\color[2][]{%
    \errmessage{(Inkscape) Color is used for the text in Inkscape, but the package 'color.sty' is not loaded}%
    \renewcommand\color[2][]{}%
  }%
  \providecommand\transparent[1]{%
    \errmessage{(Inkscape) Transparency is used (non-zero) for the text in Inkscape, but the package 'transparent.sty' is not loaded}%
    \renewcommand\transparent[1]{}%
  }%
  \providecommand\rotatebox[2]{#2}%
  \newcommand*\fsize{\dimexpr\f@size pt\relax}%
  \newcommand*\lineheight[1]{\fontsize{\fsize}{#1\fsize}\selectfont}%
  \ifx\svgwidth\undefined%
    \setlength{\unitlength}{245.65132713bp}%
    \ifx\svgscale\undefined%
      \relax%
    \else%
      \setlength{\unitlength}{\unitlength * \real{\svgscale}}%
    \fi%
  \else%
    \setlength{\unitlength}{\svgwidth}%
  \fi%
  \global\let\svgwidth\undefined%
  \global\let\svgscale\undefined%
  \makeatother%
  \begin{picture}(1,0.1960381)%
    \lineheight{1}%
    \setlength\tabcolsep{0pt}%
    \put(0,0){\includegraphics[width=\unitlength,page=1]{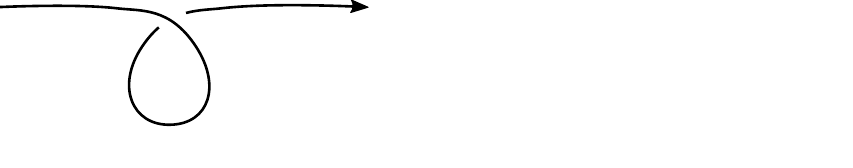}}%
    \put(0.0153366,0.14197889){\makebox(0,0)[lt]{\lineheight{1.25}\smash{\begin{tabular}[t]{l}$a$\end{tabular}}}}%
    \put(0.38715677,0.14406898){\makebox(0,0)[lt]{\lineheight{1.25}\smash{\begin{tabular}[t]{l}$a$\end{tabular}}}}%
    \put(0.15495746,0.01016111){\makebox(0,0)[lt]{\lineheight{1.25}\smash{\begin{tabular}[t]{l}$\alpha(a)$\end{tabular}}}}%
    \put(0,0){\includegraphics[width=\unitlength,page=2]{RI.pdf}}%
    \put(0.7956467,0.14299344){\makebox(0,0)[lt]{\lineheight{1.25}\smash{\begin{tabular}[t]{l}$a$\end{tabular}}}}%
    \put(0,0){\includegraphics[width=\unitlength,page=3]{RI.pdf}}%
  \end{picture}%
\endgroup%

  \caption{The colored Reidemesiter I move. The map $\alpha$ is as in \cite[Theorem 6.8]{Blanchet2018}.}%
  \label{fig:RI-move}
\end{figure}

\subsection{Quantum \texorpdfstring{$\lie{sl}_2$}{sl\_2} at a fourth root of unity}
The motivation for defining $\braidf$ is its relationship to central characters of $\U_q(\lie{sl}_2)$ for $q$ a root of unity, which we now describe.
\begin{definition}
Quantum $\lie{sl}_2$ is the algebra $\U_q = \U_q(\lie{sl}_2)$ over $\C[q, q^{-1}]$ with generators $E,F,K^{\pm1}$ and relations
\begin{align*}
  KE &= q^2 EK \\
  KF &= q^{-2} FK \\
  [E,F] &= (q - q^{-1}) (K - K^{-1})
\end{align*}
We sometimes use the generator $\F = qKF$ instead of $F$.
\end{definition}
Notice that our conventions are slightly nonstandard (in particular, they differ from \cite{Blanchet2018}.)
We want to view $\U_q$ as a deformation of the algebra of functions on $\slgroup^*$, not a deformation of the universal enveloping algebra of $\lie{sl}_2$.
For this reason, we choose $[E,F]$ as above instead of the more common $[E,F] = (K- K^{-1})/(q - q^{-1})$.

$\U_q$ is a Hopf algebra, with coproduct
\[
\Delta E = 1 \otimes E + E \otimes K, \quad \Delta F = K^{-1} \otimes F + F \otimes 1, \quad \Delta K = K \otimes K
\]
and antipode
\[
S(E) = - EK^{-1}, \quad S(F) = -KF, \quad S(K) = K^{-1}.
\]
The center of $\U$ is generated by the Casimir element
\[
  \tilde \Omega = EF + q^{-1} K + qK^{-1} .
\]
We will mostly work with the normalization $\Omega = q \tilde \Omega$.

We consider the case where $q$ is specialized to a primitive fourth root of unity $i$, which is $\ell = 4, r  = 2$ in \cite{Blanchet2018}.
The relations for $\U = \U_i(\lie{sl}_2)$ are then
\begin{align*}
  KE &= -EK, \\
  KF &= -FK, \\
  [E,F] &= 2i (K - K^{-1}).
\end{align*}
Specializing to a root of unity causes $\U$ to have a large central subalgebra
\[
  \Ucentersmall \defeq \C[K^2, K^{-2}, E^2, F^2].
\]
The center $\Ucenter$ of $\U$ is generated by $\Ucentersmall$ and the Casimir $\Omega$, subject to the relation
\[
  \Omega^2 = (K - K^{-1})^2 - E^2 F^2.
\]
We can identify the closed points of $\spec \Ucentersmall$ with the set of \emph{characters}, that is algebra homomorphisms $\chi : \Ucentersmall \to \C$.
The characters form a group with multiplication $\chi_1 \chi_2(x) \defeq (\chi_1\otimes\chi_2)(\Delta(x))$.
In fact, this group is $\slgroupdual$:
\begin{proposition}
  \label{prop:character-group}
Let $\chi$ be a $\Ucentersmall$-character and set
\begin{align*}
  \chi(E^2) = \epsilon, \quad \chi(F^2) = \phi/\kappa, \quad \chi(K^2) = \kappa.
\end{align*}
The map sending $\chi$ to the group element
\[
  \left(
    \begin{pmatrix}
      \kappa & 0\\
      \phi & 1
    \end{pmatrix},
    \begin{pmatrix}
      1 & \epsilon \\
      0 & \kappa
    \end{pmatrix}
  \right) \in \slgroupdual
\]
is an isomorphism of algebraic groups $\spec \Ucentersmall \to \slgroupdual$.
The inverse $\chi^{-1}$ of a character is the character $\chi S$ obtained by precomposition with the antipode.
\end{proposition}

From now on we identify $\Ucentersmall$-characters and the corresponding points of $\slgroupdual$.
The image of a character $\chi$ is the factorization of the matrix
\[
  \psi(\chi) \defeq
  \begin{pmatrix}
    \kappa & - \epsilon \\
    \phi & \kappa^{-1}(1- \epsilon \phi) \\
  \end{pmatrix} \in \slgroup
\]
so this identification is compatible with the factorization of $\slgroup$ in terms of $\slgroupdual$.
Here we have intentionally used the same symbol $\psi$ for the defactorization maps $\slgroupdual \to \slgroup$ and $\spec \Ucentersmall \to \slgroup$.

\begin{remark}
  \label{rem:casimirs}
  Under this correspondence, the Casimir $\Omega$ corresponds to the trace of the matrix.
  Specifically, we have
  \[
    \chi(\Omega^2)
    = \kappa + \kappa^{-1} - \epsilon \phi - 2
    = \tr \psi(\chi) - 2
    =
    \tr
    \begin{pmatrix}
      \kappa & - \epsilon \\
      \phi & \kappa^{-1}(1- \epsilon \phi) \\
    \end{pmatrix} 
    -2.
  \]
  Equivalently, if the eigenvalues of $\psi(\chi)$ are  $\mu^2$ and $\mu^{-2}$,
  \[
    \chi(\Omega^2) = \mu^2 - 2 + \mu^{-2} = (\mu - \mu^{-1})^2.
  \]
  In particular, if $\psi(\chi)$ does not have $1$ as an eigenvalue, $\chi(\Omega)^2 \ne 0$.
\end{remark}

\subsection{The braiding for \texorpdfstring{$\U$}{U}}
\label{subsec:holonomy-braiding}
Unlike $\U_q$, the algebra $\U = \U_i$ is not quasitriangular.
Instead, there is an outer automorphism
\[
  \rmat : \U \otimes \U \to \U \otimes \U[(1 + K^{-2} E^2 \otimes F^2)^{-1}]
\]
that satisfies the Yang-Baxter equations
\begin{align*}
  (\Delta \otimes 1) \rmat(u \otimes v) &= \rmat_{13} \rmat_{23} (\Delta(u) \otimes v) \\
  (1 \otimes \Delta) \rmat(u \otimes v) &= \rmat_{13} \rmat_{12} (u \otimes \Delta(v))
\end{align*}
and
\begin{align*}
  (\epsilon \otimes 1) \rmat(u \otimes v) &= \epsilon(u) v \\
  (1 \otimes \epsilon) \rmat(u \otimes v) &= \epsilon(v) u 
\end{align*}
where $\Delta$ is the coproduct, $\epsilon$ the counit, and $\rmat_{ij}$ means the action on the $i$th and $j$th tensor factors.

In a quasitriangular Hopf algebra, $\rmat$ comes from conjugation by an element called the \emph{$R$-matrix}.
This is not the case for $\U = \U_i$, but there is a version of $\U_q$ defined over formal power series in $h$ (with $q = e^h $) which has an $R$-matrix.
The conjugation action of this element is still well-defined in the specialization $q = i$, giving the outer automorphism $\rmat$.
For more details, see the paper \cite{Kashaev2004}.

It is well known that $\rmat$ gives a braid group action on tensor powers\footnote{Technically we need to take localizations at $1 + K^{-2} E^2 \otimes F^2$, so it is somewhat awkward to state this formally with more than two tensor factors. Instead we prefer to work with quotients of $\U$ at appropriate $\Ucentersmall$-characters.} of $\U$.
We think of a braid generator $\sigma$ as corresponding to the map $\algbraid \defeq \tau \rmat$.
The Yang-Baxter equations correspond to the braid relation
\[
  \algbraid_{12} \algbraid_{23} \algbraid_{12} = \algbraid_{23} \algbraid_{12} \algbraid_{23}
\]
which is the Hopf algebra version of the braid relation  $\sigma_1 \sigma_2 \sigma_1 = \sigma_2 \sigma_1 \sigma_2$.
From now only we mostly work with $\algbraid$.

\begin{lemma}
  \label{lem:central-action}
  Set
  \[
    W \defeq  1 + K^2 F^2 \otimes K^{-2} E^2 \in \Ucentersmall \otimes \Ucentersmall.
  \]
  $\algbraid$ is the unique automorphism of $\U \otimes \U[W^{-1}]$ satisfying
  \begin{align*}
    \algbraid(E \otimes 1) &= K \otimes E \\
    \algbraid(1 \otimes F) &= F \otimes K^{-1} \\
    \algbraid(K \otimes 1) &= 1 \otimes K - i KF \otimes E \\
    \intertext{and}
    \algbraid(\Delta(u)) &= \Delta(u)
  \end{align*}
  for every $u \in \U$.
  The action of $\algbraid$ on $\Ucentersmall \otimes \Ucentersmall[W^{-1}]$ is given by
  \begin{align*}
    \algbraid(K^2 \otimes 1) &=  ( 1 \otimes K^2 ) W
                             &
    \algbraid(1 \otimes K^2) &= (K^2 \otimes 1) W^{-1}
    \\
    \algbraid(E^2 \otimes 1) &= K^2 \otimes E^2
                             &
    \algbraid (1 \otimes E^2) &= E^2 \otimes K^2 + E^2 \otimes (1 - K^4 W^{-1})
    \\
    \algbraid(1 \otimes F^2) &= F^2 \otimes K^{-2}
                             &
    \algbraid(F^2 \otimes 1) &= K^{-2} \otimes F^2 + (1 - K^{-4} W^{-1}) \otimes F^2
  \end{align*}
\end{lemma}
\begin{proof}
  See \cite[\S2.2 and Appendix B]{Kashaev2004} and \cite[\S6.2]{Blanchet2018}.
\end{proof}
The action of $\algbraid$ on the central subalgebra $\Ucentersmall$ corresponds to the biquandle (that is, the braiding) on $\slgroup^* = \spec \Ucentersmall$ discussed in \S\ref{subsec:factorized-groups}.
Similarly, the localization at $W$ corresponds to the fact that the $\slgroupdual$ biquandle is only partially defined.
To compare them, we give the action of this biquandle in coordinates:

\begin{lemma}
  \label{lemma:biquandle-coords}
  Let $a_1, a_2, a_3, a_4$ be elements of $\slgroupdual$ related by the braiding
  \[
    \sigma : (a_1, a_2) \to (a_4, a_3)
  \]
  and with components
  \[
    a_i =
    \left(
      \begin{pmatrix}
        \kappa_i & 0 \\
        \phi_i & 1
      \end{pmatrix},
      \begin{pmatrix}
        1 & \epsilon_i \\
        0 & \kappa_i
      \end{pmatrix}
    \right).
  \]
  Then the components of $a_4$ and $a_3$ are given by
  \begin{align*}
    \kappa_4
    &=
    \epsilon_{1} \phi_{2} + \kappa_{2}
    &
    \phi_4
    &=
    \kappa_{1} \phi_{2}
    \\
    \kappa_3
    &=
    \frac{\kappa_{1} \kappa_{2}}{\epsilon_{1} \phi_{2} + \kappa_{2}}
    &
    \epsilon_3
    &=
    \frac{\epsilon_{1}}{\epsilon_{1} \phi_{2} + \kappa_{2}}
  \end{align*}
  and
  \begin{align*}
    \epsilon_4
    &=
    \frac{\epsilon_{1} \kappa_{2}^{2} + \epsilon_{2} \kappa_{2} + {\left(\epsilon_{1}^{2} \kappa_{2} + \epsilon_{1} \epsilon_{2}\right)} \phi_{2} - \epsilon_{1}}{\kappa_{1} \kappa_{2}}
    \\
    \phi_3
    &=
    \frac{\kappa_{2}^{2} \phi_{1} + \epsilon_{1} \phi_{2}^{2} + {\left(\epsilon_{1} \kappa_{2} \phi_{1} - {\left(\kappa_{1}^{2} - 1\right)} \kappa_{2}\right)} \phi_{2}}{\epsilon_{1} \phi_{2} + \kappa_{2}}
  \end{align*}
  
\end{lemma}
\begin{proof}
  These follow from writing out the the biquandle relations \eqref{eq:biquandlequations}.
\end{proof}

\begin{proposition}
  \label{prop:holonomy-braiding}
  Consider a $\slgroupdual$-colored braid generator 
  \[
    \sigma : (\chi_1, \chi_2) \to (\chi_4, \chi_3),
  \]
  thinking of $\slgroupdual$ elements as $\Ucentersmall$-characters.
  Then $\algbraid$ is compatible with the $\slgroupdual$ biquandle in the sense that
  \[
    (\chi_4 \otimes \chi_3) \algbraid = \chi_1 \otimes \chi_2.
  \]
  In particular, $\algbraid$ descends to a homomorphism of algebras
  \[
    \algbraid : \adjmod{\chi_1} \otimes \adjmod{\chi_2} \to \adjmod{\chi_4} \otimes \adjmod{\chi_3},
  \]
  where by $\adjmod{\chi}$ we mean the quotient of $\U$ by the ideal generated by $\ker \chi$.
\end{proposition}
\begin{proof}
  Under the correspondence of \cref{prop:character-group} we have $\chi_i(K^2) = \kappa_i, \chi_i(E^2) = \epsilon_i$, and $\chi_i(F^2) = \phi_i/\kappa_i$.
  We have claimed that
  \[
    (\chi_1 \otimes \chi_2)(K^2 \otimes 1) = \kappa_1 
  \]
  is equal to 
  \[
    (\chi_4 \otimes \chi_3)( (1 \otimes K^2) W) = \kappa_3 + \phi_4 \epsilon_3.
  \]
  By using the relations of \cref{lemma:biquandle-coords}, we see that
  \begin{align*}
    \kappa_3 + \phi_4 \epsilon_3
    =
    \frac{\kappa_1 \kappa_2}{ \epsilon_1 \phi_2 + \kappa_2} + \frac{\kappa_1 \phi_2 \epsilon_1}{ \epsilon_1 \phi_2 + \kappa_2}
    =
    \kappa_1.
  \end{align*}
  One can check similar relations for the other generators of $\Ucentersmall \otimes \Ucentersmall$.
  In this context, it is slightly more natural to use $\algbraid^{-1}$ and the equivalent relation 
  \[
    \chi_4 \otimes \chi_3 = (\chi_1 \otimes \chi_2)\algbraid^{-1}.
  \]
  The formal inversion of $W$ is not an issue, because
  \[
    (\chi_4 \otimes \chi_3)(W) = 1 + \frac{\phi_4 \epsilon_3}{\kappa_3} = 1 + \frac{\phi_1 \epsilon_2}{\kappa_2}
  \]
  is nonzero exactly when the $\slgroupdual$ colors $(\chi_1, \chi_2)$ are admissible.
\end{proof}
\begin{definition}
  \label{def:alg-functor}
  Consider the category $\mathsf{Alg}_\C$ of algebras over $\C$ and homomorphisms between them. 
  We define a functor $\algfunctor : \braid(\slgroup)^* \to \mathsf{Alg}_\C$ as follows:
  \begin{itemize}
    \item 
      For an object $(\chi_1, \dots, \chi_n)$ of $\braid(\slgroup)^*$, that is an $n$-tuple of $\Ucentersmall$-characters, set
      \[
        \algfunctor(\chi_1, \dots, \chi_n)
        =
        \adjmod{\chi_1} \otimes \cdots \otimes \adjmod{\chi_n}
        = \U^{\otimes n}/\ker(\chi_1 \cdots \chi_n)
      \]
    \item
      For a braid generator $\sigma_i : (\chi_1, \dots, \chi_n) \to (\chi_1, \dots, \chi_n')$, set
      \[
        \algfunctor(\sigma_i) = \algbraid_{i, i+1} :
        \U^{\otimes n}/\ker(\chi_1 \otimes \cdots \chi_n)
        \to
        \U^{\otimes n}/\ker(\chi_1' \cdots \chi_n')
      \]
      where  $\algfunctor_{i, i+1}$ acts on tensor factors $i$ and $i+1$, and where the image characters satisfy
      \[
        (\chi_i' \otimes \chi_{i+1}')\algbraid = \chi_i \otimes \chi_{i+1}
      \]
      and $\chi_j = \chi_j'$ for $j \ne i, i+1$.
  \end{itemize}
\end{definition}

Let $(L, \rho)$ be an $\slgroup$-link.
By representing $L$ as the closure of a braid, we identify $L$ with an endomorphism $\beta$ of $\braid(\slgroup)$.
If $\beta$ is admissible (and we can always gauge-transform so this is the case) then we can pull it back along $\Psi$ to an endomorphism of $\braidf$.
Finally, we can use $\algfunctor$ to associate $\beta$ with an automorphism of $\U^{\otimes n}/\ker(\chi_1 \otimes \cdots \chi_n)$.

We think of this construction as a (generic) representation of $\braid(\slgroup)$ in the algebra $\U$.
In the next two sections, we construct the Burau representations of $\braid(\slgroup)$ and show that they are dual to this representation.

\section{The Burau representation and torsions}
Folowing A.~Conway \cite{Conway2015}, we define the twisted reduced Burau representation of the colored braid groupoid, which is obtained via the $\rho$-twisted homology 
\[
\homol{D_n, \partial D_n; \rho}
\]
of the punctured disc relative to the boundary.
To match the braid action on the quantum group $\U$, we need to take the dual.
We achieve this by using the locally-finite homology 
\[
  \homol{D_n; \rho}[\lf]
\]
and we call the resulting representation the \emph{twisted reduced Burau representation}, or simply the \emph{Burau representation} $\burau$.
We then explain how to use $\burau$ to compute the torsion of a $\slgroup$-link $(L, \rho)$.

\subsection{Twisted homology and the Burau representation}
\label{subsec:burau-rep}
We have chosen to describe a link $L$ in $S^3$ by representing it as the closure of a braid $\beta$.
By doing this, we place $L$ inside a solid torus $T$.
We can slice $T$ open across a meridional disc $D_n$, which we think of as having $n$ punctures corresponding to the strands of $\beta$.

From this perspective, we can view the algebraic category $\braid(G)$ of the previous section as a model for a topological category $\mathsf{Map}(D_n, G)$.
This category has objects pairs $(D_n, \rho)$, where $D_n$ is an $n$-punctured disc and $\rho$ is a representation  $\pi_1(D_n) \to G$.
The morphisms are 
\[
  f : (D_n, f^*\rho) \to (D_n, \rho)
\]
for $f$ an element of the mapping class group of $D_n$, where $f^* \rho = \rho \circ f$ is the pullback.
As $\pi_1(D_n) = F_n$ is a free group, representations $\pi_1(D_n) \to G$ are $n$-tuples of elements of $G$, and since the mapping class group of $D_n$ is $\braid_n$, it is not hard to see that $\mathsf{Map}(D_n, G)$ is equivalent to $\braid(G)$.

The point of this topological description is that we obtain a colored braid action on the twisted homology of $D_n$.
We recall the definition below.

Let $X$ be a finite CW complex with fundamental group $\pi = \pi_1(X)$, and let $\rho : \pi \to \operatorname{GL}(V)$ be a representation, where $V$ is a vector space over $\C$.\footnote{More generally this works for a module over any commutative ring; this perspective is important when defining the twisted Alexander polynomial.}
We think of this as a right representation acting on row vectors, so that $V$ is a right $\Z[\pi]$-module.

Let $\tilde X$ be the universal cover of $X$.
The group $\pi = \pi_1(X)$ acts on the cells of the universal cover, and this action commutes with the differentials.
We take this to be a left action, so that the cellular chain complex $C_*(\tilde X)$ of the universal cover becomes a complex of left $\Z[\pi]$-modules.
\begin{definition}
  The \emph{$\rho$-twisted homology} $\homol[*]{X;\rho}$ of $X$ is the homology of the \emph{$\rho$-twisted chain complex}
  \[
    C_*(X; \rho) \defeq V \otimes_{\Z[\pi]} C_*(\tilde X).
  \]
\end{definition}
We have given this definition in terms of a CW complex for $X$ and a choice of lifts, but it can be shown to not depend on the choice of lifts.
In fact, the $\rho$-twisted homology also does not depend on the CW structure.
One way to see this is to give a definition in terms of $\operatorname{GL}(V)$-local systems.

The twisted Burau representation is given by the action of braids on the homology groups.
Because a braid $\beta$ acts nontrivially on the representations, it should be understood as a groupoid representation.
\begin{definition}
  The \emph{twisted Burau representation} is the functor $\braid(\operatorname{GL}(V)) \to \vect \C$ sending an object $\rho$ to the vector space $\homol{D_n;\rho}$ and a colored braid $\beta : \rho \to \rho'$ to the linear map
  \[
    \widetilde{\burau} : \homol{D_n;\rho} \to \homol{D_n;\rho'}
  \]
  corresponding to the action of $\beta$ on $D_n$.
  Here $\vect \C$ is the category of $\C$-vector spaces and linear maps.

  Any braid $\beta$ fixes the boundary of $D_n$, so we can define the \emph{boundary-reduced twisted Burau representation} as the action on homology relative to the boundary:
  \[
    \burau^{\partial}(\beta) : \homol{D_n, \partial D_n; \rho} \to \homol{D_n, \partial D_n; \rho'}.
  \]
  
\end{definition}
When passing to the reduced representation it is helpful to use a different presentation of $\pi_1(D_n)$.
Set $y_i = x_i x_{i-1} \cdots x_1$, so that the braid group on the generators is
\[
  y_j \cdot \sigma_i =
  \begin{cases}
    y_{i-1} y_{i}^{-1} y_{i+1} & i = j \\
    y_j & i \ne j
  \end{cases}
\]
\begin{figure}
  \centering
  \includegraphics{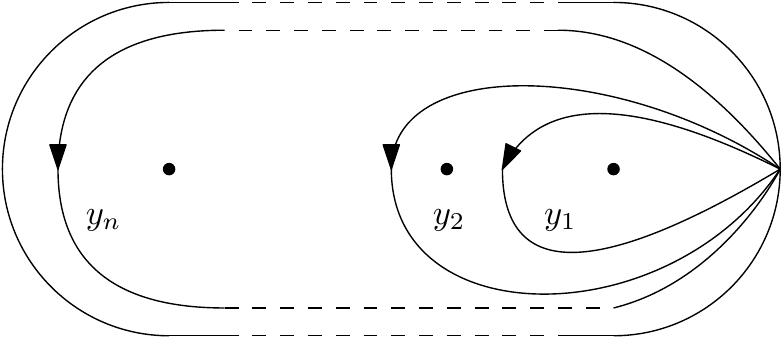}
  \caption{Alternative generators $y_i = x_i \cdots x_1$ of the fundamental group $\pi_1(D_n)$ of the  $n$-puncutred disc.}
  \label{fig:global-braid-generators}
\end{figure}
As shown in Figure \ref{fig:global-braid-generators}, $y_i$ is a path going around the first $i$ punctures.
Along with a basis $e_1, \dots, e_k$ of $V$ we obtain a basis $\{y_i \otimes e_j : i = 1, \dots, n, j = 1, \dots k\}$ for the twisted homology $\homol{D_n; \rho}$, where we identify $y_i$ with its image in homology.
Dropping $y_n$ similarly gives us a basis of $\homol{D_n, \partial D_n;\rho}$.

\begin{proposition}
  \label{prop:boundary-burau-matrix}
  Choose a basis $e_1, \dots, e_k$ of $V$, so that we can identify $\GL(V)$ with $\GL(\C^k)$.
  With respect to the basis
  \[
    \{y_i \otimes e_j : i = 1, \dots n-1, j = 1, \dots, k\}
  \]
  of $\homol{D_n, \partial D_n; \rho}$,
  the matrices of the boundary-reduced twisted Burau representation are given on braid generators $\sigma_i : \rho_0 \to \rho_1$ by
  \begin{align*}
    \left[\burau^{\partial}(\sigma_i)\right]
    &=
    I_{(i-2)k} \oplus
    \begin{bmatrix}
      I_k & 0 & 0 \\
      I_k & -\rho_0(y_{i-1} y_{i}^{-1}) & \rho_0(y_{i-1} y_{i}^{-1}) \\
      0 & 0 & I_k
    \end{bmatrix}
    \oplus I_{(n-i-2)k}
    \\
    &=
    I_{(i-2)k} \oplus
    \begin{bmatrix}
      I_k & 0 & 0 \\
      I_k & -\rho_1(y_{i} y_{i+1}^{-1}) & \rho_1(y_{i} y_{i+1}^{-1}) \\
      0 & 0 & I_k
    \end{bmatrix}
    \oplus I_{(n-i-2)k},
  \end{align*}
  where the matrices act on row vectors from the right.
\end{proposition}
We have chosen the matrices to act on row vectors so that we obtain a representation
\[
  \burau^{\partial}(\beta_1 \beta_2) = \burau^\partial(\beta_1) \burau^\partial(\beta_2)
\]
instead of an anti-representation.
\begin{proof}
  This is a standard result, which can be computed by identifying the action of the braid group on the twisted chain groups with the action of the \emph{Fox derivatives} on the free group $F_n = \pi_1(D_n)$.
  For more details, see \cite{Conway2017}, in particular Example 11.3.7.
  Our matrices differ slightly from those of \cite{Conway2017} because we have picked a different convention for the action of $\braid_n$ on $F_n$.

  To see that $\rho_0(y_{i-1}y_i^{-1}) = \rho_1(y_i y_{i+1}^{-1})$, recall that by definition $\rho' = \rho \beta^{-1}$.
\end{proof}

To match the braid action on the quantum group we want the dual of this representation.
The most convenient way to do this is to consider  \emph{locally-finite} or \emph{Borel-Moore} homology $\homol{D_n; \rho}[\lf]$.

The untwisted form of this homology has a basis spanned by arcs between the punctures of $D_n$, and it is dual to $\homol{D_n, \partial D_n}$ via the obvious intersection pairing.
For example, Figure \ref{fig:borel-moore-basis} shows the basis $y_1, y_2$ of $\homol{D_3, \partial D_3; \C}$ associated to the generators  $y_1, y_2 \in \pi_1(D_3)$ and the dual basis $z_1, z_2$ of $\homol{D_3; \C}[\lf]$.

\begin{figure}
  \includegraphics{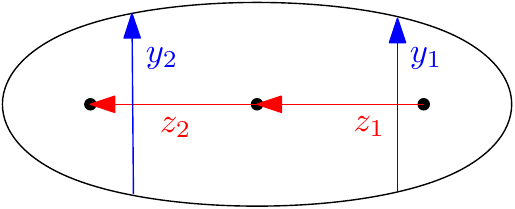}
  \caption{A basis $z_1, z_2$  (in {\color{red} red}) of the locally finite homology $\homol{D_3}[\lf]$ and the dual basis $y_1, y_2$ (in {\color{blue} blue}) of the homology rel boundary $\homol{D_3, \partial D_3}$.}
  \label{fig:borel-moore-basis}
\end{figure}

To extend this to the twisted case, we need to obtain a right $\Z[\pi]$-module dual to $V$.
The dual space $V^* \defeq \hom_\C(V, \C)$ is a right $\Z[\pi]$-module via
\[
  x \cdot f = v \mapsto f(v \rho(x^{-1}) )
\]
We write $(V^\vee, \rho^\vee)$ for this representation.

\begin{proposition}
  There is a $\pi$-equivariant nondegenerate pairing
  \[
    \homol{D_n; \rho^{\vee}}[\lf] \otimes \homol{D_n, \partial D_n; \rho} \to \C.
  \]
\end{proposition}
\begin{proof}
  This is an easy extension of the result for untwisted homology, using the $\pi$-equivariant pairing between $V^\vee$ and $V$ given by
  \[
    x\cdot (f \otimes v) \mapsto f(\rho(x^{-1}) \rho(x) v) = f(v). \qedhere
  \]
\end{proof}

\begin{definition}
  The \emph{reduced twisted Burau representation} (from here on the \emph{Burau representation}) is the functor 
  sending a braid $\beta : \rho_0 \to \rho_1$ to the map
  \[
    \burau(\beta) : H_1^{lf}(D_n; \rho_0^{\vee}) \to H_1^{lf}(D_n; \rho_1^{\vee}).
  \]
\end{definition}

\begin{corollary}
  \label{cor:burau-coords}
  Let $e^1, \dots, e^k$ be the basis of $V^*$ dual to the basis chosen in Proposition \ref{prop:boundary-burau-matrix}, and similarly let $z_1, \dots, z_{n-1}$ be the basis dual to $y_1, \dots, y_{n-1}$.
  Then with respect to the basis $\{z_i \otimes e^j : i = 1, \dots, n-1, j = 1, \dots k\}$ of $\homol{D_n; \rho}[\lf]$, the matrices of the Burau representation $\burau$ are given on braid generators by
  \[
    \left[\burau(\sigma_i)\right]
    =
    I_{(i-2)k} \oplus
    \begin{bmatrix}
      I_k & I_k & 0 \\
      0 & - \rho_1^{\vee}(y_i y_{i+1}^{-1}) & 0 \\
      0 & \rho_1^{\vee}(y_i y_{i+1}^{-1}) & I
    \end{bmatrix}
    \oplus I_{(n-i-2)k},
  \]
  where the matrices $\rho^{\vee}$ are the inverse transposes of those of $\rho$.
\end{corollary}

The above representation is very close to the action (\ref{eq:quantum-burau-action}) on the quantum group but to get them to match we need to change basis.
\begin{proposition}
  \label{prop:burau-coords-nice}
  Let $\sigma_i : \rho_0 \to \rho_1$ be a colored braid generator.
  Assume that it is admissible, so that we can consider $\rho_0$ and $\rho_1$ as objects of $\braidf$, with
  \[
    \rho_1 = (b_1, \dots, b_n) \text{ and }
    b_i = 
    \left(
      \begin{pmatrix}
        \kappa_i & 0 \\
        \phi_i & 1
      \end{pmatrix}
      ,
      \begin{pmatrix}
        1 & \epsilon_i \\
        0 & \kappa_i
      \end{pmatrix}
    \right).
  \]
  There exists a family of bases of the cohomology $\homol{D_n; \rho}[\lf]$ such that the matrix of $\burau(\sigma_i)$ is given by  
  \begin{equation}
    \label{eq:burau-coords-nice}
    I_{2(i-2)} \oplus
    \begin{bmatrix}
      1 & 0 & \kappa_i^{-1} & -\phi_i \kappa_i^{-1} & &  &  \\
      0 & 1 & 0 & 1  \\
        &  & -\kappa_i^{-1} & \phi_i\kappa_i^{-1} &  &  \\
        &  & -\epsilon_{i+1} & -\kappa_{i+1} &  &  \\
        &  & 1 & 0 & 1 & 0 \\
        &  &  \epsilon_{i+i} & \kappa_{i+1} & 0 & 1
    \end{bmatrix}
    \oplus I_{2(n-i-2)}
  \end{equation}
  Later we will denote these bases by $v_j^\nu = v_j^\nu(\rho)$ for $j =1, \dots, n-1$ and $\nu = 1,2$.
\end{proposition}
The matrix \eqref{eq:burau-coords-nice} is exactly the image of \eqref{eq:quantum-burau-action} under the $\Ucentersmall$-characters corresponding to $\rho$.

\begin{proof}
  For a groupoid representation $\mathcal{F} : G \to \vect\C$, choosing bases means choosing a basis of the vector space $\mathcal{F}(\rho)$ for each object $\rho$ of $G$, which gives matrices $[\mathcal{F}(g)]$ for each morphism $g : \rho_0 \to \rho_1$ of $C$.
  Changing the bases transforms the matrix of $g$ as
  \[
    [\mathcal{F}(g)] \mapsto Q_{\rho_0} [\mathcal{F}(g)] Q_{\rho_1}^{-1}
  \]
  where we now have two different change-of-basis matrices on each side.
  (Recall that our matrices are acting on row vectors, so the domain $\rho_0$ goes on the left.)
  The proposition follows from the correct choice of $Q_{\rho_0}$.

  Recall $\rho_1 = (b_1, \dots, b_n)$ and similarly write $\rho_0 = (a_1, \dots, a_n)$.
  To avoid cumbersome notation, we temporarily write $a_i^\pm$ and $b_i^{\pm}$ for the \emph{transposes} of the components of the elements of $\slgroupdual$, for example
  \[
    b_j = 
    (b_j^+, b_j^-) =
    \left(%
      \begin{pmatrix}
        \kappa_j & \phi_j \\
        0 & 1
      \end{pmatrix},
      \begin{pmatrix}
        1 & 0 \\
        \epsilon_j & \kappa_j
      \end{pmatrix} %
    \right).
  \]
  Setting
\[
  p_i(a) \defeq a_i^+ \cdots a_1^+, \quad m_i(a) \defeq a_i^- \cdots a_i^-
\]
we have
\begin{align*}
  \rho_1^\vee(y_i) &= p_i(b)^{-1} m_i(b), \\
  \rho_0^\vee(y_i) &= p_i(a)^{-1} m_i(a), \\
  \intertext{and in particular}
  \rho_1^\vee(y_{i+1}y_{i}^{-1}) &= p_{i+1}(b)^{-1} b_{i+1}^- p_{i}(b), \\
  \rho_0^\vee(y_{i}y_{i-1}^{-1}) &= p_{i}(a)^{-1} a_{i}^- p_{i-1}(a),
\end{align*}
so the non-identity block of the matrix of \cref{cor:burau-coords} is
\begin{align*}
  \begin{bmatrix}
    I_2 & I_2 & 0 \\
    0 & -\rho^\vee_0(y_{i}y_{i-1}^{-1}) & 0 \\
    0 & \rho^\vee(y_{i+1}y_i^{-1}) & I_2
  \end{bmatrix}
  &=
  \begin{bmatrix}
    I_2 & I_2 & 0 \\
    0 & -p_i(a)^{-1} a_i^- p_{i-1}(a) & 0 \\
    0 & p_{i+1}(b)^{-1} b_{i+1}^- p_i(b) & I_2
  \end{bmatrix}
\end{align*}

We want to change basis by the matrices
\[
  Q_a =
  \begin{bmatrix}
    p_1(a) \\
    & \ddots \\
    & & p_{n-1}(a)
  \end{bmatrix}
\]
Because $p_j(a) = p_j(b)$ for all $j \ne i$, we see the identity blocks of the matrix of \cref{cor:burau-coords} are unchanged, while the nontrivial block becomes
\begin{gather*}
  \begin{bmatrix}
    p_{i-1}(a) & 0 & 0 \\
    0 & p_i(a) & 0 \\
    0 & 0 & p_{i+1}(a)
  \end{bmatrix}
  \begin{bmatrix}
    I_2 & I_2 & 0 \\
    0 & -p_{i}(a)^{-1} a_i^- p_{i-1}(a) & 0 \\
    0 & p_{i+1}(b)^{-1} b_{i+1}^- p_i(b) & I_2
  \end{bmatrix}
  \begin{bmatrix}
    p_{i-1}(b)^{-1} & 0 & 0 \\
    0 & p_i(b)^{-1} & 0 \\
    0 & 0 & p_{i+1}(b)^{-1}
  \end{bmatrix}
  \\
  =
  \begin{bmatrix}
    p_{i-1}(a) & p_{i-1}(a) & 0 \\
    0 & -a_i^- p_{i-1}(a) & 0 \\
    0 & p_{i+1}(a) p_{i+1}(b)^{-1} b_{i+1}^- p_i(b) & p_{i+1}(a)
  \end{bmatrix}
  \begin{bmatrix}
    p_{i-1}(b)^{-1} & 0 & 0 \\
    0 & p_i(b)^{-1} & 0 \\
    0 & 0 & p_{i+1}(b)^{-1}
  \end{bmatrix}
  \\
  =
  \begin{bmatrix}
    p_{i-1}(a) p_{i-1}(b)^{-1} & p_{i-1}(a) p_i(b)^{-1} & 0 \\
    0 & - a_i^- p_{i-1}(a) p_i(b)^{-1} & 0 \\
    0 & p_{i+1}(a) p_{i+1}(b)^{-1} b_{i+1}^- p_i(b) p_i(b)^{-1} & p_{i+1}(a)p_{i+1}(b)^{-1}
  \end{bmatrix}
  \\
  =
  \begin{bmatrix}
    p_{i-1}(b) p_{i-1}(b)^{-1} & p_{i-1}(b) p_{i-1}(b)^{-1} (b_i^+)^{-1} & 0 \\
    0 & - a_i^- p_{i-1}(a) p_{i-1}(a)^{-1}(b_i^+)^{-1} & 0 \\
    0 & p_{i+1}(b) p_{i+1}(b)^{-1} b_{i+1}^- p_i(b) p_i(b)^{-1} & p_{i+1}(b)p_{i+1}(b)^{-1}
  \end{bmatrix}
  \\
  =
  \begin{bmatrix}
    I_2 & (b_i^+)^{-1} & 0 \\
    0 & - a_i^- (b_i^+)^{-1} & 0 \\
    0 & b_{i+1}^- & I_2
  \end{bmatrix}
\end{gather*}
Again the cancellations follow from the fact that $p_j(a) = p_j(b)$ for all $j \ne i$.
We have immediately that
\[
  (b_i^+)^{-1}
  =
  \begin{bmatrix}
    \kappa_i^{-1} & - \phi_i \kappa_i^{-1} \\
    0 & 1
  \end{bmatrix}
  \text{ and }
  b_{i+1}^-
  =
  \begin{bmatrix}
    1 & 0 \\
    \epsilon_{i+1} & \kappa_{i+1}
  \end{bmatrix}
\]
so it remains only to check that $-a_i^-(b_i^+)^{-1}$ gives the correct $2 \times 2$ matrix.
Writing
\[
  a_i^-
  =
  \begin{bmatrix}
    1 & 0 \\
    \tilde \epsilon_{i} & \tilde \kappa_i
  \end{bmatrix}
  \text{ and }
  (b_i^+)^{-1}
  =
  \begin{bmatrix}
    \kappa_i^{-1} & - \phi_i \kappa_i^{-1} \\
    0 & 1
  \end{bmatrix}
\]
we have
\[
  a_i^-(b_i^+)^{-1}
  =
  \begin{bmatrix}
    \kappa_i^{-1} & - \phi_i \kappa_{i}^{-1} \\
    \tilde \epsilon_i \kappa_i^{-1} & \tilde \kappa_i - \tilde \epsilon_i \phi_i \kappa_{i}^{-1}
  \end{bmatrix}
\]
To simplify the bottom row we need the identities
\begin{align*}
  \tilde \epsilon_i &= \kappa_i \epsilon_{i+1}
  \\
  \tilde \kappa_i &= \kappa_{i+1} + \phi_i \epsilon_{i+1}
  \\
  \intertext{which follow from the identities}
  \algbraid(E^2 \otimes 1) &= K^2 \otimes E^2
  \\
  \algbraid(K^2 \otimes 1) &= 1 \otimes K^2 + K^2 F^2 \otimes E^2
\end{align*}
from \cref{lem:central-action}.
Then we see that
\[
  a_i^-(b_i^+)^{-1}
  =
  \begin{bmatrix}
    \kappa_i^{-1} & - \phi_i \kappa_{i}^{-1} \\
    \tilde \epsilon_i \kappa_i^{-1} & \tilde \kappa_i - \tilde \epsilon_i \phi_i \kappa_{i}^{-1}
  \end{bmatrix}
  =
  \begin{bmatrix}
    \kappa_i^{-1} & - \phi_i \kappa_{i}^{-1} \\
    \epsilon_{i+1} & \tilde \kappa_i - \epsilon_{i+1} \phi_i
  \end{bmatrix}
  =
  \begin{bmatrix}
    \kappa_i^{-1} & - \phi_i \kappa_{i}^{-1} \\
    \epsilon_{i+1} & \kappa_{i+1}
  \end{bmatrix}
\]
as claimed.
\end{proof}

\subsection{Torsions}
When the complex $C_*(S^3 \setminus L; \rho)$ is acyclic, that is when each space $\homol[*]{S^3 \setminus L; \rho}$ is trivial, we can still extract an invariant called the \emph{torsion}.
Details on the classical case of untwisted/abelian torsions are found in the book \cite{Turaev_2001}.
Twisted torsions and the related twisted Alexander polynomial are discussed in the article \cite{Conway2015} and thesis \cite{Conway2017}, as well as the survey article
\cite{Friedl2009}.

We sketch the definition of the torsion.
Acyclicity is equivalent to exactness of the sequence
\[
  \cdots C_i \xrightarrow{\partial_i} C_{i-1} \cdots
\]
in which case we get isomorphisms $\ker \partial_i = \im \partial_{i-1}$.
If we choose a basis of each $C_i$, we can use the above isomorphisms to change these bases.
The alternating product of determinants of the basis-change matrices gives an invariant of the acyclic complex $C_*$.
In general this torsion can depend on the choice of basis for each chain space, but for link complements it does not.

Given a presentation of $L$ as the closure of a braid $\beta$ we get a presentation of $\pi_L = \left\langle y_1, \dots, y_n | y_i = y_i \cdot \beta\right\rangle$, which in turn gives a CW structure on $S^3 \setminus L$; the $2$-cells are obtained by the relations $y_i = y_i \cdot \beta$.
Link complements are aspherical, so we do not need to add any higher-dimensional cells.

\begin{definition}
  Let $\rho : \pi_1(S^3 \setminus L) \to \operatorname{GL}_k(\C)$ be a representation such that the $\rho$-twisted chain complex $C^*(S^3 \setminus L; \rho)$ is acyclic, in which case we say the $\operatorname{GL}_k(\C)$-link $L$ is \emph{acyclic.}
  Then the \emph{$\rho$-twisted torsion} $\tau(L, \rho)$  is the torsion of the $\rho$-twisted homology $C^*(S^3 \setminus L; \rho)$.
\end{definition}
Usually when $\rho$ has abelian image this is called the \emph{Reidemeister torsion}.
When the image of $\rho$ is nonabelian it is called the \emph{twisted torsion.}
We prefer to instead refer to these cases as abelian and nonabelian torsions.
The torsion can be computed using the Burau representation:
\begin{proposition}
  \label{prop:torsion-computation}
  Let $(L, \rho)$ be a $\operatorname{GL}_k(\C)$-link, and let $\beta$ be a braid whose closure is $L$.
  View $\beta : (g_1, \dots, g_n) \to (g_1, \dots g_n)$ as a morphism of the colored braid groupoid $\braid(\operatorname{GL}_k(\C))$, and suppose that
  \[
    \det(1 - g_n \cdots g_1) \ne 0,
  \]
  that is, that the holonomy $g_n \cdots g_1$ of a path around all the punctures of $D_n$ does not have $1$ as an eigenvalue.
  Then
  \begin{enumerate}
    \item The twisted homology $\homol[*]{S^3 \setminus L, \rho}$ is acyclic, so the torsion $\tau(L, \rho)$ is a complex number defined up to $\pm \det \rho$, 
    \item we can compute the torsion as
      \[
        \tau(L, \rho)
        = \frac{\det(1 - \burau^{\partial}(\beta))}{\det(1 - g_n \cdots g_1)}
        = \frac{\det(1 - \burau(\beta))}{\det(1 - (g_n \cdots g_1)^{-1})},
      \]
    \item and if $(L, \rho)$ is an $\slgroup$-link such that $\det(1 - \rho(x_i)) \ne 0$ for every meridian $x_i$, then such a braid $\beta$ always exists.
  \end{enumerate}
\end{proposition}
\begin{proof}
  (1) and (2) are standard results in the theory of torsions.
The idea is that we use the basis corresponding to $y_1, \dots, y_{n-1}$ for $\homol{D_n; \rho}$ and to $y_n$ for $\homol[0]{D_n; \rho}$, and these bases give nondegenerate matrix $\tau$-chains \cite[\S2.1]{Turaev_2001} for the complex, so they compute the torsion.
  More details can be found in \cite[Theorem 3.15]{Conway2015}; that paper discusses twisted Alexander polynomials, which correspond with the torsion when the variables $t_i$ are all $1$.
  Finally, we can use $\burau$ and locally-finite homology instead of $\burau^\partial$ and ordinary homology to compute the torsion because these are dual.

  The only novel (to our knowlege) claim is (3).
  The proof is due to \citeauthor{PatureauPersonal} \cite{PatureauPersonal}.
  Represent $(L, \rho)$ as the closure of a $\slgroup$-braid $\beta$ on $n$ strands which is an endomorphism of the color tuple $(g_1, \dots, g_n)$, and write $h_n  = g_n \cdots g_1$ for the total holonomy.
  Consider the colored braids
  \begin{align*}
    \beta &: (g_1, \cdots, g_n) \to (g_1, \cdots, g_n)\\
    \beta\sigma_n &: (g_1, \cdots, g_n, g_n ) \to (g_1, \cdots, g_n, g_n)\\
    \beta\sigma_n \sigma_{n+1} &: (g_1, \cdots, g_n, g_n, g_n) \to (g_1, \cdots, g_n, g_n, g_n)
  \end{align*}
  Their closures are all $(L, \rho)$, and they have total holonomies
  \[
    h_n, \ \ g_n h_n, \ \ g_n^2 h_n 
  \]
  respectively.
  Because these matrices all lie in $\slgroup$, we have
  \[
    \tr(g_n^2 h_n) + \tr(h_n) = \tr(g_n)\tr(g_n h_n)
  \]
  Recall that an element $g \in \slgroup$ has $1$ as an eigenvalue if and only if $\tr g = 2$.
  Since $\tr g_n \ne 2$, at least one of $\tr(g_n^2 h_n),$ $\tr(g_n h_n),$ or $\tr h_n$ has trace not equal to $2$.
  We conclude that at least one braid with closure $(L, \rho)$ has nontrivial total holonomy.
\end{proof}
Taking the closure of a braid relates the complex $C^*(D_n; \rho)$ to $C^*(S^3 \setminus L; \rho)$ by adding a term in dimension $2$, so it is reasonable to expect a relationship between the torsion and the Burau representation.
Notice that when the image of $\rho$ lies in $\operatorname{SL}_n$ the torsion is defined up to an overall sign.

\section{Schur-Weyl duality for \texorpdfstring{$\U_i(\lie{sl}_2)$}{Ui sl2}}
In this section we prove our first major result, Theorem \ref{thm:schur-weyl}, which gives a Schur-Weyl duality between the (reduced twisted) Burau representation $\burau$ and the algebra $\U_i(\lie{sl}_2)$.

First, we explain what we mean by ``Schur-Weyl duality.''
Consider a Hopf algebra $H$ and a simple $H$-module~$V$ with structure map $\pi : H \to \End_\C(V)$.
The algebra $H$ acts on $V^{\otimes n}$ via the map $\pi^{\otimes n } \circ \Delta^n : H \to H^{\otimes n} \to \End_\C(V^{\otimes n})$.

We want to understand the decomposition of the tensor product module $V^{\otimes n}$ into simple factors.
One way is to find a subalgebra $B \subseteq H^{\otimes n}$ that commutes with $\Delta^n(H)$, the image of $H$ under the iterated coproduct.
If $B$ is large enough, then we can use the double centralizer theorem to understand the decomposition of $V^{\otimes n}$.
In this section, we address this problem in the case $H = \U_i(\lie{sl}_2)$, with a few modifications.

To get a satisfactory answer, we want think of $\U_i(\lie{sl})_2$ as a superalgebra and find a subalgebra $\clifford_n$ (a Clifford algebra generated by a space $\opspace_n$) that supercommutes with $\Delta^n (\U_i(\lie{sl}_2))$.
In addition, to match the $\slgroup$-colored braid groupoid and its Burau representation, we consider tensor products of the form
\[
  \adjmod{\chi_1} \otimes \cdots \otimes \adjmod{\chi_n}
\]
where $\chi_i : \Ucentersmall \to \C$ are $\Ucentersmall$-characters, equivalently points of $\slgroup^*$.
Since the Burau representation is a braid group representation, we also describe the braiding on $\U_i(\lie{sl}_2)$ and its action on our subalgebra.

\subsection{\texorpdfstring{$\U$}{U} as a superalgebra}

\begin{definition}
  A \emph{superalgebra} is a $\Z/2$-graded algebra.
  We call the degree $0$ and $1$ the even and odd parts, respectively, and write $|x|$ for the degree of $x$.
  We say that $x$ and $y$ \emph{supercommute} if
  \[
    xy - (-1)^{|x| |y|} yx = 0.
  \]
\end{definition}

\begin{example}
  Let $V$ be a module over a commutative ring $Z$ and $\eta$ a symmetric $Z$-valued bilinear form on $V$.
  The \emph{Clifford algebra} generated by $V$ is the quotient of the tensor algebra on $V$ by the relations
  \[
    vw + wv = \eta(v, w)
  \]
  for $v,w \in V$.
  By considering the image of $V$ to be odd the Clifford algebra becomes a superalgebra.
\end{example}

Recall the notation $\F = iKF$.
We can regard $\U$ as the algebra generated by $K^{\pm 1}, E, \F$ with relations
\[
  \{E, K\} = \{\F,  K \} = 0, \quad [E, \F] = \frac{i}{2} K \Omega,
\]
where $\{A, B\} \defeq AB + BA$ is the anticommutator and
\[
  \Omega = K - K^{-1}(1 + E \F).
\]

\begin{proposition}
  $\U$ is a superalgebra with grading
  \[
    |E| = |\F| = 0, \quad |K| = |\Omega| = 1.
  \]
\end{proposition}

The choice that $E$ and $iKF$ (instead of $iKE$ and $F$) are even is for compatibility with the map $\algbraid$.
More generally, our choice of grading here is motivated by Theorem \ref{thm:schur-weyl}, is rather ad hoc, and seems very special to the case $q = i$.
At $q$ a $4m$th root of unity we expect a $\Z/m$-grading instead.

\subsection{Schur-Weyl duality}
We are now equipped to prove Theorem \ref{thm:schur-weyl}.

\begin{definition}
  For $j = 1, \dots, n$, consider the elements
  \begin{align*}
    \alpha_j^1 & \defeq K_1 \cdots K_{j-1} E_j \Omega_j^{-1}  \\
    \alpha_j^2 & \defeq K_1 \cdots K_{j-1} \F_j \Omega_j^{-1}
  \end{align*}
  of $\U[\Omega^{-1}]$, where $\F = KF$, and set
  \[
    \beta_j^\nu = \alpha_{j}^\nu - \alpha_{j+1}^{\nu}.
  \]
  We write $\opspace_n \subseteq (\U[\Omega^{-1}])^{\otimes n}$ for the $(\Ucentersmall[\Omega^{-2}])^{\otimes n}$-span of the $\beta_j^\nu$.
  Similarly, we write $\clifford_n$ for the subalgebra generated by $\opspace_n$.
\end{definition}

\begin{lemma}
  \label{lemma:Cn-is-clifford}
  $\clifford_n$ is a Clifford algebra over the ring $(\Ucentersmall[\Omega^{-2}])^{\otimes n}$.
\end{lemma}
\begin{proof}
  The $\alpha_j^\nu$ satisfy anticommutation relations
  \begin{align*}
    \{\alpha_j^1, \alpha_k^1\} & = 2 \delta_{jk} K_1^2 \cdots K_{j-1}^2 E_j^2 \Omega_j^{-2} \\
    \{\alpha_j^2, \alpha_k^2\} & = 2 \delta_{jk} K_1^2 \cdots K_{j-1}^2 K_j^2 F_j^2 \Omega_j^{-2} \\
    \{\alpha_j^1, \alpha_k^2\} & = 2i \delta_{jk} K_1^2 \cdots K_{j-1}^2 (1 - K_j^{-2}) \Omega_j^{-2}.
  \end{align*}
  In particular, their anticommutators lie in $(\Ucentersmall[\Omega^{-2}])^{\otimes n}$, so the same holds for anticommutators of elements of $\opspace_n$.
\end{proof}

\begin{lemma}
  \label{lem:raising-op-computation}
  The braiding automorphism acts by
  \begin{align*}
    \algbraid(\alpha_1^1) &= \alpha_2^1 \\
    \algbraid(\alpha_2^1) &= K_2^2 \alpha_1^1 + (1 - K_2^2) \alpha_2^1 - E_2^2 (\alpha_1^2 - \alpha_2^2) \\
    \algbraid(\alpha_1^2) &= (1- K_1^{-2}) \alpha_2^1 + K_1^{-2} \alpha_2^2 + F_1^2 (\alpha_1^1 - \alpha_2^1) \\
    \algbraid(\alpha_2^2) &= \alpha_2^1
  \end{align*}
  so that the matrix of $\algbraid_{i, i+1}$ acting on $\opspace_n$ is given by
  \begin{equation}
    \label{eq:quantum-burau-action}
    I_{2(i-2)} \oplus
    \begin{bmatrix}
      1 & 0 & K_i^{-2} & - F_i^2 \\
      0 & 1 & 0 & 1 \\
        & & - K_{i}^{-2} & F_i^2 \\
        & & - E_{i+1}^2 & - K_{i+1}^{2} \\
        & & 1 & 0 & 1 & 0 \\
        & & E_{i+1}^2  & K_{i+1}^2  & 0 & 1
    \end{bmatrix}
    \oplus I_{2(n- 1) - 2(i + 1)}
  \end{equation}
  with the matrix action given by right multiplication on row vectors with respect to the basis $\{\beta_1^2, \beta_1^1, \cdots, \beta_{n-1}^2, \beta_{n-1}^1\}$ of $\opspace_n$.
\end{lemma}
\begin{proof}
  This is straightforward to verify.
\end{proof}

\begin{definition}
  \label{def:nonsingular-char}
  We say that a $\Ucentersmall$-character $\chi$ is \emph{nonsingular} if $\chi(\Omega^2) \ne 0$, equivalently if $\psi(\chi) \in \slgroup$ does not have $1$ as an eigenvalue.
  (See Remark \ref{rem:casimirs}.)
  Similarly, we say an object $(\chi_1, \dots, \chi_n)$ of $\braidf$ is nonsingular if each $\chi_i$ is, and a braid $\beta : \rho \to \rho'$ is nonsingular if $\rho$ and $\rho'$ are.
\end{definition}

In particular, for any nonsingular character $\chi$ the localization $(\adjmod{\chi})[\Omega^{-1}]$ makes sense.

\begin{definition}
  Recall the basis $v_j^\nu$ of  $\homol{D_n;\rho}[\lf]$ constructed in Proposition \ref{prop:burau-coords-nice}.
  For each nonsingular object $\rho = (\chi_1, \dots, \chi_n)$ of $\braidf$, define a linear map
   \[
    \phi_\rho :
    \left\{
    \begin{aligned}
      \homol{D_n;\rho}[\lf]
      &\to
      \opspace_n/\ker{\chi_1 \otimes \cdots \otimes \chi_n}\\
      v_j^\nu &\mapsto \beta_{j}^{\nu}
    \end{aligned}
  \right.
  \]
\end{definition}

\begin{bigtheorem}[Schur-Weyl duality for the Burau representation]
  \label{thm:schur-weyl}
  Let $\beta : \rho \to \rho'$ be a nonsingular $\slgroupdual$-colored braid.
  Write $(\chi_1, \dots, \chi_n)$ for the $\Ucentersmall$-characters corresponding to $\rho$, and similarly for $\rho'$.
  Then, for $n \ge 2$:
  \begin{enumerate}
    \item 
      The diagram commutes:
  \[
    \begin{tikzcd}
      \homol{D_n; \rho}[\lf] \arrow[r, "\burau(\beta)"] \arrow[d, "\phi_\rho"] & \homol{D_n; \rho}[\lf] \arrow[d, "\phi_{\rho'}"] \\
      \opspace_n/\ker(\chi_1 \otimes \cdots \otimes \chi_n) \arrow[r, "\algfunctor(\beta)"] & \opspace_n/\ker(\chi_1' \otimes \cdots \otimes \chi_n')
    \end{tikzcd}
  \]
    \item The subspace $\opspace_n$ generates a Clifford algebra $\clifford_n$ inside $\U^{\otimes n}$ which super-commutes with $\Delta(\U)$, the image of $\U$ in $\U^{\otimes n}$ under the coproduct.
  \end{enumerate}
\end{bigtheorem}

\begin{proof}
  The proof of (1) is essentially done: the last ingredient is the observation that the image of the matrix (\ref{eq:quantum-burau-action}) under  $\chi_1' \otimes \cdots \otimes \chi_n'$ is exactly the matrix (\ref{eq:burau-coords-nice}).

  It remains to prove (2).
  We showed in Lemma \ref{lemma:Cn-is-clifford} that the image of  $\opspace_n$ generates a Clifford algebra.
  We therefore think of the elements $\beta_j^\nu$ as being odd, so to check that they supercommute, we must show that
  \begin{align*}
    \{\Delta K, \beta^\nu_k\} &= 0
                              &
    [\Delta E, \beta^\nu_k] &= 0
    \\
    \{\Delta \Omega, \beta^\nu_k\}  &=  0 
                                    &
    [\Delta \F, \beta^\nu_k] &= 0
  \end{align*}
  where $\{A, B\} \defeq AB + BA$ and  $[A,B] \defeq AB - BA$.
  To check this, we can use the anticommutation relations
  \begin{align*}
    \{\alpha_j^1, \alpha_k^1\} & = 2 \delta_{jk} K_1^2 \cdots K_{j-1}^2 E_j^2 \Omega_j^{-2} \\
    \{\alpha_j^2, \alpha_k^2\} & = 2 \delta_{jk} K_1^2 \cdots K_{j-1}^2 K_j^2 F_j^2 \Omega_j^{-2} \\
    \{\alpha_j^1, \alpha_k^2\} & = 2i \delta_{jk} K_1^2 \cdots K_{j-1}^2 (1 - K_j^{-2}) \Omega_j^{-2},
  \end{align*}
  the fact that $\beta_j^\nu \defeq \alpha_j^{\nu} - \alpha_{j+1}^{\nu}$, and the identity
  \[
    \Omega = K - K^{-1}(1 + E \F). \qedhere
  \]
\end{proof}

\subsection{How to apply Schur-Weyl duality}
\label{subsec:strategy}
The motivation for Theorem \ref{thm:schur-weyl} is to prove Theorem \ref{thm:T-is-torsion}.
The strategy is as follows:
suppose we have a family $X_{\chi}$ of $\U$-modules parametrized by points $\chi$ of $\slgroupdual$ (that is, by $\Ucentersmall$-characters).%
\footnote{In our examples we need extra data, specifically an extension of $\chi$ to a character $\Ucenter = \Ucentersmall[\Omega] \to \C$. This corresponds to the extra choice $\casimirsystem$ appearing in \cref{def:fractional-eigenvalues,thm:BGPR-inv}.}
The choice of modules $X_{\chi}$ leads directly to a quantum holonomy invariant of links, although there are somewhat subtle normalization issues that arise in the holonomy case, as discussed in \S\ref{subsec:module-braiding}.

The value of the invariant on a link with $n$ strands is related to the braid action on tensor products of the form $X_{\chi_1} \cdots X_{\chi_n}$.
To understand them, use semisimplicity\footnote{While the representation theory of $\U_i$ is not semisimple, restricting to nonsingular characters avoids the non-semisimple part. See \cref{prop:semisimplicity}.} to write
\[
  X_{\chi_1} \otimes \cdots \otimes X_{\chi_n} \iso \bigoplus_{i} M_i \otimes_\C Y_i
\]
where the $Y_i$ are distinct simple $\U$-modules and the $M_i$ are the corresponding multiplicity spaces.
If we choose the $X_{\chi}$ appropriately, in particular so that $\clifford_n$ acts faithfully, we can identify the spaces $W_i$ with the supercommutant $\clifford_n$ of $\Delta(\U)$, hence (via Theorem \ref{thm:schur-weyl}) with the Burau representation.%
\footnote{It turns out that there are only two such $Y_i$, which correspond to the even and odd parts of~$\clifford_n$.}
This leads directly to a proof that the torsion is a quantum invariant.

Picking the correct family $X_\chi$ is somewhat difficult, however.
The simple modules $\irrmod{\chi}$ of \S\ref{subsec:irrmods} used in the definition of the BGPR invariant are too small, in the sense that $\clifford_n$ does not act faithfully.
This problem leads us to introduce the quantum double (norm-square) $\doubledfunctor$ in \S\ref{sec:quantum-double}.
This construction also solves the normalization issues alluded to before.

For clarity, we first describe the BGPR functor $\mainfunctor$, summarizing the material of \cite[\S6]{Blanchet2018} in our notation.
We then (\S\ref{sec:quantum-double}) give the construction of $\doubledfunctor$ and its relation to $\mainfunctor$, then finally (\S\ref{sec:proof-of-thm-2}) prove that $\doubledfunctor$ computes the torsion.

\section{The BGPR holonomy invariant}
Theorem \ref{thm:T-is-torsion} refers two holonomy invariants, denoted $\mainfunctor$ and $\doubledfunctor$.
$\mainfunctor$ is the holonomy invariant constructed by Blanchet, Geer, Patureau-Mirand, and Reshetikhin \cite{Blanchet2018}, so we call it the \emph{BGPR invariant}, while $\doubledfunctor$ is the ``quantum double'' or ``norm-square'' of $\mainfunctor$.
(There is also a third holonomy invariant $\scalarfunctor$, which should be understood as a change in normalization of $\mainfunctor$: see \S\ref{subsec:scalar-invariant}.)
Since $\doubledfunctor$ is built using $\mainfunctor$ we first recall the construction of $\mainfunctor$ from \cite{Blanchet2018}.

\subsection{Weight modules for \texorpdfstring{$\U$}{U}}
\label{subsec:weight-module-cat}

\begin{definition}
  A \emph{$\U$-weight module} is a $\U$-module $V$ on which the central subalgebra $\Ucentersmall$ acts diagonalizably.
  Let $\chi$ be a $\Ucentersmall$-character, i.e.~an algebra homomorphism $\Ucentersmall \to \C$.
  We say a representation $V$ of $\U$ \emph{has character} $\chi$ if
  \[
    Z \cdot v = \chi(Z) \cdot v
  \]
  for every $Z \in \Ucentersmall$ and $v \in V$.
  Equivalently, a $\U$-module has character $\chi$ if and only if the structure map factors through $\adjmod{\chi}$.

  We write $\modcat$  for the category of finite-dimensional weight modules, and $\modcat_\chi$ for the subcategory of weight modules with character $\chi$.
\end{definition}
Every simple weight module has a character by definition, and in general any finite-dimensional weight module $V$ decomposes as a direct sum
\[
  V = \bigoplus_{\chi} V_\chi
\]
where $V_\chi$ is the submodule on which $\Ucentersmall$ acts by $\chi$.
More generally, 
\[
  \modcat = \bigoplus_{\chi \in \slgroupdual} \modcat_\chi
\]
is a $\slgroupdual$-graded category.

\begin{theorem}
  \label{prop:semisimplicity}
  Let $\chi$ be a nonsingular $\Ucentersmall$-character.
  Then
  \begin{enumerate}
    \item $\modcat_\chi$ is semisimple,
    \item the simple objects of $\modcat_\chi$ are all $2$-dimensional and projective, and
    \item isomorphism classes of simple objects are parametrized by the Casimir $\Omega$, which acts by a square root of $\tr \psi(\chi) - 2$.
  \end{enumerate}
\end{theorem}
\begin{proof}
  This is a special case of \cite[Theorem 6.2]{Blanchet2018}.
  The idea is to use a certain Hamiltonian flow (the quantum coadjoint action of Kac-de Concini-Procesi) on $\operatorname{Spec} \Ucentersmall$  to reduce to the case $\chi(E^2) = \chi(F^2) = 0$.
\end{proof}

In particular, the character $\epsilon(K^2) = 1$, $\epsilon(E^2) = \epsilon(F^2) = 0$ corresponding to the identity matrix is singular.
The category $\modcat_\epsilon$ is the category of modules of the small quantum group, which is not semisimple.

\subsection{Simple weight modules}
\label{subsec:irrmods}
We discuss the modules of Proposition \ref{prop:semisimplicity} in more detail.
\begin{definition}
  \label{def:casimir-values}
  Let $\chi$ be a $\Ucentersmall$-character corresponding to the $\slgroupdual$ element $a$, and let $\mu$ be a complex number with $(\mu - \mu^{-1})^2 = \chi(\Omega^2) = \tr \psi(\chi) - 2$.
  Since $\mu^2$ is an eigenvalue of $\psi(\chi)$ (Remark \ref{rem:casimirs}), we call $\mu$ a \emph{fractional eigenvalue} for $\chi$.

  The character $\chi$ is nonsingular if and only if $\mu \ne \pm 1$.
  In this case, we write
  \[
    \irrmod{\chi, \mu} = \irrmod{a, \mu} = \irrmod{\hat \chi}
  \]
  for the simple module of dimension $2$ with character $\chi$ on which the Casimir $\Omega$ acts by the scalar $\mu - \mu^{-1}$.
  Here $\hat \chi$ is the \emph{extension} of $\chi$ to $\Ucenter$ given by $\hat \chi(\Omega) = \mu - \mu^{-1}$.
\end{definition}

\begin{example}
  \label{ex:simple-modules}
  Let $\chi$ be a nonsingular character and $\mu$ a factional eigenvalue, and let $\irrmod{\chi, \mu}$ be the corresponding irreducible $2$-dimensional $\U$-module.
  It is not hard to see that we can always choose an eigenvector $\ket 0 $ of $K$ such that $ \ket 0 , \ket 1 \defeq E \ket 0$ is a basis of $\irrmod{\chi, \mu}$.
  
  First consider the case where $\epsilon\phi \ne 0$; this is generically true, since non-triangular matrices are dense in $\slgroup$.
  Then with respect to the basis $\ket 0, \ket 1$, the generators act by
  \begin{align*}
    \pi(K) &= \begin{pmatrix}
      \sqrt \kappa & 0\\
      0 & -\sqrt \kappa
    \end{pmatrix}, \quad
    \pi(E) = \begin{pmatrix}
      0 & \epsilon \\
      1 & 0
    \end{pmatrix}, \\
    \pi(F) &= \begin{pmatrix}
      0 & -i(\omega + \sqrt \kappa - \sqrt \kappa^{-1}) \\
      -i(\omega - \sqrt \kappa + \sqrt \kappa^{-1})/\epsilon & 0
    \end{pmatrix}
  \end{align*}
  where $\omega = \mu - \mu^{-1}$ and $\sqrt \kappa$ is an arbitrarily chosen square root of $\kappa$.
  We can think of $\ket 0 $ and $\ket 1$ as a weight basis.
  Since $E$ and $F$ act invertibly $\irrmod{\chi, \omega}$ is sometimes called a \emph{cyclic} module.
  
  The case $\epsilon \phi = 0$ is simpler.
  Then one (or both) of $E,F$ act nilpotently, so $\irrmod{\chi, \mu}$ is said to be \emph{semi-cyclic} (or \emph{nilpotent.})
  Suppose in particular that $\phi = 0$ (the case $\epsilon = 0$ is similar) and choose an eigenvector $\ket 0 $ of $K$ with $F \ket 0 = 0$.
  Then the action of the generators is given by
  \[
    \pi(K) = \begin{pmatrix}
      \mu & 0\\
      0 & -\mu
    \end{pmatrix}, \quad
    \pi(E) = \begin{pmatrix}
      0 & \epsilon \\
      1 & 0
    \end{pmatrix}, \quad
    \pi(F) = \begin{pmatrix}
      0 & -2i (\mu - \mu^{-1}) \\
      0 & 0
    \end{pmatrix}.
  \]
\end{example}

\begin{proposition}
  \label{prop:tensordecomp} 
  Let $\chi_i$ be nonsingular characters with fractional eigenvalues $\mu_i$.
  If the product character $\chi_1 \cdots \chi_n$ is nonsingular, then
  \[
    \bigotimes_{i=1}^n \irrmod{\chi_i, \mu_i}
    \iso
    \irrmod{\chi_1 \cdots \chi_n, \mu}^{\oplus 2^{n-2}} \oplus  \irrmod{\chi_1 \cdots \chi_n, -\mu}^{\oplus 2^{n-2}}
  \]
  where $\mu$ is a fractional eigenvalue for the product character.
\end{proposition}
\begin{proof}
  This is easy to check for $n = 2$, and the general case follows by induction.
\end{proof}
We think of the left-hand side as representing a colored braid on $n$ strands with colors $(\chi_1, \dots, \chi_n)$, so a path wrapping around the entire braid has holonomy $\chi_1 \cdots \chi_n$.
The proposition says that the corresponding tensor product of irreps decomposes in to an equal number of summands of each module with $\Ucentersmall$-character $\chi_1 \cdots \chi_n$.

\subsection{The braiding for modules}
\label{subsec:module-braiding}
The automorphism $\algbraid$ acts on the algebras $\U$, but to construct a holonomy invariant we want a braiding on $\U$-modules.
Such a braiding is a family of maps intertwining $\algbraid$ in the following sense:
\begin{definition}
  \label{def:holonomy-braiding}
  Let $\chi_1, \chi_2$ be nonsingular $\Ucentersmall$-characters such that $(\chi_4, \chi_3) = B(\chi_1, \chi_2)$ exists (equivalently, such that the $\slgroupdual$-colored braid $\sigma : (\chi_1, \chi_2) \to (\chi_4, \chi_3)$ is admissible.)
  For each $i$, let $X_{\chi_i}$ be a module with character $\chi_i$.
  We say a map 
  \[
    c : X_{\chi_1} \otimes X_{\chi_2} \to X_{\chi_4} \otimes X_{\chi_3}
  \]
  of $\U$-modules is a \emph{holonomy braiding} if for every $u \in \U \otimes \U$ and $x \in X_{\chi_1} \otimes X_{\chi_2}$, we have
  \[
    c(u \cdot x) = \algbraid(u) \cdot c(x)
  \]
\end{definition}
  Since $\algbraid$ preserves the coproduct, a holonomy braiding is automatically a map of $\U$-modules.
  Because the modules $\irrmod{\chi, \mu}$ are simple, the choice of a holonomy braiding is essentially unique:
\begin{proposition}
  \label{prop:holonomy-braiding-exits}
  Let $\chi_i, i = 1, \dots, 4$  be characters as in Definition \ref{def:holonomy-braiding}, and choose Casimir values $\mu_1, \mu_2$ for $\chi_1, \chi_2$.
  Then there is a nonzero holonomy braiding
  \[
    c : \irrmod{\chi_1, \mu_1} \otimes \irrmod{\chi_2, \mu_2} \to \irrmod{\chi_4, \mu_2} \otimes \irrmod{\chi_3, \mu_1}
  \]
  unique up to an overall scalar.
\end{proposition}
\begin{proof}
  We first explain why $\mu_1$ is a fractional eigenvalue for $\chi_3$.
  Observe that, by \eqref{eq:biquandlequations}, the matrix $\psi(\chi_3)$ is conjugate to $\psi(\chi_1)$, so it has the same eigenvalues.
  In parallel, we have the fact that
  \[
    \algbraid(\Omega \otimes 1) = 1 \otimes \Omega.
  \]
  A similar argument shows that $\mu_2$ is a fractional eigenvalue for $\chi_4$.

  The remainder of the proof follows the discussion proceeding \cite[Theorem 6.2]{Blanchet2018}.
  Write $\hat \chi_i$ for the $\Ucenter$-character extending $\chi_i$ by $\hat \chi_i(\Omega) = \mu_i - \mu_i^{-1}$, setting $\mu_3 = \mu_1, \mu_4 = \mu_2$.
  For each $i$, the algebra
  \[
    \adjmod{\hat \chi_i}
  \]
   is isomorphic to the $\C$-endomorphism algebra $\operatorname{Mat}_2(\C)$ of $\irrmod{\hat \chi_i} = \irrmod{\chi_i, \mu_i}$.
  In particular, the automorphism $\algbraid$ induces an automorphism of matrix algebras
  \[
    \operatorname{Mat}_2(\C) \otimes \operatorname{Mat}_2(\C) \to \operatorname{Mat}_2(\C) \otimes \operatorname{Mat}_2(\C),
  \]
  equivalently, an automorphism of matrix algebras
  \[
    \check{R} : \operatorname{Mat}_4(\C) \to \operatorname{Mat}_4(\C)  .
  \]
  By linear algebra, any such automorphism is inner, given by $\check R(X) = c X c^{-1}$ for some invertible matrix $c$, which is unique up to an overall scalar.
  The matrix $c$ gives the holonomy braiding with respect to the bases of $V(\chi_i, \mu_i)$ implicit in the isomorphisms $\adjmod{\hat \chi_i} \iso \operatorname{Mat}_2(\C)$.
\end{proof}

The holonomy braidings fit together into a representation of the colored braid groupoid into $\modcat$, although at the moment only a projective one.
To state this precisely, we need a variant of $\braidf$ that keeps track of the fractional eigenvalues.

\begin{definition}
  \label{def:extended-chars}
  An \emph{extended $\Ucentersmall$-character}%
  \footnote{From an algebraic perspective, it would be slightly simpler to say that an extended character is an extension of $\chi : \Ucentersmall \to \C$ to a homomorphism $\chi : \Ucenter \to \C$.
  Such an extension is the same as a choice of scalar $\omega$ satisfying $\omega^2 = \chi(\Omega^2)$.
  To connect with the geometric situation we prefer to emphasize the fractional eigenvalue $\mu$, which has $\omega = \mu - \mu^{-1}$.}
  is a $\Ucentersmall$-character $\chi$ and a fractional eigenvalue $\mu$, that is a complex number with $(\mu - \mu^{-1})^2 = \chi(\Omega^2) - 2$.
  The biquandle $B$ on $\Ucentersmall$-characters of Definition \ref{def:sl2star-biquandle} extends to extended characters via
  \[
    B( (\chi_1, \mu_1), (\chi_2, \mu_2) ) = ((\chi_4, \mu_2), (\chi_3, \mu_1))
  \]
  that is, by permuting the fractional eigenvalues.
  This is well-defined because $\algbraid(1 \otimes \Omega) = \Omega \otimes 1$ and $\algbraid(\Omega \otimes 1) = 1 \otimes \Omega$, as discussed in the proof of Proposition \ref{prop:holonomy-braiding-exits}.

  The \emph{extended $\slgroupdual$-colored braid groupoid} is the category $\braidfh$ with objects tuples of extended $\Ucentersmall$-characters and morphisms admissible braids between them, with the action on extened characters given by the biquandle $B$.
\end{definition}

\begin{proposition}
  The holonomy braidings $c$ of Proposition \ref{prop:holonomy-braiding-exits} give a projective functor 
  \[
    \widetilde{\mainfunctor} : \braidfh \to \modcat
  \]
  defined by
  \[
    \widetilde{\mainfunctor} \left( (\chi_1, \mu_1), \dots, (\chi_n, \mu_n) \right)
    =
    \irrmod{\chi_1, \mu_1} \otimes \cdots \otimes \irrmod{\chi_n, \mu_n}
  \]
  and, if $\sigma : \left((\chi_1, \mu_1), (\chi_2, \mu_2)\right) \to \left((\chi_4, \mu_4), (\chi_3, \mu_3)\right)$ is a braid generator,
  \[
    \widetilde{\mainfunctor}(\sigma)
    =
    c :
    \irrmod{\chi_1, \mu_1} \otimes \irrmod{\chi_2, \mu_2}
    \to
    \irrmod{\chi_4, \mu_4} \otimes \irrmod{\chi_3, \mu_3}
  \]
\end{proposition}
Here by a projective functor we mean that the image maps are considered only up to multiplication by an arbitrary nonzero scalar.
More formally, we could say the codomain is the category $\modcat/\C^{\times}$ where two morphisms are equal if $f = zg$ for some $z \in \C\setminus \{0\}$.
\begin{proof}
  The map $\algbraid$ satisfies the colored braid relations of $\braidf$.
  Because the maps $c_{\chi_1, \chi_2}$ are defined uniquely to intertwine $\algbraid$ they must as well.
\end{proof}

The projective functor $\widetilde{\mainfunctor}$ yields link invariants in $\C$ that are defined up to multiplication by an element of $\C \setminus \{0\}$.
These are not very useful, so we want to remove or at least reduce the scalar indeterminacy in the holonomy braidings $c_{\chi_1, \chi_2}$
However, this is a rather subtle problem, because we really need to choose a \emph{family} of scalars parametrized by pairs of elements of $\slgroupdual$.

\citeauthor{Blanchet2018} show that the scalar ambiguity can at least be reduced to a power of $i$:

\begin{proposition}
  \label{prop:mainfunctor}
  The holonomy braidings can be chosen to satisfy the relations of $\braidfh$ up to a fourth root of unity, thus defining a functor
  \[
    \mainfunctor : \braidfh \to \modcat/\left\langle i \right\rangle.
  \]
\end{proposition}
Here by $\modcat/\left\langle i \right\rangle$ we mean the category that is the same as $\modcat$, except that two morphisms $f, g : V \to V'$ are considered equal if $f = i^k g$ for some $k$.
\begin{proof}
  This is a special case of \cite[Theorem 6.2]{Blanchet2018}.
  The idea is to choose a consistent family of bases for the modules $\irrmod{\chi, \mu}$, then scale the holonomy braidings $c_{\chi_1, \chi_2}$ so that their matrices with respect to these bases have determinant $1$.
  To do this we must divide by $(\det c)^{1/4}$, which gives the indicated ambiguity.
\end{proof}

\begin{remark}
  \label{rem:future-norms}
  This method of defining the holonomy braiding is rather ad hoc and does not give a completely unambiguous solution.
  Later we show that there is a natural way to normalize the quantum double $\doubledfunctor$ of $\mainfunctor$, but this still leaves the problem of understanding $\mainfunctor$ on its own.
  We expect that future work \cite{McPhailSnyderUnpub1} will clarify the situation.
\end{remark}

\subsection{Modified traces}
\label{subsec:modified-trace}
If $(L, \rho)$ is the closure of an endomorphism $\beta$ of $\braidfh$, then the quantum trace of $\mainfunctor(\beta)$ will be an invariant of $(L, \rho)$.
However, because the category $\modcat$ is not semisimple, the quantum dimensions of the modules $\irrmod{\chi, \mu}$ are all zero, so link invariants constructed in this way will also be uniformly zero.

To fix this, we use the theoy of \emph{modified traces.}
We first recall the construction of the usual quantum trace on $\modcat$.
\begin{proposition}
  \label{prop:pivotal-cat}
  $\modcat$ is a pivotal category with pivot $K^{-1} \in \U$.
  That is, for an object $V$ of $\modcat$ with basis $\{v_j\}$ and dual basis $\{v^j\}$ the \emph{coevaluation (creation, birth)} and \emph{evaluation (annihilation, death)} morphisms are given by
  \begin{align*}
    \coev{V} &: \C \to V \otimes V^*,  & &1 \mapsto  \textstyle \sum_j v_j \otimes v^j \\
    \coevbar{V} &: \C \to V^* \otimes V,  & &1 \mapsto  \textstyle \sum_j v_j \otimes K v^j \\
    \ev{V} &: V \otimes V^* \to \C,  & &v \otimes f \mapsto f(K^{-1} v) \\
    \evbar{V} &: V^* \otimes V \to \C,& &  f \otimes v \mapsto f( v)
  \end{align*}
  so that for any morphism $f : V \to V$, the  \emph{quantum trace} is the complex number
  \[
    \tr f \defeq \ev V (f \otimes \id_{V^*} ) \coev V 
  \]
  where we identify linear maps $\C \to \C$ with elements of $\C$.
  The \emph{quantum dimension} of $V$ is $\tr(\id_V)$.
  Furthermore, $\modcat$ is spherical: the right trace above agrees with the left trace
  \[
    \evbar V (\id_{V^*} \otimes f) \coevbar V.
  \]
\end{proposition}
\begin{proof}
  This works because the square of the antipode of $\U$ is given by conjugation with the grouplike element $K^{-1}$.
  It is not hard to check directly that the left and right traces agree.
\end{proof}

We explain the basic idea of modified traces before giving a formal definition.
Let $f : V_1 \otimes \cdots \otimes V_n \to V_1 \otimes \cdots \otimes V_n$ be an endomorphism of $\modcat$.
We can take the partial quantum trace on the right-hand tensor factors $V_2 \otimes \cdots \otimes V_n$ to obtain a map
\[
  \operatorname{ptr}_q^r(f) : V_1 \to V_1
\]
If $V$ is irreducible, $\operatorname{ptr}_q^r(f) = x \id_{V_1}$ for some scalar $x$, so we say that $f$ has modified trace
\[
  \rentr( f ) \defeq x\, \rendim( V_1)
\]
where $\rendim(V_1)$ is the \emph{renormalized dimension} of $V_1$.

Of course, the trace will depend on the choice of renormalized dimensions.
In the case $V_1 = \cdots = V_n$, such as when defining the abelian Conway potential, the choice of renormalized dimension only affects the normalization of the invariant.
However, when the modules can differ, the numbers $\rendim(V)$ must be chosen carefully to insure that we obtain a link invariant.

\begin{theorem}
  \label{prop:C-has-rentrs}
  Let $\proj$ be the subcategory of $\modcat$ of projective modules.
  $\proj$ admits a nontrivial modified trace $\rentr$, unique up to an overall scalar.
  That is, for every projective object $V$ of $\modcat$ there is a linear map
  \[
    {\rentr}_V : \End_\modcat(V) \to \C,
  \]
  these maps are cyclic in the sense that for any $f : V \to W$ and $g : W \to V$ we have
  \[
    { \rentr }_V(gf) = { \rentr }_W(fg),
  \]
  and they agree with the partial quantum traces in the sense that if $V \in \operatorname{Proj}$ and $W$ is any object of $\modcat$, then for any $f \in \End_\modcat(V \otimes W)$, we have
  \[
    { \rentr }_{V \otimes W}(f) = { \rentr }_V( \operatorname{ptr}_W^r(f))
  \]
  where $\tr_W^r$ is the partial quantum trace on $W$.

  This trace corresponds to the renormalized dimensions
  \[
    \rendim(\irrmod{\hat \chi}) \defeq \rentr(\id_{\irrmod{\chi, \mu}}) =  \frac{1}{\mu - \mu^{-1} }
  \]
  on simple modules.
\end{theorem}
We usually omit the subscript on $\rentr$, and we have chosen a different normalization of the dimensions than in \cite{Blanchet2018} which is more natural for $q = i$.
Notice that the renormalized dimensions are gauge-invariant because the fractional eigenvalues are.
\begin{proof}
  See \cite[\S6.3]{Blanchet2018} and \cite{Geer2017}.
  The paper \cite{Geer2018} gives a general construction of modified traces, which we summarize in Appendix \ref{appendix:modified-traces}.
\end{proof}

\subsection{Construction of the invariant}
To match the extra data in $\braidfh$, we need to make some extra choices on our $\slgroup$-links.
\begin{definition}
  \label{def:fractional-eigenvalues}
  Let $(L, \rho)$ be an $\slgroup$ link with components $L_1, \dots, L_k$, and let $x_i \in \pi_1(S^3 \setminus L)$ represent a meridian of the $i$th component $L_i$.
  Any other choice of meridian is conjugate to $x_i$, so the eigenvalues of $\rho(x_i)$ depend only on $L$ and $\rho$.
  A \emph{fractional eigenvalue} for $L_i$ is a complex number $\mu_i$ such that $\mu_i^2$ is an eigenvalue of $\rho(x_i)$.

  Let ${\bm \mu} = \{\mu_i\}$ be a choice of fractional eigenvalue for each component of $L_i$.
  We call $(L, \rho, {\bm \mu})$ an \emph{extended $\slgroup$-link}, and say that it is \emph{admissible} if (forgetting the $\bm \mu$) it is the closure of an admissible $\slgroup$-braid.
\end{definition}

\begin{proposition}
  Any extended $\slgroup$-link is gauge-equivalent to an admissible link.
\end{proposition}
\begin{proof}
  Gauge transformations do not affect the fractional eigenvalues $\bm \mu$, so we can apply Proposition \ref{prop:admissible-exist}.
\end{proof}

\begin{theorem}
  \label{thm:BGPR-inv}
  Let $(L, \rho)$ be a framed $\slgroup$-link that is nonsingular, in the sense that $\rho(x)$ does not have $1$ as an eigenvalue for any meridian $x$ of $L$.
  Choose fractional eigenvalues $\bm \mu$ for $L$.
  Then $L$ is gauge-equivalent to a link $(L, \rho')$ that can be written as the closure of a braid $\beta$ in $\braidfh$, and the complex number
  \[
    \mainfunctor(L, \rho, {\bm \mu}) \defeq \rentr(\mainfunctor(\beta))
  \]
  is a invariant of $(L, \rho, {\bm \mu})$, depending only on the conjugacy class of $\rho$ and defined up to a fourth root of unity.
  We call $\mainfunctor(L, \rho, {\bm \mu})$ the \emph{Blanchet--Geer--Patureau-Mirand--Reshetikhin holonomy invariant}, or the \emph{BGPR invariant}.
\end{theorem}
\begin{proof}
  The details of this construction are given in \cite{Blanchet2018}, and this result is a special case of \cite[Corollary 6.11]{Blanchet2018}.
  We summarize some of the key points:
  \begin{itemize}
    \item Because $\mainfunctor$ is a functor with domain $\braidfh$, we get invariance under Reidemeister II and III moves.
    \item Since we think of $L$ as a framed link, we do not need to check the Reidemeister I move.
      We do need to check that left and right twists agree, which is \cite[Theorem 6.8]{Blanchet2018}.
    \item Gauge invariance is nontrivial to check, but is a consequence of the fact that we chose simple modules $\irrmod{\chi, \mu}$ \cite[Theorem 5.10]{Blanchet2018}.\qedhere
  \end{itemize}
\end{proof}

\begin{remark}
  The dependence on the framing is somewhat unsatisfactory.
  We expect that future work \cite{McPhailSnyderUnpub1} will explicitly compute the framing dependence and allow us to eliminate it.
\end{remark}

By \cref{thm:T-is-torsion}, $\mainfunctor(L, \rho, {\bm \mu})$ (with a slightly different normalization: see \S\ref{subsec:scalar-invariant}) is a sort of square root of the nonabelian torsion $\tau(L, \rho)$.
As discussed in \S\ref{subsec:conway-potential}, this leads us to interpret it as a nonabelian Conway potential.
We can immediately see that it gives the usual Conway potential in the abelian case.

\begin{corollary}
  Let $L$ be any link in $S^3$, and consider the representation $\rho$ of its complement defined by
   \[
     \rho(x) =
     \begin{pmatrix}
       t & 0 \\
       0 & t^{-1}
     \end{pmatrix}
  \]
  where $x$ is any meridian of $L$.
  Then $\mainfunctor(L, \rho, t^{1/2})$ is the Conway potential of $L$, i.e.
  \[
    \mainfunctor(L, \rho, t^{1/2}) = \frac{\Delta_L(t)}{t^{1/2} - t^{-1/2}}
  \]
  where $\Delta_L(t)$ is the Alexedander polynomial of $L$, normalized so that it is symmetric under $t \to t^{-1}$.
\end{corollary}
\begin{proof}
  By \cite[Theorem 4.11]{Blanchet2018}, for abelian representations as above the BGPR invariant at a $2r$th root of unity gives the $r$th ADO invariant (originally defined in \cite{Akutsu1992}).
  The present case is $r = 2$, so the ADO invariant is the Conway potential.

  Alternately, we can directly compare this simple case of our construction to the construction of the Conway potential in \cite{Viro2002}.
\end{proof}

\section{The graded quantum double}
\label{sec:quantum-double}
For a number of reasons, particularly the indeterminacy of the braiding, the BGPR functor $\mainfunctor$ is difficult to work with directly.
Instead we use a functor $\doubledfunctor$ constructed from $\mainfunctor$ by a $G$-graded version of the quantum double.
We will define $\doubledfunctor$ as an ``external'' tensor product of $\mainfunctor$ and its mirror image $\overline{\mainfunctor}$.

The point of the somewhat elaborate definition of $\doubledfunctor$ is that we can directly relate the braiding given by $\doubledfunctor$ on modules to the braiding $\algfunctor$ on the algebra $\U$, which allows us to apply \cref{thm:schur-weyl} to compute the invariants coming from $\doubledfunctor$ and thereby prove \cref{thm:T-is-torsion}.

\subsection{The mirror image \texorpdfstring{$\overline{\mainfunctor}$}{bar F} of \texorpdfstring{${\mainfunctor}$}{F}}
\label{subsec:mirror-image}
We will define (up to scalars) $\doubledfunctor \defeq \mainfunctor \boxtimes \overline{\mainfunctor}$, so we must first define $\overline{\mainfunctor}$.

\begin{definition}
  Write $\U^{\cop}$ for the quantum group $\U_i(\lie{sl}_2)$ with the opposite coproduct.
  As for $\U$, a \emph{weight module} $V$ is $\U^{\cop}$-module on which the center $\Ucenter$ acts diagonalizably.
  We write $\overline{\modcat}$ for the category of finite-dimensional $\U^{\cop}$-weight modules.

  We say that an object $V$ of $\overline{\modcat}$ has character $\chi$ if for any $z \in \Ucentersmall$ and $v \in V$,
  \[
    z \cdot v = \chi(S(z)) v,
  \]
  where $S$ is the antipode of $\U$.
\end{definition}

\begin{example}
  The duals 
  \[
    \irrmod{\chi, \mu}^* \defeq \hom_\C(\irrmod{\chi, \mu}, \C),
    \quad
    (x \cdot f)(v) = f(S(x) \cdot v)
  \]
  of the simple $\U$-modules $\irrmod{\chi, \mu}$ of \S\ref{subsec:irrmods} are objects of $\overline{\modcat}$ with character $\chi$.
\end{example}

\begin{remark}
  The character $\chi S$ is the inverse of $\chi$ in the group $\spec \Ucentersmall$.
  We have chosen to invert the gradings on $\overline{\modcat}$ so that it is $\slgroupdual$-graded, not $(\slgroupdual)^{\op}$-graded.
  This will be more convenient for when we take tensor products in \cref{subsec:tensor-products}, although it makes the mirror image functor (\cref{def:mirror-image-functor}) slightly more complicated.
\end{remark}

\begin{proposition}
  \label{prop:mirror-modified-dimensions}
  There exists a modified trace $\rentr$ on the projective objects of $\overline{\modcat}$ which assigns
  \[
    \rendim(\irrmod{\chi, \mu}^*) = \frac{1}{\mu - \mu^{-1}}.
  \]
\end{proposition}
\begin{proof}
  This is an easy corollary of \cref{prop:C-has-rentrs} given the general theory of Appendix \ref{appendix:modified-traces}.
\end{proof}

Recall that the holonomy braidings (\cref{def:holonomy-braiding}) for $\modcat$ are maps intertwining $\algbraid = \tau \rmat$.
This leads to a braiding because
\[
  \algbraid \Delta = \Delta,
  \text{ equivalently }
  \rmat \Delta = \Delta^\op.
\]
Dually, can think of $\rmat^{-1}$ as intertwining $\Delta^{\op}$ and $\Delta = (\Delta^{\op})^{\op}$,
\[
  \rmat^{-1} \Delta^{\op} = \Delta,
\]
so to get maps commuting with $\Delta^\op$ we look for those that intertwine the automorphism
\[
  \algbraidd \defeq \tau \rmat^{-1} = \tau \algbraid^{-1} \tau
\]
of $\U^{\cop} \otimes \U^{\cop}$.
More generally, a holonomy braiding for $\overline{\modcat}$ should be a family of linear maps intertwining $\algbraidd$.

\begin{lemma}
  Consider a colored braid generator
  \[
    \sigma : (\chi_1 , \chi_2) \to (\chi_4, \chi_3)
  \]
  where $\chi_1, \dots, \chi_4$ are $\Ucentersmall$-characters, so that
  \[
    (\chi_4 \otimes \chi_3)
    \algbraid
    =
    \chi_1 \otimes \chi_2.
  \]
  Then
  \[
    (\chi_4^{-1} \otimes \chi_3^{-1})
    \algbraidd
    =
    (\chi_1^{-1} \otimes \chi_2^{-1})
  \]
  where $\chi_i^{-1} = \chi_i\circ S$ is the inverse character.
\end{lemma}
\begin{proof}
  We can check this via direct computations on the central generators using the relations of \cref{lem:central-action}.
\end{proof}

\begin{proposition}
  \label{prop:mirror-holonomy-brading}
  Let $\chi_1, \chi_2$ be $\Ucentersmall$-characters such that $\sigma : (\chi_1, \chi_2) \to (\chi_4, \chi_3)$ is well-defined, and let $\mu_1$ and $\mu_2$ be fractional eigenvalues for $\chi_1$ and $\chi_2$, respectively.
  Then there is a linear map
  \[
    \overline{c} :
    \irrmod{\chi_1, \mu_1}^* \otimes \irrmod{\chi_2, \mu_2}^*
    \to
    \irrmod{\chi_4, \mu_2}^* \otimes \irrmod{\chi_3, \mu_1}^*
  \]
  intertwining $\algbraidd$ in the sense that for any $x \in \U^{\cop} \otimes \U^{\cop}$,
  \[
    \overline{c}(x \cdot v) = \algbraidd(x) \cdot \overline{c}(v).
  \]
  The map $\overline{c}$ is unique up to an overall scalar.
\end{proposition}
\begin{proof}
  The previous lemma shows that the characters transform correctly under the braiding.
  The proof then goes exactly as for \cref{prop:holonomy-braiding-exits}.
\end{proof}

We will choose the normalization of $\overline{c}$ to match $c$ in the mirror image, so first we need to explain how to take mirror images of colored braids.

\begin{definition}
  For a link $L$, the \emph{mirror} $\overline{L}$ of $L$ is the image of $L$ under an orientation-reversing homeomorphism $r : S^3 \to S^3$.
  For an $\slgroup$-link $(L, \rho)$, the \emph{mirror} is defined to be $(\overline{L}, \overline{\rho})$, where $\overline{\rho} \defeq \rho r_*$ is obtained by pulling back from $\pi_{\overline{L}}$ to $\pi_L$ along $r$.
\end{definition}

\begin{definition}
  \label{def:mirror-image-functor}
  The \emph{mirror image functor} $\mathcal{M} : \braidf \to \braidf$ is given by
  \[
    \mathcal{M}(\chi_1, \dots, \chi_n) = (\chi_n^{-1}, \dots, \chi_1^{-1})
  \]
  on objects.
  It is defined on braid generators $\sigma : (\chi_1, \chi_2) \to (\chi_4, \chi_3)$ by
  \[
    \mathcal{M}(\sigma) = \sigma^{-1} : (\chi_2^{-1}, \chi_1^{-1}) \to (\chi_3^{-1}, \chi_4^{-1})
  \]
  with the obvious extension to all morphisms.

  We emphasize that $\mathcal{M}$ is covariant.
  We can extend $\mathcal{M}$ to $\braidfh$ by preserving the fractional eigenvalues, so that
  \[
    \mathcal{M}((\chi_1, \mu_1), (\chi_2, \mu_2)) = ((\chi_2^{-1}, \mu_2), (\chi_1^{-1}, \mu_1)).
  \]
\end{definition}

\begin{proposition}
  If $(L, \rho, \casimirsystem)$ is a nonsingular $\slgroup$-link with fractional eigenvalues $\bm \mu$ expressed as the closure of endormorphism $\beta$ of $\braidfh$, then its mirror image $(\overline{L}, \overline{\rho}, \casimirsystem)$ is the closure of $\mathcal{M}(\beta)$.
\end{proposition}

\begin{lemma}
  \label{lemma:mirror-braiding-def}
  Let
  \[
    \sigma : (\chi_1 , \chi_2) \to (\chi_4, \chi_3)
  \]
  be a colored braid generator, write $\tau(x \otimes y) = y \otimes x$ for the flip map, and choose a family of $\U$-module ismorphisms
  \[
    f_{\chi, \mu} : \irrmod{\chi, \mu}^* \iso \irrmod{\chi^{-1}, \mu}.
  \]
  Then, abbreviating $f_i = f_{\chi_i, \mu_i}$,
  \begin{equation}
    \label{eq:mirror-braiding-def}
        (f_4^{-1} \otimes f_3^{-1}) \tau \mainfunctor(\mathcal{M}(\sigma)) \tau (f_1 \otimes f_2)
  \end{equation}
  gives a family of holonomy braidings in the sense of \cref{prop:mirror-holonomy-brading} which satisfy the colored braid relations up to a power of $i$.
\end{lemma}
\begin{proof}
  Because $\mathcal{M}(\sigma)$ is a negative braiding,
  \[
    \mainfunctor(\mathcal{M}(\sigma)) :
    \irrmod{\chi_2^{-1}, \mu_2} \otimes \irrmod{\chi_1^{-1}, \mu_1}
    \to
    \irrmod{\chi_3^{-1}, \mu_3} \otimes \irrmod{\chi_4^{-1}, \mu_4}
  \]
  is a map intertwining $\algbraid^{-1} = (\tau \rmat)^{-1} = \rmat^{-1} \tau$.
  It follows that
  \[
    \tau \mainfunctor(\mathcal{M}(\sigma)) \tau :
    \irrmod{\chi_1^{-1}, \mu_1} \otimes \irrmod{\chi_2^{-1}, \mu_2}
    \to
    \irrmod{\chi_4^{-1}, \mu_4} \otimes \irrmod{\chi_3^{-1}, \mu_3}
  \]
  is a map intertwining $\tau \algbraid^{-1} \tau = \tau \rmat^{-1} = \algbraidd$.
  After composing with the isomorphisms $f_i$, we see that \eqref{eq:mirror-braiding-def} is satisfies the relations of \cref{prop:mirror-holonomy-brading}.

  Because the braidings $\mainfunctor(\sigma)$ satisfy the colored braid relations up to a power of $i$, the braidings \eqref{eq:mirror-braiding-def} must as well.
\end{proof}

\begin{theorem}
  \label{thm:BGPR-mirror}
  There is a functor $\overline{\mainfunctor} : \braidfh \to \overline{\modcat}/\left\langle i \right\rangle$ defined on objects by
  \[
    \overline{\mainfunctor}((\chi_1, \mu_1), \dots, (\chi_n, \mu_n))
    =
    \bigotimes_{j=1}^{n} \irrmod{\chi_j, \mu_j}^{*}
  \]
  and on braid generators $\sigma : (\chi_1 , \chi_2) \to (\chi_4, \chi_3)$ by
  \[
    \overline{\mainfunctor}(\sigma)
    =
    (f_4^{-1} \otimes f_3^{-1}) \tau \mainfunctor(\mathcal{M}(\sigma)) \tau (f_1 \otimes f_2)
  \]
  as in \cref{lemma:mirror-braiding-def}.

  If $(L, \rho, {\bm \mu})$ is a nonsingular $\slgroup$-link with fractional eigenvalues $\bm \mu$ expressed as the closure of an admissible endormorphism $\beta$ of $\braidfh$, then 
  \[
    \overline{\mainfunctor}(L, \rho, {\bm \mu}) \defeq \rentr(\overline{\mainfunctor}(\beta))
  \]
  is an invariant of the framed extended $\slgroup$-link $(L, \rho, {\bm \mu})$ defined up to a fourth root of unity and unchanged by gauge transformations.
  Furthermore,
  \[
  \mainfunctor(L, \rho, \casimirsystem) = \overline{\mainfunctor}(\overline{L}, \overline{\rho}, \casimirsystem).
  \]
\end{theorem}
  Our definition of $\overline{\mainfunctor}$ depends on the family of isomorphisms $f_{\chi, \mu}$, but by the last part of the theorem this choice does not affect the value of $\overline{\mainfunctor}$ on links.
  In addition, the proof of \cref{prop:mainfunctor} requires choosing bases of the $\irrmod{\chi, \mu}$, hence isomorphisms $f_{\chi, \mu}$.
\begin{proof}
  Using \cref{lemma:mirror-braiding-def} all but the last statement can be proved just as in \cref{prop:mainfunctor,thm:BGPR-inv}.
  It remains only to check the final claim about mirror images.

  Consider the extended flip map
   \[
     {\bm \tau} :
     \begin{cases}
       V_1 \otimes \cdots \otimes V_n \to V_n \otimes \cdots \otimes V_1\\
       v_1 \otimes \cdots \otimes v_n \mapsto v_n \otimes \cdots v_1
     \end{cases}
  \]
  which has
  \[
    \mainfunctor(\beta) =  (f_1 \otimes \cdots \otimes f_n) {\bm \tau} \overline{\mainfunctor}(\mathcal{M}(\beta)) {\bm \tau} (f_1^{-1} \otimes f_n^{-1}).
  \]
  By \cref{prop:mirror-modified-dimensions} the renormalized dimensions of $\irrmod{\chi, \mu}$ and $\irrmod{\chi, \mu}^*$ agree, so
  \begin{align*}
    \mainfunctor(L, \rho, \casimirsystem)
    &=
    \rentr(\mainfunctor(\beta))
    \\
    &=
    \rentr \left( (f_1 \otimes \cdots \otimes f_n) {\bm \tau} \overline{\mainfunctor}(\mathcal{M}(\beta)) {\bm \tau} (f_1^{-1} \otimes f_n^{-1}) \right)
    \\
    &=
    \rentr(\overline{\mainfunctor}(\mathcal{M}(\beta))
    \\
    &=
    \overline{\mainfunctor}(\overline{L}, \overline{\rho}, \casimirsystem). \qedhere
  \end{align*}
\end{proof}

\subsection{Internal and external tensor products}
\label{subsec:tensor-products}
The category $\modcat$ of finite-dimensional weight modules is a monoidal category: given two objects $V$ and $W$, we can take their tensor product $V \otimes W$, which becomes a $\U$-module via the coproduct of $\U$.
Because it stays inside $\modcat$ call this an \emph{internal} tensor product and denote it by $\otimes$.

To construct $\doubledfunctor$, we want to consider an \emph{external} tensor product that takes two categories or functors and produces another, larger category or functor.
To distinguish this from the internal tensor product, we denote it by $\boxtimes$.

\begin{definition}
  Let $H_1$ and $H_2$ be Hopf algebras over $\C$, and write $\modc{H_i}$ for the category of (finite-dimensional) $H$-modules.
  The \emph{tensor product} of $\modc{H_1}$ and $\modc{H_2}$ is the category
  \[
    \modc{H_1} \boxtimes \modc{H_2} \defeq \modc{(H_1 \otimes_\C H_2)}
  \]
  of finite-dimensional $H_1 \otimes_\C H_2$-modules.
  If $V_1$ and  $V_2$ are modules for $H_1$ and $H_2$ respectively, then we write
  \[
    V_1 \boxtimes V_2 \defeq V_1 \otimes_\C V_2
  \]
  for the corresponding $(H_1 \otimes_\C H_2)$-module, an object of $\modc{H_1} \boxtimes \modc{H_2}$.

  Now suppose $f_1 : V_1 \to W_1$ is a morphism of $\modc{H_1}$, and similarly for $f_2: V_2 \to W_2$.
  Then we define their external tensor product $f_1 \boxtimes f_2 : V_1 \boxtimes V_2 \to W_1 \boxtimes W_2$ by
  \[
    (f_1 \boxtimes f_2)(v_1 \boxtimes v_2) = f_1(v_1) \boxtimes f_2(v_2)
  \]
  for all $v_1 \in V_1$, $v_2 \in V_2$.
\end{definition}
This is a special case of the Deligne tensor product \cite[\S1.11]{EGNO2015} of categories, which is one reason we use the notation~$\boxtimes$.
Since we only are interested in subcategories of $\modc{\U}$, we do not need the construction in full generality.

\begin{remark}
  \label{rem:internal-external-commute}
  Notice that $\otimes$ and $\boxtimes$ commute%
  \footnote{Formally speaking, this equality should be a natural isomorphism.}
  in the sense that
  \[
    (V_1 \otimes W_1) \boxtimes (V_2 \otimes W_2) = (V_1 \boxtimes V_2) \otimes (W_1 \boxtimes W_2).
  \]
  Here the tensor products $\otimes$ on the left are the internal tensor products of $\modc{H_1}$ and $\modc{H_2}$, respectively, while on the right $\otimes$ is the internal tensor product of $\modc{(H_1 \otimes_\C H_2)}$.

  Depending on the context, both sides of the above equation are useful.
  For example, it is much easier to describe an external tensor product of two maps using the left-hand side.
\end{remark}

\begin{definition}
  Let $H_1$ and $H_2$ be Hopf algebras as above.
  Suppose we have functors
  \[
    \mathcal{F}_i : \braidf \to \modc{H_i}, \quad i = 1, 2
  \]
  from a colored braid groupoid to the category modules of a Hopf algebra.%
  \footnote{$\mainfunctor$ is an example of such a functor, at least after composing with some forgetful functors.}
  The \emph{external} tensor product of $\mathcal{F}_1$ and $\mathcal{F}_2$ is the functor
  \[
    \mathcal{F}_1 \boxtimes \mathcal{F}_2 : \braidf \to \modc{(H_1 \otimes_\C H_2)}
  \]
  defined by
  \[
    (\mathcal{F}_1 \boxtimes \mathcal{F}_2)(\chi_1, \cdots, \chi_n)
    =
    (\mathcal{F}_1(\chi_1, \dots, \chi_n)) \boxtimes (\mathcal{F}_2(\chi_1, \dots, \chi_n))
  \]
  on objects and by
  \[
    (\mathcal{F}_1 \boxtimes \mathcal{F}_2)(\beta)
    =
    \mathcal{F}_1(\beta) \boxtimes \mathcal{F}_2(\beta)
  \]
  on morphisms.%
  \footnote{Since $\braidf$ is a groupoid, it is natural to use the grouplike coproduct $\beta \to \beta \boxtimes \beta$ as we have here.}
\end{definition}

\begin{remark}
  \label{rem:internal-external-commute-ii}
  Assume that the $\mathcal{F}_i$ are monoidal in the sense that for $\chi \in \slgroup$ there are objects $\mathcal{F}_i(\chi)$ of $\modc{H_i}$ with
  \[
    \mathcal{F}_i(\chi_1, \dots, \chi_n) = \mathcal{F}_i(\chi_1) \otimes \cdots \otimes \mathcal{F}_i(\chi_n).
  \]
  For example, $\mainfunctor$ and $\overline{\mainfunctor}$ are of this type.
  As noted in \cref{rem:internal-external-commute}, we can think of the image of $(\chi_1, \dots, \chi_n)$ under their tensor product in two equivalent ways:
  \[
    \bigotimes_{j=1}^{n} \mathcal{F}_1(\chi_j) \boxtimes \mathcal{F}_2(\chi_j)
    =
    (\mathcal{F}_1 \boxtimes \mathcal{F}_2)(\chi_1, \dots, \chi_n)
    =
    \left(\bigotimes_{j=1}^{n} \mathcal{F}_1(\chi_j)\right)
    \boxtimes
    \left(\bigotimes_{j=1}^{n} \mathcal{F}_2(\chi_j)\right).
  \]
  The left-hand side is more useful when trying to understand tensor products in $\modc{H_1} \boxtimes \modc{H_2}$, while the right-hand side is more useful when defining the image of morphisms under $\mathcal{F}_1 \boxtimes \mathcal{F}_2$.
\end{remark}

\subsection{The double \texorpdfstring{$\doubledfunctor$}{T}}
We will define $\doubledfunctor \defeq \mainfunctor \boxtimes \overline{\mainfunctor}$, up to some scalar factors discussed in \S\ref{subsec:scalar-invariant}.
In connection with the definition of modified traces, we emphasize the subcategory $\doubledcat$ of $\modcat \boxtimes \overline{\modcat}$ containing the image of $\doubledfunctor$.

\begin{definition}
  Let $W$ be an object of $\modcat \boxtimes \overline{\modcat}$; in more detail, this means $W$ is a finite dimensional $(\U \otimes_\C \U^{\cop})$-module on which $\Ucenter \otimes_\C \Ucenter$ acts diagonalizably.
  We say $W$ is \emph{locally homogeneous} if for every $z \in \Ucenter$ and $w \in W$,
  \[
    (z \otimes 1) \cdot w = (1 \otimes S(z)) \cdot w.
  \]
  We define $\doubledcat$ to be the subcategory of $\modcat \boxtimes \overline{\modcat}$ of locally homogeneous modules.
  An object $W$ of  $\doubledcat$ has \emph{degree} $\chi \in \slgroupdual$ if
  \[
    (z \otimes 1) \cdot w = (1 \otimes S(z)) \cdot w = \chi(z) w
  \]
  for every $z \in \Ucentersmall$ and $w \in W$.
\end{definition}

\begin{example}
  Let $\chi$ be a $\Ucentersmall$-character and $\mu$ a fractional eigenvalue for $\chi$.
  Then
  \[
    \dirrmod{\chi, \mu} \defeq \irrmod{\chi, \mu} \boxtimes \irrmod{\chi, \mu}^*
  \]
  is a simple projective object of $\doubledcat$ of degree $\chi$.
\end{example}

\begin{definition}
  Let $\chi_i$ be characters with fractional eigenvalues $\mu_i$ related via the braidings as
  \[
    \sigma : (\hat \chi_1, \hat \chi_2) \to (\hat \chi_4, \hat \chi_3),
  \]
  where we abbreviate $\hat \chi_i = (\chi_i, \mu_i)$.
  A \emph{holonomy braiding} for the modules $\dirrmod{\chi_i}$ is a linear map
  \[
    C : \dirrmod{\hat \chi_1} \otimes \dirrmod{\hat \chi_2} \to \dirrmod{\hat \chi_4} \otimes \dirrmod{\hat \chi_3}
  \]
  intertwining $\algbraid \boxtimes \algbraidd$ in the sense that
  \[
    C(x \cdot w) = (\algbraid \boxtimes \algbraidd)(x) \cdot C(w)
  \]
  for every $w \in \dirrmod{\hat \chi_1} \otimes \dirrmod{\hat \chi_2}$ and $x \in (\U \otimes \U^{\cop}) \otimes (\U \otimes \U^{\cop})$.
  By $\algbraid \boxtimes \algbraidd$ we mean the automorphism of $(\U \otimes \U^{\cop}) \otimes (\U \otimes \U^{\cop})$ that acts by $\algbraid$ on tensor factors $1$ and $3$ and by $\algbraidd$ on tensor factors $2$ and $4$.
  (See \cref{rem:internal-external-commute-ii}.)
\end{definition}

\begin{proposition}
  \label{prop:doubled-holonomy-brading}
  Such a holonomy braiding exists and is unique up to an overall scalar.
\end{proposition}
\begin{proof}
  The argument goes exactly as in \cref{prop:holonomy-braiding-exits,prop:mirror-holonomy-brading}, but now for simple $(\U \otimes \U^{\cop})$-modules.
\end{proof}

\begin{example}
  Let $c$ be the holonomy braiding for the modules $\irrmod{\hat \chi}$ from \cref{thm:BGPR-inv}, and let $\overline{c}$ be the holonomy braiding  for the modules $\irrmod{\hat \chi}^*$ from \cref{thm:BGPR-mirror}.
  Then setting 
  \[
    C = c \boxtimes \overline{c}
  \]
  gives a holonomy braiding for the modules $\dirrmod{\hat \chi}$, defined up to a fourth root of unity.
  This holonomy braiding corresponds to the external tensor product of functors
  \[
    \mainfunctor \boxtimes \overline{\mainfunctor}:
    \braidfh
    \to
    \doubledcat \subseteq \modcat \boxtimes \overline{\modcat}.
  \]
\end{example}
The external tensor product $\mainfunctor \boxtimes \overline{\mainfunctor}$ is \emph{almost} our desired functor $\doubledfunctor$, but it has the wrong normalization.
We want to define $\doubledfunctor$ in a slightly different way.

\begin{lemma}
  \label{lemma:invariant-vector}
  By \cref{prop:doubled-holonomy-brading}, there exists a projective functor
  \[
    \widetilde{\doubledfunctor} : \braidfh \to \doubledcat
  \]
  defined on objects by
  \[
    \widetilde{\doubledfunctor}(\hat \chi_1, \dots, \chi_n)
    =
    \dirrmod{\hat \chi_1} \otimes \cdots \otimes \dirrmod{\hat \chi_n},
  \]
  where $\hat \chi_i = (\chi_i, \mu_i)$ is a $\Ucentersmall$-character along with a choice of fractional eigenvalue.
  There is a family
  \[
    v_0(\hat \chi_1, \dots, \hat \chi_n)
    \in
    \widetilde{\doubledfunctor}(\hat \chi_1, \dots, \hat \chi_n) = 
    \bigotimes_{j=1}^{n} \dirrmod{\hat \chi_j}
  \]
  of vectors which are braiding-invariant in the sense that for every admissible braid $\beta : (\hat \chi_1, \dots, \hat \chi_n) \to (\hat \chi_1', \dots, \hat \chi_n')$, 
  \[
    \widetilde{\doubledfunctor}(\beta)(v_0(\hat \chi_1, \dots, \hat \chi_n)) = \alpha v_0(\hat \chi_1', \dots, \hat \chi_n')
  \]
  for some nonzero scalar $\alpha$.
\end{lemma}
The proof of this lemma is a technical computation involving $(\U\otimes_\C\U^{\cop})$-modules, so we delay it to \cref{appendix:proof-of-invariant-vector}.
We expect an explicit computation of the braidings \cite{McPhailSnyderUnpub1} will clarify the proof of this lemma and extend it to all roots of unity.

\begin{corollary}
  The projective functor $\widetilde{\doubledfunctor}$ lifts to a functor 
  \[
    {\doubledfunctor} : \braidfh \to \doubledcat
  \]
  with no scalar ambiguity.
\end{corollary}
\begin{proof}
  We can normalize $\doubledfunctor$ by the condition that
  \[
    {\doubledfunctor}(\beta)(v_0(\hat \chi_1, \dots, \hat \chi_n)) = v_0(\hat \chi_1', \dots, \hat \chi_n').
  \]
  Because the holonomy braiding is unique up to an overall scalar this uniquely characterizes $\doubledfunctor$.
\end{proof}

Once we know that there are appropriate renormalized traces for $\modcat$ we can define the holonomy invariant corresponding to $\doubledfunctor$.

\begin{theorem}
  \label{thm:D-has-rentrs}
  Let $\operatorname{Proj}(\doubledcat)$ be the subcategory of projective $\U \otimes \U^{\cop}$-modules in $\doubledcat$.
  $\operatorname{Proj}(\doubledcat)$ admits a nontrivial modified trace with renormalized dimensions
  \begin{align*}
    \rendim\left( \irrmod{\chi, \mu} \boxtimes \irrmod{\chi, \mu}^* \right) &= \frac{1}{(\mu - \mu^{-1})^2}
    \\
    \rendim\left( \irrmod{\chi, \mu} \boxtimes \irrmod{\chi, -\mu}^* \right) &= -\frac{1}{(\mu - \mu^{-1})^2}
  \end{align*}
  which is compatible with the traces on $\modcat$ and $\overline{\modcat}$.
  That is, let $X$ be a projective object of $\modcat$  and $\overline{X}$ a projective object of $\overline{\modcat}$.
  Then for any endomorphisms $f : X \to X$, $g : \overline{X} \otimes \overline{X}$,
  \[
    \rentr(f \boxtimes g) = \rentr(f) \rentr(g)
  \]
  with $f \boxtimes g$ the obvious endomorphism of $X \boxtimes \overline{X}$.
\end{theorem}
\begin{proof}
  This theorem is easy to prove using the techniques of \cite{Geer2018}.
  For completeness we include the proof in Appendix \ref{appendix:modified-traces}.
\end{proof}

\begin{theorem}
  \label{thm:BGRP-double}
  Let $(L, \rho)$ be a nonsingular framed $\slgroup$-link with fractional eigenvalues~$\casimirsystem$.
  If necessary, gauge transform so that $L$ is the closure of an admissible braid $\beta$ in $\braidfh$.
  Then the renormalized trace
  \[
    \doubledfunctor(L, \rho, \mu)
    \defeq
    \rentr(\doubledfunctor(\beta))
  \]
  of $\beta$ defines an invariant of $(L, \rho, \casimirsystem)$ with no scalar ambiguity and unchanged by gauge transformations.
\end{theorem}
\begin{proof}
  The proof goes exactly as before, with one exception.
  Showing directly that $\doubledfunctor$ is gauge-invariant requires a slight extension of the methods of \cite{Blanchet2018} that were used for $\mainfunctor$.
  We do not include the details and instead prove that $\doubledfunctor$ is gauge-invariant by showing that it agrees with the torsion, which is known to be gauge-invariant.
\end{proof}
We can now state our second main result:
\begin{bigtheorem}
  \label{thm:T-is-torsion}
  Let $(L, \rho, \casimirsystem)$ be a nonsingular $\slgroup$-link with a choice $\casimirsystem$ of fractional eigenvalues.
  If $\tau(L, \rho)$ is the $\slgroup$-twisted Reidemeister torsion of $S^3 \setminus L$, then
  \[
    \doubledfunctor(L, \rho, \casimirsystem) = \tau(L, \rho).
  \]
  Since the sign of $\tau$ is not defined, this equation holds up to sign.
\end{bigtheorem}
The proof is given in \S\ref{sec:proof-of-thm-2}.
As a corollary, we see that (up to sign) $\doubledfunctor(L, \rho, \casimirsystem)$ does not depend on the framing of $L$ or the choice of $\casimirsystem$.

\subsection{The relationship between \texorpdfstring{$\mainfunctor \boxtimes \overline{\mainfunctor}$ and $\doubledfunctor$}{F boxtimes F bar and T}}
\label{subsec:scalar-invariant}

To understand the factor $\scalarfunctor$ appearing in \eqref{eq:scalar-inv-def} and \cref{thm:scalar-invariant}, it helps to recall a similar situation for groups.
Suppose $\tilde t$ is a projective group representation, that is a homomorphism $\tilde t : G \to \gl(V)/\C^\times$, where $\C^\times$ is the multiplicative group of nonzero complex numbers.
A lift of $\tilde t$ is a map $t : G \to \gl(V)$ such that the diagram
\[
  \begin{tikzcd}
    & \gl(V) \arrow[d, "\pi"] \\
    G \arrow[ur, "t"] \arrow[r, swap, "\tilde t"] & \gl(V)/\C^\times
  \end{tikzcd}
\]
commutes.
If $t_1$ and $t_2$ are two lifts, it follows that there is a map $k : G \to \C^\times$ with
\[
  t_1 = k \cdot t_2.
\]
We can think of $k$ as the change in normalization between $t_1$ and $t_2$.

The situation is similar for $\mainfunctor \boxtimes \overline{\mainfunctor}$ and $\doubledfunctor$, which are lifts of the projective functor $\widetilde \doubledfunctor$ of \cref{prop:doubled-holonomy-brading}.
Instead of group representations into $\gl(V)$, we have functors from groupoids into $\vect \C$, and the commutative diagrams are
\[
    \begin{tikzcd}
    & \vect \C \arrow[d, "\pi"] \\
      \braidfh \arrow[ur, "\doubledfunctor"] \arrow[r, swap, "\widetilde{\doubledfunctor}"] & \vect \C / \C^\times
    \end{tikzcd}
    \text{ and }
    \begin{tikzcd}
    & \vect \C / \left\langle i \right\rangle \arrow[d, "\pi"] \\
      \braidfh \arrow[ur, "\mainfunctor \boxtimes \overline{\mainfunctor}"] \arrow[r, swap, "\widetilde{\doubledfunctor}"] & \vect \C / \C^\times
    \end{tikzcd}
\]
There is also a slight complication in that $\mainfunctor \boxtimes \overline{\mainfunctor}$ is only a lift up to powers of $i$, so its image lies in $\C^\times$ modulo the subgroup $\left\langle i \right\rangle$ of powers of $i$.

Thinking of the group $\C^\times/ \left\langle i \right\rangle$ as a groupoid with a single object, we see as before that there is a functor $\scalarfunctor : \braidfh \to \C^\times / \left\langle i \right\rangle$ with
\[
  \doubledfunctor = \scalarfunctor \cdot \mainfunctor \boxtimes \overline{\mainfunctor}.
\]
We think of $\scalarfunctor$ as an anomaly, as it comes from an invertible theory.%
\footnote{That is, it is an invertible element in the monoid of functors $\braidfh \to \vect \C$ with product $\boxtimes$.}

The functor $\scalarfunctor$ gives a holonomy invariant of links (defined up to fourth roots of unity) by evaluating on colored braids, as usual.\footnote{Here we construct a trace by interpreting endomorphisms of $\C^\times / \left\langle i \right\rangle $ as scalars in the obvious way, equivalently by assigning the trivial $\U_\xi$-module $\C$ (modified) quantum dimension $1$.}
We conclude that:
\begin{proposition}
  \label{thm:scalar-invariant}
  There is a holonomy invariant $\scalarfunctor(L, \rho, \casimirsystem)$ such that
  \[
    \doubledfunctor(L, \rho, \casimirsystem)
    =
    \scalarfunctor(L, \rho, \casimirsystem)
    \mainfunctor(L, \rho, \casimirsystem) 
    \overline{\mainfunctor}(L, \rho, \casimirsystem) 
  \]
  up to fourth roots of unity.
\end{proposition}

The proposition completely characterizes $\scalarfunctor$, but gives no information about how to compute it, other than separately computing  $\mainfunctor \boxtimes \overline{\mainfunctor}$ and $\doubledfunctor$ and taking their ratio.
On the other hand, it is fairly difficult to compute $\mainfunctor \boxtimes \overline{\mainfunctor}$ because there is no explicit formula for the braiding or its determinant.
An explicit formula for $\scalarfunctor$ is therefore not particularly useful either.
We hope to clarify this situation in future work.

\begin{remark}
  \label{rem:future-work}
  In \cite{McPhailSnyderUnpub1}, the author and Re\-she\-tikh\-in explicitly compute the matrix coefficients of the braiding for $\mainfunctor$, $\overline{\mainfunctor}$, and $\doubledfunctor$.
  We expect that this computation will yield a different normalization $\mainfunctor'$ of $\mainfunctor$ which will satisfy $\mainfunctor' \boxtimes \overline{\mainfunctor'} = \doubledfunctor$, or at least $\scalarfunctor' \cdot \mainfunctor' \boxtimes \overline{\mainfunctor'} = \doubledfunctor$ for some $\scalarfunctor'$ with an explicit description.
\end{remark}

\begin{remark}
  For any biquandle $X$ there is a groupoid $\braid(X)$ of braids colored by $X$ generalizing $\braidf$.
  We expect that functors $\mathcal{L} : \braid(X) \to A$ for an abelian group $A$ (such as $\scalarfunctor$) can be identified with \emph{biquandle $2$-cocycles} \cite{Kamada2018} on $X$ with values in $A$.

  Such a functor $\mathcal{L}$ is determined by its values on colored braid generators, and a biquandle $2$-cochain is essentially an $A$-valued function on colored braid generators.
  We can identify the colored Reidemeister III relation for functors with the $2$-cocycle condition for cochains.
\end{remark}

\section{Nonabelian torsion is a quantum invariant}
\label{sec:proof-of-thm-2}
In this section we apply the duality of \cref{thm:schur-weyl} to prove \cref{thm:T-is-torsion}.

\subsection{Graded multiplicity spaces for \texorpdfstring{$\doubledcat$}{D}}
To prove \cref{thm:T-is-torsion}, we want to understand the decomposition of
\[
  \doubledfunctor(\hat \chi_1, \dots, \hat \chi_n) = \dirrmod{\hat \chi_1} \otimes \cdots \otimes \dirrmod{\hat \chi_n}
\]
into simple summands.
To guarantee that such a decomposition exists, we need the product character $\chi_1 \cdots \chi_n$ to lie in the semisimple part of $\doubledcat$.

\begin{definition}
  Let $a = (\hat \chi_1, \dots, \hat \chi_n)$  be a tuple of nonsingular extended characters, that is an object of $\braidfh$.
  We say $a$ has \emph{nonsingular total holonomy} if $\chi_1 \cdots \chi_n$ is nonsingular, that is if $\psi(\chi_1 \cdots \chi_n)$ does not have $1$ as an eigenvalue. 
\end{definition}

Notice that by \cref{prop:torsion-computation} this condition is also necessary to compute the torsion using the reduced Burau representation~$\burau$.
By the same proposition we can always modify a colored braid via stabilization moves so that is has nonsingular total holonomy.

From now on, let $a = (\hat \chi_1, \dots, \hat \chi_n)$  be a tuple of nonsingular extended characters with nonsingular total holonomy.
Write $\chi = \chi_1, \dots, \chi_n$ for the product character, and choose a fractional eigenvalue $\mu$ for $\chi$.

By \cref{prop:tensordecomp}, we have
\[
  \bigotimes_{i=1}^{n} \irrmod{\hat \chi}
  \iso
  \irrmod{\chi, \mu}^{\oplus 2^{n-2}} \oplus  \irrmod{\chi, -\mu}^{\oplus 2^{n-2}}
  =
  X_0 \otimes_\C \irrmod{\chi, \mu} \oplus X_1 \otimes_\C \irrmod{\chi, - \mu}
\]
where $X_0$ and $X_1$ are \emph{multiplicity spaces}.
Similarly, we have
\[
  \bigotimes_{i=1}^{n} \irrmod{\hat \chi}^*
  \iso
  \overline{X}_0 \otimes_\C \irrmod{\chi, \mu}^* \oplus \overline{X}_1 \otimes_\C \irrmod{\chi, - \mu}^*
\]
for multiplicity spaces $\overline{X}_0$ and $\overline{X}_1$.

\begin{definition}
  For $\epsilon_i \in \Z/2$, set
  \[
    \dirrmod[\epsilon_1 \epsilon_2]{\chi, \mu}
    \defeq
    \irrmod{\chi, (-1)^{\epsilon_1} \mu} \boxtimes \irrmod{\chi, (-1)^{\epsilon_2} \mu}^*
  \]
  In terms of our previous notation, $\dirrmod[00]{\chi, \mu} = \dirrmod{\chi, \mu}$.
\end{definition}
It is not hard to see (using \cref{prop:semisimplicity}) that every simple object of $\doubledcat$ with character $\chi$ is one of the modules $\dirrmod[\epsilon_1 \epsilon_2]{\chi, \mu}$.

\begin{proposition}
  \label{prop:tensor-decomp-D}
  \begin{enumerate}
    \item 
      The module $\doubledfunctor(\hat \chi_1, \dots, \hat \chi_n)$ decomposes as
      \[
        \doubledfunctor(\hat \chi_1, \dots, \hat \chi_n)
        =
        \bigotimes_{i=1}^n \dirrmod{\hat \chi_i}
        \iso
        \bigoplus_{\epsilon_1, \epsilon_2 \in \Z/2}
        (X_{\epsilon_1} \otimes \overline{X}_{\epsilon_2}) \otimes \dirrmod[\epsilon_1 \epsilon_2]{\chi, \mu}.
      \]
    \item
      The renormalized dimension of $\dirrmod[\epsilon_1 \epsilon_2]{\chi, \mu}$ is 
      \[
        \rendim(\dirrmod[\epsilon_1 \epsilon_2]{\chi, \mu}) = \frac{(-1)^{\epsilon_1 + \epsilon_2}}{(\mu - \mu^{-1})^2}.
      \]
    \item Let $f \in \End_{\doubledcat}(\doubledfunctor(\hat \chi_1, \dots, \hat \chi_n))$ be an endomorphism.
      Then there are linear maps $g_{\epsilon_1 \epsilon_2} \in \End_\C(X_{\epsilon_1} \otimes \overline{X}_{\epsilon_2})$ with
      \[
        f = \bigoplus_{\epsilon_1, \epsilon_2 \in \Z/2} g_{\epsilon_1 \epsilon_2} \otimes \id_{\dirrmod[\epsilon_1\epsilon_2]{\chi, \mu}},
      \]
      and the renormalized trace of $f$ is given by
      \[
        \rentr(f) = \sum_{\epsilon_1, \epsilon_2 \in \Z/2} \frac{(-1)^{\epsilon_1 + \epsilon_2}}{(\mu - \mu^{-1})^2} \tr g_{\epsilon_1 \epsilon_2}.
      \]
  \end{enumerate}
\end{proposition}
\begin{proof}
  \begin{enumerate}
    \item As discussed before the proof, this follows from \cref{prop:tensordecomp}.
    \item The dimensions of the modules $\dirrmod[\epsilon_1 \epsilon_2]{\chi, \mu}$ are computed in \cref{thm:D-has-rentrs}.
    \item Because the $\dirrmod[\epsilon_1 \epsilon_2]{\chi, \mu}$ are simple we can apply Schur's Lemma.
      The second claim about the trace follows from (2).
      \qedhere
  \end{enumerate}
\end{proof}

Part (3) says that the trace of an endomorphism of $\doubledfunctor(\hat \chi_1, \dots, \hat \chi_n)$ can be computed as a $\Z/2$-graded trace on the multiplicity space.

\begin{definition}
  A \emph{super vector space} is a $\Z/2$-graded vector space $Y = Y_0 \oplus Y_1$.
  We call $Y_0$ and $Y_1$ the \emph{even} and \emph{odd} parts, respectively.
  A morphism $f = f_0 \oplus f_1$ of super vector spaces preserves the grading, and we define the \emph{supertrace} by
  \[
    \str f \defeq \tr f_0 - \tr f_1.
  \]
\end{definition}

\begin{example}
  \label{ex:extp}
  If $W$ is an ordinary vector space, then the exterior algebra $\extp W$ becomes a super vector space by setting the image of $W$ in $\extp W$ to be odd.
  A Clifford algebra on $W$ becomes a super vector space in the same way.
\end{example}

\begin{definition}
  The \emph{multiplicity superspace} of $\doubledfunctor(\hat \chi_1, \dots, \hat \chi_n)$ is the super vector space
  \[
    Y = Y(\hat \chi_1, \dots, \hat \chi_n)
  \]
  with even part
  \begin{align*}
    Y_0 &\defeq X_0 \otimes \overline{X}_0 \oplus X_1 \otimes \overline{X}_1
    \\
    \intertext{and odd part}
    Y_1 &\defeq X_0 \otimes \overline{X}_1 \oplus X_1 \otimes \overline{X}_0.
  \end{align*}
\end{definition}

We see that the problem of computing the modified trace $\rentr(\doubledfunctor(\beta))$ of a braid can be reduced to understanding the action of $\doubledfunctor(\beta)$ on the multiplicity superspace~$Y$.
To solve this problem, we identify $Y$ with the exterior algebra $\extp \homol{D_n, \rho}[\lf]$ of the twisted homology, then apply \cref{thm:schur-weyl} to compute the braid action on $Y$ in terms of the Burau representation $\burau$.

We describe a slightly wrong way to do this, then explain how to fix it.
\begin{proof}[Incorrect proof]
As in the previous section, let $a = (\hat \chi_1, \dots, \hat \chi_n)$ be a tuple of nonsingular extended characters with nonsingular total holonomy.

Recall the vectors $v(a) \in \doubledfunctor(a)$ that are invariant under the braiding.
We can show that $\Omega \cdot v(a) = (\mu - \mu^{-1}) v(a)$, so $v(a)$ lies in the even part of $Y$.
In \cref{thm:schur-weyl} we described a subspace $\opspace_n \subseteq \U^{\otimes n}$ that supercommutes with $\Delta(\U)$.
If $x_1, \dots, x_k \in \opspace_n$, then the vector
\[
  x_1 \cdots x_k \cdot v(a)
\]
will lie in $Y_{(-1)^k}$, because the  $x_i$ anticommute with the element $\Delta(\Omega)$ whose action gives the grading.
We can show that the action of $\opspace_n$ is faithful, so that the vectors
\[
  \beta_{j_1}^{\nu_1} \cdots \beta_{j_k}^{\nu_k} \cdot v(a)
\]
give a basis of a subspace of the multiplicity superspace $Y$.
By dimension counting they give a basis of all of $Y$.
Finally, by \cref{thm:schur-weyl} we can identify the $\beta_j^\nu$ with a basis of $\homol{D_n, \rho}[\lf]$, so we can similarly identify $Y$ with the Clifford algebra $\clifford_n$ generated by $\opspace_n$.
\end{proof}

\subsection{Schur-Weyl duality for \texorpdfstring{$\doubledfunctor$}{T}}
The problem with the previous discussion is that $v(a)$ lies not in a $\U$-module, but in a $\U \otimes \U^{\cop}$-module.
To fix this, we need to introduce mirrored versions $\overline{\opspace}_n$ of the operators $\opspace_n$.

\begin{definition}
  For $j = 1, \dots, n$, set  
  \begin{align*}
    \overline{\alpha}_j^1 & \defeq E_j \Omega_j^{-1} K_{j+1} \cdots K_n  \\
    \overline{\alpha}_j^2 & \defeq \F_j \Omega_j^{-1} K_{j+1} \cdots K_n \\
    \intertext{and}
    \overline{\beta}_j^\nu & \defeq \alpha_j^{\nu} - \alpha_{j+1}^{\nu}.
  \end{align*}
  We write $\overline{\opspace}_n$ for the $(\Ucentersmall[\Omega^{-2}])$-span of the $\beta_j^\nu$ and $\overline{\clifford}_n$ for the algebra generated by $\overline{\opspace}_n$.
\end{definition}
\begin{lemma}
  \label{lemma:supercommute-bar}
  $\overline{\clifford}_n$ is the Clifford algebra on $\overline{\opspace}_n$ and it supercommutes with the image of $\U$ under the opposite coproduct:
  \begin{align*}
    \{\Delta^{\op} K, \overline{\beta}^\nu_k\} &= 0 &
    [\Delta^{\op} E, \overline{ \beta }^\nu_k] &= 0 \\
    [\Delta^{\op} \F, \overline{ \beta }^\nu_k] &= 0 &
    \{\Delta^{\op} \Omega, \overline{ \beta }^\nu_k\}  &=  0 
  \end{align*}
\end{lemma}
\begin{proof}
  This follows from the relations
  \begin{align*}
    \{\overline{ \alpha }_j^1, \overline{ \alpha }_k^1\} & = 2 \delta_{jk} E_j^2 \Omega_j^{-2} K_{j+1}^2 \cdots K_{n}^2 \\
    \{\overline{ \alpha }_j^2, \overline{ \alpha }_k^2\} & = 2 \delta_{jk} K_j^2 F_j^2 \Omega_j^{-2} K_{j+1}^2 \cdots K_{n}^2\\
    \{\overline{ \alpha }_j^1, \overline{ \alpha }_k^2\} & = 2i \delta_{jk} (1 - K_j^{-2}) \Omega_j^{-2} K_{j+1}^2 \cdots K_{n}^2 \qedhere
  \end{align*}
  and then from the same argument as in the proof of \cref{thm:schur-weyl}.
\end{proof}

\begin{lemma}
  The braiding automorphism $\algbraidd$ for $\U^{\cop}$ acts by
  \begin{align*}
    \algbraidd(\overline{\alpha}_1^1) &= \overline{\alpha}_2^1 \\
    \algbraidd(\overline{\alpha}_2^1) &= K_2^{-2} \overline{\alpha}_1^1 + (1 - K_2^{-2}) \overline{\alpha}_2^1 + K_2^{-2} E_2^2 (\overline{\alpha}_1^2 - \overline{\alpha}_2^2) \\
    \algbraidd(\overline{\alpha}_1^2) &= (K_1^2 - 1) \overline{\alpha}_1^1 + K_1^2 \overline{\alpha}_2^2 + K_1^2 F_1^2 (\overline{\alpha}_1^1 - \overline{\alpha}_2^1) \\
    \algbraidd(\overline{\alpha}_2^2) &= \overline{\alpha}_1^2
  \end{align*}
  so the matrix of $\algbraidd_{i,i+1}$ acting on $\overline{\opspace}_n$ is
  \begin{equation}
    \label{eq:quantum-burau-action-opinv}
    I_{2(i-2)} \oplus
    \begin{bmatrix}
      1 & 0 & K_i^{2} & - K_i^2 F_i^2 \\
      0 & 1 & 0 & 1 \\
        & & - K_{i}^{2} & K_i^2 F_i^2 \\
        & & - K_{i+1}^{-2} E_{i+1}^2 & - K_{i+1}^{-2} \\
        & & 1 & 0 & 1 & 0 \\
        & & K_{i+1}^{-2} E_{i+1}^2  & K_{i+1}^{-2}  & 0 & 1
    \end{bmatrix}
    \oplus I_{2(n- 1) - 2(i + 1)}
  \end{equation}
\end{lemma}

Recall the functor $\algfunctor$ (\cref{def:alg-functor}) constructed from the automorphism $\algbraid$.
The mirror version $\overline{\algfunctor}$ is constructed in the same way from $\algbraidd$, and we have shown that it also satisfies a Schur-Weyl duality:
\begin{lemma}
  \label{lemma:schur-weyl-mirror}
  Let $\beta : \rho \to \rho'$ be an admissible braid in $\slgroup$, and let $\rho$ and $\rho'$ correspond to characters $\chi_i$ and $\chi_i'$, respectively.
  Define linear maps $\overline{\phi}_\rho : \homol{D_n; \rho} \to \overline{\opspace}_n$ by
  \[
    \overline{\phi}_{\rho}(v_j^{\nu}) = \overline{\beta}_{j}^{\nu}.
  \]
  Then the diagram commutes:
  \[
    \begin{tikzcd}
      \homol{D_n; \rho}[\lf] \arrow[r, "\burau(\beta)"] \arrow[d, "\phi_\rho"] & \homol{D_n; \rho}[\lf] \arrow[d, "\phi_{\rho'}"] \\
      \overline{\opspace}_n/\ker(\chi_1^{-1} \otimes \cdots \otimes \chi_n^{-1}) \arrow[r, "\overline{\algfunctor}(\beta)"] & \overline{\opspace}_n/\ker((\chi_1')^{-1} \otimes \cdots \otimes (\chi_n')^{-1})
    \end{tikzcd}
  \]
\end{lemma}
\begin{proof}
  This is proved exactly the same way as \cref{thm:schur-weyl} but using the computations above.
  A key observation is that the matrix \eqref{eq:quantum-burau-action-opinv} is sent to the matrix \eqref{eq:burau-coords-nice} under the map $(\chi_1')^{-1} \otimes \cdots (\chi_n')^{-1} = (\chi_1' \circ S) \otimes \cdots \otimes (\chi_n' \circ S)$.
\end{proof}

Now that we have proved Schur-Weyl duality twice (the first time was \cref{thm:schur-weyl}) we are prepared to prove it one final time.
\begin{definition}
  Set
  \begin{align*}
    \gamma^\nu_j &\defeq \alpha^\nu_j \boxtimes 1 + \Delta K \boxtimes \overline{\alpha}^\nu_j \\
    \theta^\nu_j &\defeq \beta^\nu_j \boxtimes 1 + \Delta K \boxtimes \overline{\beta}^\nu_j = \gamma_j^{\nu} - \gamma_{j+1}^{\nu}
  \end{align*}

  Write $\dopspace n$ for the $(\Ucentersmall[\Omega^{-2}] \otimes \Ucentersmall[\Omega^{-2}])^{\otimes n} $-span of the $\theta_j^\nu$ and $\dclifford n$ for the algebra generated by $\dopspace n$.
  As before, $\dclifford n$ is a Clifford algebra.
\end{definition}
\begin{remark}
  Here by $X \boxtimes Y$ we mean an element of $\U \otimes \U^{\cop}$, the algebra underlying $\doubledcat$, and we write expressions like
  \[
    X \boxtimes Y \otimes Z \boxtimes W
  \]
  for operators in $(\U \otimes \U^{\cop})^{\otimes 2}$ to emphasize the different tensor factors.  
  The multiple tensor products here can be hard to parse, so we give an example.
  If
  \[
    \alpha = X_1 \otimes X_2, \ \ \overline{\alpha} = Y_1 \otimes Y_2,
  \]
  by $\alpha \boxtimes 1 + \Delta K \boxtimes \overline{\alpha}$ we mean
  \[
    (X_1 \boxtimes 1) \otimes (X_2 \boxtimes 1) + (K \boxtimes Y_1) \otimes (K \boxtimes Y_2).
  \]
\end{remark}

\begin{theorem}
  \label{thm:schur-weyl-double}
  Let $\beta : \rho \to \rho'$ be an admissible braid in $\slgroup$, and let $\rho$ and $\rho'$ correspond to characters $\chi_i$ and $\chi_i'$, respectively.
  Define linear maps $\Phi_\rho : \homol{D_n; \rho}[\lf] \to \dopspace n$ by
  \[
    \Phi_{\rho}(v_j^{\nu}) = \theta_{j}^{\nu}.
  \]
  Then the diagram
  \[
    \begin{tikzcd}
      \homol{D_n; \rho}[\lf] \arrow[r, "\burau(\beta)"] \arrow[d, "\Phi_\rho"] & \homol{D_n; \rho}[\lf] \arrow[d, "\Phi_{\rho'}"] \\
      \dopspace n/\rho \arrow[r, "\algfunctor \boxtimes \overline{\algfunctor}(\beta)"] & \dopspace n/\rho'
    \end{tikzcd}
  \]
  commutes, where
  \[
    \dopspace n /\rho \defeq \dopspace n / \ker( (\chi_1 \boxtimes \chi_1^{-1}) \otimes \cdots \otimes (\chi_n \boxtimes \chi_n^{-1}))
  \]
  and similarly for $\rho'$.
\end{theorem}
\begin{proof}
  It suffices to check braid generators.
  By the definition of the holonomy braiding, we have
  \begin{align*}
    \doubledfunctor(\sigma_i)( \theta_k^\nu \cdot v_0(a) ) &= (\algbraid_{i,i+1} \boxtimes \algbraidd_{i,i+1})(\theta_k^\nu) \cdot \doubledfunctor(\sigma_i)(v_0(a)) \\
                                                           &= (\algbraid_{i,i+1} \boxtimes \algbraidd_{i,i+1})(\theta_k^\nu) \cdot v_0(a')
  \end{align*}
  where we have used the invariance of the family $v_0(-)$ in the second equality.
  Since
  \begin{align*}
    (\algbraid_{i,i+1} \boxtimes \algbraidd_{i,i+1}) (\theta_k^\nu) &= \algbraid_{i,i+1}(\beta_k^\nu) \boxtimes \algbraid_{i,i+1}^*(1) + \algbraid_{i,i+1}(\Delta K) \boxtimes \algbraidd_{i,i+1}(\overline{\beta}_k^\nu) \\
                                                                    &= \algbraid_{i,i+1}(\beta_k^\nu) \boxtimes 1 + \Delta K \boxtimes \algbraidd_{i,i+1}(\overline{\beta}_k^\nu)
  \end{align*}
  the result follows from \cref{thm:schur-weyl,lemma:schur-weyl-mirror}.
\end{proof}

\subsection{Schur-Weyl duality for \texorpdfstring{$\doubledcat$}{D}}
We still are not quite ready to use \cref{thm:schur-weyl} to compute the multiplicity superspaces.

The problem is that $\dclifford n$ does not quite commute with the superalgebra $\U \otimes \U^{\cop}$, regardless of whether we take the ordinary or super tensor product.
Fortunately, this is not necessary, because we do not need to compute the detailed multiplicity spaces $X_{\epsilon_1} \otimes \overline{X}_{\epsilon_2}$, only the spaces $Y_0$ and $Y_1$.
To accomplish this it suffices to consider a weaker sort of supercommutativity.

Let $W$ be a $\U \otimes \U^{\cop}$-module.
It becomes a $\U$-module via the action ($w \in W$)
\[
  \begin{aligned}
    K \cdot w &= ( K \boxtimes K) \cdot w,
    \\
    E \cdot w &= (E \boxtimes K + 1 \boxtimes E) \cdot w,
    \\
    F \cdot w &= (F \boxtimes 1 + K^{-1} \boxtimes F) \cdot w,
  \end{aligned}
\]
that is by embedding $\U$ into $\U \otimes \U^{\cop}$ via the coproduct.
\begin{proposition}
  As a $\U$-module, $\dirrmod[\epsilon_1 \epsilon_2]{\chi, \mu} \iso P_{\epsilon_1 + \epsilon_2}$, where $P_0$ and $P_1$ are the modules of \cref{def:unit-cover,def:parity-rep}.
\end{proposition}
\begin{proof}
  Apply \cref{prop:plus-and-minus-covers}.
\end{proof}
In particular, to compute the spaces $Y_\epsilon(\hat \chi_1, \dots, \hat \chi_n)$ it suffices to understand them as $\U$-modules instead of $\U \otimes \U^{\cop}$-modules.
\begin{lemma}
  \label{lemma:mult-space-calc}
  Let $a = (\hat \chi_1, \cdots, \hat \chi_n)$ be a tuple of nonsingular extended $\Ucentersmall$-characters with nonsingular total holonomy.
  Write $\pi_a$ for the structure map
  \[
    \pi_a : (\U \boxtimes \U)^{\otimes n}  \to \End_{\C}(\doubledfunctor(a))
  \] 
  Then
  \begin{enumerate}
    \item $\pi_a(\dclifford n)$ is an exterior algebra on $2(n-1)$ generators,
    \item $\dclifford n$ acts faithfully on $\doubledfunctor(a)$, and
    \item thinking of $\doubledfunctor(a)$ as a $\U$-module, $\pi_a(\dclifford n)$ super-commutes with $\U$.
  \end{enumerate}
\end{lemma}
\begin{proof}
  (1) We show that the anticommutators
  \[
    \pi_a(\{ \gamma_j^\mu, \gamma_k^\nu \})
  \]
  vanish, so that the image is an exterior algebra on the $2n$ independent generators $\pi_a(\gamma_k^\nu)$, $k = 1, \dots, n$, $\nu = 1, 2$.
  Since $\dclifford n$ is generated by the $\theta_k^{\nu} = \gamma_k^{\nu} - \gamma_{k+1}^{\nu}$ we get the desired result.

  Observe that, because $\{\alpha_k^\nu, \Delta K\} = 0$,
  \begin{align*}
    \{ \gamma_j^\mu, \gamma_k^\nu \} &= \{ \alpha_j^\mu, \alpha_k^\nu \} \boxtimes 1 + \Delta K^2 \boxtimes \{ \overline{\alpha}_j^\mu, \overline{\alpha}_k^\nu \}
  \end{align*}
  By using the anticommutator computations of \cref{lemma:supercommute-bar,lemma:Cn-is-clifford} we can show directly that these vanish.
  For example, the above expression vanishes unless $j = k$.
  We give the case $\mu = \nu = 1$ in detail; the remaining others follow similarly.

  Observe that
  \begin{align*}
    &\{ \alpha_j^1, \alpha_j^1 \} \boxtimes 1 + \Delta K^2 \boxtimes \{ \overline{\alpha}_j^1, \overline{\alpha}_j^1 \} \\
    &=  2 K_1^2 \cdots K_{j-1}^2 E_j^2 \Omega_j^{-2} \boxtimes 1 + 2 K_1 \cdots K_n^2 \boxtimes E_j^2 \Omega_j^{-2} K_{j+1}^2 \cdots K_n^2
  \end{align*}
  Write $\hat \chi_j(K^2) = \kappa_j$, $\hat \chi_j(E^2) = \epsilon_j$, $\hat \chi_j(\Omega^2) = \omega_j^2$, so that
  \begin{align*}
    \pi_a(K_j^2 \boxtimes 1) &= \kappa_j & \pi_a(1 \boxtimes K_j^2) &= \kappa_j^{-1} \\
    \pi_a(E_j^2 \boxtimes 1) &= \epsilon_j & \pi_a(1 \boxtimes E_j^2) &= - \epsilon_j \kappa_j^{-1} \\
    \pi_a(\Omega_j^2 \boxtimes 1) &= \omega_j^2 & \pi_a(1 \boxtimes \Omega_j^2) &= \omega_j^2 
  \end{align*}
  using the fact that the representations in the second half of the $\boxtimes$ product (corresponding to $\overline{\mainfunctor}$) use the inverse characters.
  Hence
  \begin{align*}
    & \pi_a( 2 K_1^2 \cdots K_{j-1}^2 E_j^2 \Omega_j^{-2} \boxtimes 1 + 2 K_1 \cdots K_n^2 \boxtimes E_j^2 \Omega_j^{-2} K_{j+1}^2 \cdots K_n^2) \\
    &= \frac{2}{\omega_j^2} \left( \kappa_1 \cdots \kappa_{j-1} \epsilon_j + \kappa_1 \cdots \kappa_n (-\epsilon_j \kappa_j^{-1}) \kappa_{j+1}^{-1} \cdots \kappa_n^{-1} \right) = 0
  \end{align*}
  as claimed.

  (2) It is enough to show that the operators $\pi_a(\gamma_k^\nu)$ all act independently.
  Since up to a scalar $\gamma_k^1, \gamma_k^2$ only act on the $k$th $\otimes$-factor of the product
  \[
    \doubledfunctor(a) = \bigotimes_{j = 1}^n \irrmod{\hat \chi_j} \boxtimes \irrmod{\hat \chi_j}^*
  \]
  it is enough to check that $\gamma_k^1$ and $\gamma_k^2$ act independently.
  It is not hard to compute explicitly that the vectors
  \[
    \pi_a(\gamma_k^1) \cdot v_0(\hat \chi_1, \dots, \hat \chi_n) \text{ and } \pi_a(\gamma_k^2) \cdot v_0(\hat \chi_1, \dots, \hat \chi_n)
  \]
  are independent, where $v_0$ is the invariant vector of Lemma \ref{lemma:invariant-vector}, and (2) follows.

  (3) We can check directly that
  \begin{align*}
    &[ \Delta E \boxtimes \Delta^{\op} K + 1 \boxtimes \Delta^{\op} E, \theta_k^\nu] \\
    &= [ \Delta E \boxtimes \Delta^{\op} K + 1 \boxtimes \Delta^{\op} E, \beta_k^\nu \boxtimes 1 + \Delta K \boxtimes \overline{\beta}_k^\nu] \\
    &=  [ \Delta E, \beta_k^\nu ] \boxtimes 1 + \Delta K \boxtimes [ \Delta^{\op} E, \overline{\beta}_k^\nu ] \\
    &= 0
  \end{align*}
  The other generators $K, \F$ follow similarly.
\end{proof}

\begin{corollary}
  \label{cor:mult-space-calc}
  The $\Z/2$-graded multiplicity space $Y(a)$ of $\doubledfunctor(a)$ is isomorphic as a vector space to $\dclifford n$.
\end{corollary}
\begin{proof}
  This is simply an application of (a super version) of the double centralizer theorem.
  By (3) of \cref{lemma:mult-space-calc},  $\pi_a$ induces an inclusion $\dclifford n \to Y(a)$, and by (2) this inclusion is injective.
  But both spaces have dimension $2^{2n-2}$ over $\C$, so it is an isomorphism.
\end{proof}

We can describe a basis for $Y(a)$ explicitly.
Let $v_0(a)$ be the invariant vector of \cref{lemma:invariant-vector}.
By \cref{lemma:mult-space-calc} and its corollary, a basis for the multiplicity space $Y(a)$ consists of the vectors
\[
  (\theta_{k_1^1}^1 \cdots \theta_{k_{s_1}^1}^1 \theta_{k_1^2}^2 \cdots \theta_{k_{s_2}^2}^2) \cdot v_0(a)
\]
where $1 < k_\nu \cdots < k_{s_\nu} \le n -1$ and $s_\nu = 0, \dots, n-1$ for $\nu = 1, 2$.

By \cref{cor:mult-space-calc,thm:schur-weyl-double}, we get
\begin{corollary}[Schur-Weyl duality for modules]
  \label{cor:schur-weyl-modules}
  Let $a = (\chi_1, \dots, \chi_n)$ be a tuple of nonsingular extended characters with nonsingular total holonomy, and let $\beta : a \to a'$ be an admissible braid, where $a' = (\chi_1', \dots, \chi_n')$.
  Write $\rho, \rho'$ for the corresponding representations of $\pi_1(D_n)$.

  The linear maps $\Phi_a$ of \cref{thm:schur-weyl-double} extend to super vector space isomorphisms
  \[
    \extp \Phi_\rho : \extp \homol{D_n;\rho}[\lf] \to \dclifford n / \ker \rho \to Y(a)
  \]
  where we consider the exterior algebra as a super vector space in the usual way,%
  \footnote{As in \cref{ex:extp}, $v_1 \wedge \cdots \wedge v_k$ lies in degree $k \!\! \mod 2$.}
  and this identification is compatible with the braiding in the sense that the diagram
  \[
    \begin{tikzcd}
      \extp \homol{D_n;\rho}[\lf] \arrow[r, "\extp \burau(\beta)"] \arrow[d, "\extp \Phi_\rho"] & \extp \homol{D_n;\rho'}[\lf] \arrow[d, "\extp \Phi_{\rho'}"] \\
      Y(a) \arrow[r, "\doubledfunctor(\beta)"] & Y(a')
    \end{tikzcd}
  \]
  commutes.
\end{corollary}

\subsection{Conclusion of the proof}
We can now prove \cref{thm:T-is-torsion}.
First, recall a fact about exterior powers:
\begin{proposition}
  \label{prop:str-to-det}
  Let $W$ be a vector space of dimension $N$ and $A : W \to W$ a linear map.
  Write $\extp A$ for the induced map $\extp W \to \extp W$ on the exterior algebra of $W$.
  Then
  \[
    \str\left( \extp A \right) = (-1)^N \det(1 - A).
  \]
\end{proposition}
\begin{proof}
  Recall that
  \[
    \det (\lambda - A) = \sum_{k = 0}^N \lambda^{N-k}(-1)^{N-k} \tr \left( \extp^k A \right)
  \]
  so in particular
  \[
    \det(1 - A) = (-1)^N \sum_{k=0}^N (-1)^k \tr \left( \extp^kA \right) = (-1)^N \operatorname{str} \left( \extp{}A \right).
  \]
\end{proof}

\begin{proof}[Proof of \cref{thm:T-is-torsion}]
  Let $(L, \rho, \casimirsystem)$ be a nonsingular $\slgroup$-link with a choice $\casimirsystem$ of fractional eigenvalues.
  By Proposition \ref{prop:torsion-computation} we can represent $(L, \rho)$ as the closure of an $\slgroup$-braid $\beta_0$ with nonsingular total holonomy, and by Proposition \ref{prop:admissible-exist} we can pull $\beta_0$ back to a nonsingular braid $\beta : a \to a$ in $\braidfh$, possibly after a gauge transformation.
  By definition, 
  \[
    T(L, \rho, \casimirsystem) = \rentr(\doubledfunctor(\beta)).
  \]

  Choose a fractional eigenvalue $\mu$ for the total holonomy $\psi(a)$ of $a$.
  By \cref{cor:schur-weyl-modules} the intertwiner $\doubledfunctor(\beta)$ factors through the multiplicity superspace $Y(a)$ of $\doubledfunctor(a)$ as $\extp \burau(\beta)$, so by \cref{prop:tensor-decomp-D,prop:str-to-det} we have
  \[
    \rentr(\doubledfunctor(\beta)) = \frac{\str \left( \extp \burau(\beta) \right) } {(\mu - \mu^{-1})^2} = \frac{ \det(1 - \burau(\beta)) }{(\mu - \mu^{-1})^2}.
  \]
  The total holonomy $\psi(a)$ has eigenvalues $\mu^2$ and $\mu^{-2}$, so
  \[
    \det(1 - \psi(a)) = (1- \mu^2)(1-\mu^{-2}) = -\mu^2 - \mu^{-2} + 2 = -(\mu - \mu^{-1})^2.
  \]
  Therefore
  \[
    \rentr(\doubledfunctor(\beta)) = \frac{ \det(1 - \redburau(\beta)) }{(\mu - \mu^{-1})^2} = -\frac{ \det(1 - \redburau(\beta)) }{\det(1 - \psi(a))} = -\tau(L, \rho)
  \]
  by Proposition \ref{prop:torsion-computation}.
  Since $\tau(L, \rho)$ is only defined up to sign, $\tau(L, \rho) = \rentr(\doubledfunctor(\beta)) = \doubledfunctor(L, \rho, \casimirsystem)$ as claimed.
\end{proof}

\appendix

\section{Unit-graded representations}
\label{appendix:unitmods}
In \S\ref{sec:proof-of-thm-2} and \cref{appendix:modified-traces} we consider certain modules corresponding to singular characters.
This section gives some of their properties.
\begin{example}
  As usual for the category of representations of a Hopf algebra, the tensor unit $\tsunit$ is the vector space $\C$, with the action of $\U$ given by the counit:
  \[
    \epsilon(K) = 1, \ \ \epsilon(E) = \epsilon(F) = 0.
  \]
  This module is irreducible, with $\Ucentersmall$-character $\epsilon = \epsilon|_{\Ucentersmall}$.
  The corresponding element of $\slgroup^*$ is the identity element, which we expect, since for a module $V$ with character $\chi$, the character of $\tsunit \otimes V \iso V$ should be $\epsilon \cdot \chi = \chi$.
\end{example}
Because $\modcat$ is not semisimple, $\tsunit$ is not a projective $\U$-module.
It is not hard to describe its projective cover, however.
\begin{definition}
  \label{def:unit-cover}
  $P_0$ is the $4$-dimensional $\U$-module described by
  \[
    K \mapsto \begin{pmatrix}
      1 \\
      & -1 \\
      & & -1 \\
      & & & 1
    \end{pmatrix}, \ \
    E \mapsto \begin{pmatrix}
      0 & 0 & 0 & 0 \\
      1 & 0 & 0 & 0 \\
      0 & 0 & 0 & 0\\
      0 & 0 & 1 & 0
    \end{pmatrix}, \ \
    F \mapsto \begin{pmatrix}
      0 & 0 & 0 & 0 \\
      0 & 0 & 0 & 0 \\
      -i & 0 & 0 & 0\\
      0 & -i & 0 & 0
    \end{pmatrix}
  \]
  $P_0$ has $\Ucentersmall$-character the counit $\epsilon$, thought of as an algebra homomorphism $\Ucentersmall \to \C$.
  Note that $P_0$ is indecomposable, but not irreducible, and that $\Omega$ does not act diagonalizably on $P_0$.
\end{definition}

The action of the generators of $\U$ on $P_0$ is best described diagrammatrically.
There is a basis $x, y_1, y_2, z$ of $P_0$ with
\[
  K \cdot x = x, \ \ K \cdot z = z, \ \ K \cdot y_i = - y_i
\]
and with the action of $E$ and $F$ given by the diagram 
\[
  \begin{tikzcd}
    & x \arrow[dl, swap, "E"] \arrow[dr, "F"] \\
    y_1 \arrow[dr, swap, "F"] & & y_2 \arrow[dl, "E"] \\
                        & z
  \end{tikzcd}
\]
where missing arrows mean action by $0$, e.g.~$E \cdot y_1 = 0$.

\begin{proposition}
  \label{prop:unit-cover-properties}
  \begin{enumerate}
    \item For any nonsingular $\Ucentersmall$-character $\chi$ with fractional eigenvalue $\mu$,
      \[
        \irrmod{\chi, \mu} \otimes \irrmod{\chi, \mu}^* \iso P_0.
      \]
    \item $P_0$ is the projective cover (dually, the injective hull) of the tensor unit $\tsunit$.
    \item The vector spaces $\hom_\U(P_0, \tsunit)$ and $\hom_\U(\tsunit, P_0)$ are one-dimensional.
      The isomorphism in (1) takes $\evbar V$ to a basis of $\hom_\U(P_0, \tsunit)$ and $\coev V$ to a basis of $\hom_\U(\tsunit, P_0)$.\footnote{Here $\evbar{}$ and $\coev{}$ are part of the pivotal structure on the category of $\U$-weight modules.
      See Proposition \ref{prop:pivotal-cat} for details.}
  \end{enumerate}
\end{proposition}
\begin{proof}
  Recall the basis $\ket 0, \ket 1$ of $\irrmod{\chi, \mu}$ given in \cref{ex:simple-modules}, and let $\bra 0, \bra 1$ be the dual basis of $\irrmod{\chi, \mu}^*$.
  We consider the case $\epsilon \ne 0$; the other cases follow by similar reasoning.
  Let $\alpha$ be a square root $\alpha^2 = \kappa$ of $\kappa$, and set $\omega = \mu - \mu^{-1}$.
  Then the generators act on the dual space by
  \begin{align*}
    \pi^*(K)
    &=
    \begin{pmatrix}
      \alpha^{-1} & 0 \\
      0 & -\alpha^{-1}
    \end{pmatrix},
    \quad
    \pi^*(E)
    =
    \begin{pmatrix}
      0 & -\alpha^{-1} \\
      \epsilon \alpha^{-1} & 0
    \end{pmatrix},
    \\
    \pi^*(F)
    &=
    \begin{pmatrix}
      0 & i\alpha(\omega - \alpha + \alpha^{-1}) /\epsilon \\
      -i\alpha(\omega + \alpha - \alpha^{-1}) & 0
    \end{pmatrix}.
  \end{align*}
  
  Define a linear map $f : P_0 \to \irrmod{\chi, \mu} \otimes \irrmod{\chi, \mu}^*$ by
  \begin{align*}
    f(x)
    &=
    \ket 0  \bra 0 - \ket 1  \bra 1 
    \\
    f(y_1)
    &=
    2 \alpha^{-1}
    \left(
      \epsilon \ket 0  \bra 1 + \ket 1  \bra 0
    \right)
    \\
    f(y_2)
    &=
    2i
    \left(
      (\alpha - \alpha^{-1} + \omega) \ket 0  \bra 1
      +\epsilon^{-1} (\alpha - \alpha^{-1} - \omega) \ket 1  \bra 0
    \right)
    \\
    f(z)
    &=
    -\frac{4i \omega}{\alpha}
    \left(
      \ket 0  \bra 0 + \ket 1  \bra 1
    \right)
  \end{align*}
  where we write $\ket j \bra k = \ket j \otimes \bra k$.
  It is not hard to check that $f$ is an isomorphism of $\U$-modules, which proves (1).

  For (2), first observe that $\irrmod{\chi, \mu}$ is projective by \cref{prop:semisimplicity}.
  Because $\modcat$ is pivotal, $\irrmod{\chi, \mu}^*$ is also projective, so the tensor product $\irrmod{\chi, \mu} \otimes \irrmod{\chi, \mu}^* \iso P_0$ is projective (and injective) as well.

  To show that $P_0$ is the injective hull of $\tsunit$, we must show that the submodule $N$ spanned by $z$ is \emph{essential}, i.e.~that for any other submodule $M$ of $P_0$, $N \cap M = 0$ implies $M = 0$.
  This is clear from the diagram describing $P_0$: if $M$ contains $x$, $y_1$, or $y_2$, then it must contain $z$, so it has a nonzero intersection with $N$.

  It remains to prove (3).
  It is not hard to see that $P_0$ is indecomposable, so by (2) it satisfies the hypotheses of \cite[Lemma 5.1]{Geer2018}.
  It follows that the spaces $\hom_\U(P_0, \tsunit)$ and $\hom_\U(\tsunit, P_0)$ are $1$-dimensional, as $\tsunit$ is both the head and the tail of $P_0$.
  By (1),
  \[
    \evbar V(1) = \ket 0 \bra 0 - \ket 1 \bra 1 = f(x)
  \]
  and similarly
  \[
    \coev V f(x) = \coev V f(y_1) = \coev V f(y_2) = 0 \text{ and } \coev V f(z) = \frac{-4i}{\alpha} \ne 0,
  \]
  so $\evbar V$ and $\coev V$ give bases as claimed.
\end{proof}

Finally, we consider two more modules with character $\epsilon$.
\begin{definition}
  \label{def:parity-rep}
  The \emph{parity module} $\parmod$ is the $1$-dimensional $\U$-representation with action
  \[
    K \mapsto -1, \ \ E \mapsto 0, \ \ F \mapsto 0
  \]
  Its projective cover is
  \[
    P_1 \defeq \parmod \otimes P_0
  \]
\end{definition}
It is easy to describe $P_1$: the action of $\U$ is the same, except that the sign of $K$ is switched.
The name ``parity module'' is because of the following proposition:
\begin{proposition}
  \label{prop:plus-and-minus-covers}
  For any admissible $\Ucentersmall$-character $\chi$ with fractional eigenvalue $\mu$,
  \[
    \parmod \otimes \irrmod{\chi, \mu} \iso \irrmod{\chi, \mu} \otimes \parmod \iso \irrmod{\chi, - \mu} .
  \]
  Similarly, we have
  \[
    \irrmod{\chi, \mu} \otimes \irrmod{\chi, - \mu}^* \iso P_1.
  \]
\end{proposition}
\begin{proof}
 $\irrmod{\chi, \mu} \otimes \irrmod{\chi, - \mu}^* \iso \irrmod{\chi, \mu} \otimes \irrmod{\chi, \mu}^* \otimes \Pi \iso P_0 \otimes \Pi \iso P_1.$
\end{proof}

\section{Construction of modified traces}
\label{appendix:modified-traces}
We apply the methods of Geer, Kujawa, and Patureau-Mirand \cite{Geer2018} to construct the modified traces of \S\ref{subsec:modified-trace}.
It is simple to derive our results from their general framework, but we include the details for logical completeness.
The approach of \cite{Geer2018} is rather abstract, and few concrete examples have appeared in the literature, so this appendix may also be helpful as a guide to applying their techniques to quantum topology.

In this appendix we frequently state results for a pivotal $\C$-linear category $\cat$, by which mean a pivotal category whose hom spaces are vector spaces over $\C$ and whose tensor product is $\C$-bilinear.
$\modcat$, $\overline{\modcat}$, and $\doubledcat$ (or more generally the category of representations of a pivotal Hopf $\C$-algebra) are all examples of such categories.

More specific results of \cite{Geer2018} place extra conditions on $\cat$ (local finiteness) and on certain distinguished objects (absolute decomposability, end-nilpotence, etc.) which are satisfied for finite-dimensional representations of an algebra over an algebraically closed field, perhaps with some diagonalizability assumptions.
All our examples satisfy these hypotheses.

\subsection{Projective objects, ideals, and traces}
\begin{definition}
  Let $\cat$ be a category.
  We say an object $P$ of $\modcat$ is \emph{projective} if for any epimorphism $p : X \to Y$ and any map $f : P \to Y$, there is a lift $g : P \to X$ such that the diagram commutes:
  \[
    \begin{tikzcd}
      & X \arrow[d, "p"] \\
      P \arrow[ur, dashed, "g"] \arrow[r, "f"] & Y
    \end{tikzcd}
  \]
  We say $I$ is \emph{injective} if $I$ is a projective object in $\cat^{\op}$, i.e.~ if $I$ satisfies the opposite of the above diagram.
  We write $\proj(\cat)$ for the class of projective objects of $\cat$.
\end{definition}

\begin{definition}
  Let $\cat$ be a pivotal $\C$-category.
  A \emph{right (left) ideal} $I$ is a full subcategory of $\cat$ that is:
  \begin{enumerate}
    \item \emph{closed under right (left) tensor products:} If $V$ is an object of $I$ and $W$ is any object of $\cat$, then $V \otimes W$ ($W \otimes V$) is an object of $I$.
    \item \emph{closed under retracts:} If $V$ is an object of $I$, $W$ is any object of $\cat$, and there are morphisms $f, g$ with
      \[
        \begin{tikzcd}
          W \arrow[r, "f"] \arrow[rr, bend right=40, swap, "\id_{W}"]  & V \arrow[r, "g"] & W
        \end{tikzcd}
      \]
      commuting, then $W$ is an object of $I$.
  
  An \emph{ideal} of $\cat$ is a full subcategory which is both a left and right ideal.
  \end{enumerate}
\end{definition}

\begin{proposition}
  Let $\cat$ be a  pivotal category.
  Then the projective and injective objects coincide and $\proj(\cat)$ is an ideal.
\end{proposition}
\begin{proof}
  See \cite[Lemma 17]{Geer2013a}.
\end{proof}

\begin{definition}
  Let $W$ be an object of a pivotal $\C$-category $\cat$.
  The \emph{right partial trace} is the map
  \[
    \tr^r_W : \hom_\cat(V \otimes W, X \otimes W) \to \hom_\cat(V, X)
  \]
  defined by
  \[
    \tr^r_W(g) = (\id_X \otimes \ev{W})(g \otimes \id_{W^*} ) (\id_V \otimes \coev{W} )
  \]
where $\ev{W} : \tsunit \to W \otimes W^*$ and $\coev{W} : W \otimes W^* \to \tsunit$ are the maps coming from the pivotal structure of $\cat$ and $\tsunit$ is the tensor unit of $\cat$.
(See Proposition $\ref{prop:pivotal-cat}$.)

  Now let $I$ be a right ideal in $\cat$.
  A \emph{(right) modified trace} (or m-trace) on $I$ is a family of $\C$-linear functions
  \[
    \{{\rentr}_V : \hom_\cat(V, V) \to \C \}_{V \in I}
  \]
  for every object $V$ of $I$ that are
  \begin{enumerate}
    \item \emph{compatible with partial traces:} If $V \in I$ and $W \in \cat$, then for any $f \in \hom_\cat(V \otimes W, V \otimes W)$,
      \[
        {\rentr}_{V \otimes W}(f) = {\rentr}_{V} \left( \tr^r_W(f) \right)
      \]
    \item \emph{cyclic:} If $U, V \in I$, then for any morphisms $f : V \to U$, $g: U \to V$, we have
      \[
        {\rentr}_V(gf) = {\rentr}_U(fg)
      \]
  \end{enumerate}
\end{definition}
We can similarly define left partial traces and left modified traces.
The pivotal structure on $\modcat$ means that a right modified trace on an ideal will also give a left modified trace.

With the usual graphical notation for pivotal categories, we can draw the right partial trace of a map $f : V \otimes W \to X \otimes W$ as
\begin{center}
  \includegraphics{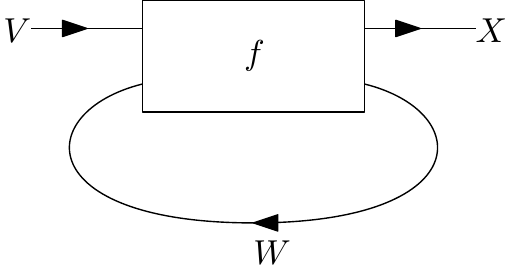}
\end{center}
Here we are breaking convention by writing the diagram left-to-right instead of vertically.
\footnote{When drawing string diagrams in this manner, we interpret the ``right'' in right trace to mean ``on the right as seen by $f$.''}

\subsection{Construction of modified traces}
Let $\cat$ be a pivotal $\C$-category with tensor unit $\tsunit$.
Consider the projective cover $P \to \tsunit$,\footnote{If $\modcat$ is semisimple, then $\tsunit$ is projective and we recover the usual trace in a pivotal category.} and assume that $P$ is finite-dimensional.
Then $P$ is indecomposable and projective and the space $\hom_\cat(P, \tsunit)$ is $1$-dimensional over $\C$.
Because $\cat$ is pivotal, $P$ is also injective and $\hom_\cat(\tsunit, P)$ is similarly $1$-dimensional.

The choice of $P$ and a basis of each space are the data necessary to define a modified trace on $\proj(\cat)$, which we call a trace tuple.
Our definition is a special case (\cite[\S5.3]{Geer2018}) of the more general trace tuples of \cite{Geer2018}, setting ${\color{RubineRed} \alpha} = {\color{RubineRed} \beta} = \tsunit$.
These more general traces can be defined for larger ideals than $\proj(\cat)$.
\begin{definition}
  Let $\cat$ be a pivotal $\C$-category with tensor unit $\tsunit$, and let $P \to \tsunit$ be a finite-dimensional cover.
  $(P, \iota, \pi)$ is a \emph{trace tuple} if $P$ is indecomposable and projective, $\iota$ is a basis of $\hom_\cat(\tsunit, P)$, and $\pi$ is a basis of $\hom_\cat(P, \tsunit)$.
\end{definition}
\begin{example}
  Let $\cat = \modcat$, the category of finite-dimensional $\U$-weight modules, and let $P_0$ be the projective cover of $\tsunit$ defined in \S\ref{appendix:unitmods}.
  Let $V$ be any of the irreducible $2$-dimensional modules of \S\ref{subsec:irrmods}.
  Then $(P_0, \coev{V}, \overline{\ev{V}})$ is a trace tuple.
\end{example}
Because $P$ is indecomposable, projective, and finite-dimensional, any endomorphism $f \in \End_\cat(P)$ decomposes $f = a + n$ as an automorphism plus a nilpotent part.
Because $\C$ is algebraically closed, $a$ is a scalar, and we write $\langle f \rangle = a \in \C$.

If $g \in \hom_\cat(\tsunit, P), h \in \hom_\cat(P, \tsunit)$ are any morphisms, we can similarly define $\langle g \rangle_\iota, \langle h \rangle_\pi \in \C$ by
\[
  g = \langle g \rangle_\iota \iota, \ \ h = \langle h \rangle_\pi \pi
\]

\begin{lemma}
  \label{lemma:brackets-agree}
  Let $(P, \iota, \pi)$ be a trace tuple.
  Then for any $f \in \End_\cat(P)$,
  \begin{enumerate}
    \item $\pi f = \langle f \rangle_\pi \pi$
    \item $f \iota = \langle f \rangle_\iota \iota$
    \item  $\langle f \rangle = \langle f \iota \rangle_\iota = \langle \pi f \rangle_\pi$
  \end{enumerate}
\end{lemma}
\begin{proof}
  We have $f = \langle f \rangle \id_P + n$ for some nilpotent $n$.
  The first statement follows from $\pi n = 0$.
  Since $\pi$ is a basis for $\hom_\cat(P, \tsunit)$, we have $\pi n = \lambda \pi$ for some $\lambda \in \C$.
  But $n^k = 0$ for some $k$, so $\lambda^k = 0 \Rightarrow \lambda = 0$ because $\C$ is an integral domain.
  The second statement follows from a similar argument, and the third from the first two.
\end{proof}

\begin{lemma}
  \label{lemma:lifts-exist}
  Let $(P, \iota, \pi)$ be a trace tuple for $\cat$ and $V$ a projective object.
  Then there are maps $\sigma_V : P \otimes V \to V$, $\tau_V : V \to P \otimes V$ such that the diagrams commute:
  \[
    \begin{tikzcd}
      & P \otimes V \arrow[dl, swap, "\sigma_V"] \\
      V & V \iso \tsunit \otimes V \arrow[l, "\iota \otimes \id_V"] \arrow[u, swap, "\iota \otimes \id_{V}"] 
    \end{tikzcd} \quad \quad
    \begin{tikzcd}
      & P \otimes V \arrow[d, "\pi \otimes \id_V"] \\
      V \arrow[ur, "\tau_V"] \arrow[r, swap, "\id_V"] & V \iso \tsunit \otimes V
    \end{tikzcd}
  \]
\end{lemma}
\begin{proof}
  $V$ is projective and  $\pi \ts \id_V: P \otimes V \to \tsunit \otimes V \to V$ is an epimorphism, so a lift $\tau_V$ exits.
  The dual argument works for $\sigma_V$.
\end{proof}

\begin{theorem}
  Let $(P, \iota, \pi)$ be a trace tuple for $\cat$ and choose maps as in Lemma \ref{lemma:lifts-exist}.
  Then there exits a right modified trace on $\proj(\cat)$ defined for $f \in \hom_\cat(V, V)$ by 
  \[
    {\rentr}_V(f) = \langle \tr_V^r(\tau_V f) \rangle_\iota = \langle \tr_V^r(\sigma_V f) \rangle_\pi
  \]

\end{theorem}
This is a special case of \cite[Theorem 4.4]{Geer2018}.
\begin{proof}
  In the diagrams in this proof, we identify
  \[
    \End_\cat(P)/J \iso \hom_\cat(\tsunit, P) \iso \hom_\cat(P, \tsunit) \iso \C
  \]
  via the maps $\langle - \rangle$, $\langle - \rangle_\iota$, and $\langle - \rangle_\pi$.
  Here $J$ is the ideal of nilpotent elements of $\End_\cat(P)$, so when we draw a diagram representing a morphism $P \to P$ we really mean its image in this quotient.

  $\tau_V$ and $\sigma_V$ exist by Lemma \ref{lemma:lifts-exist}, but are not unique.
  We show that the trace does not depend on the choice of either. 
  In graphical notation, $\tr_V^r(\tau_V f)$ can be written as
  \begin{center}
    \includegraphics{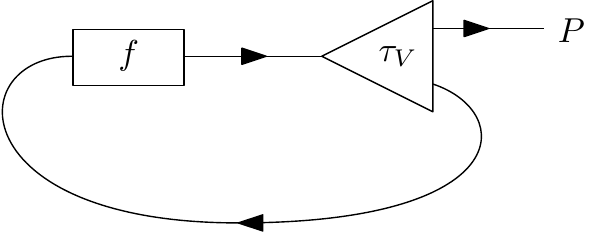}
  \end{center}
  Since $\sigma_V(\iota \otimes \id_V) = \id_V$, we can rewrite this morphism as
  \begin{center}
    \includegraphics{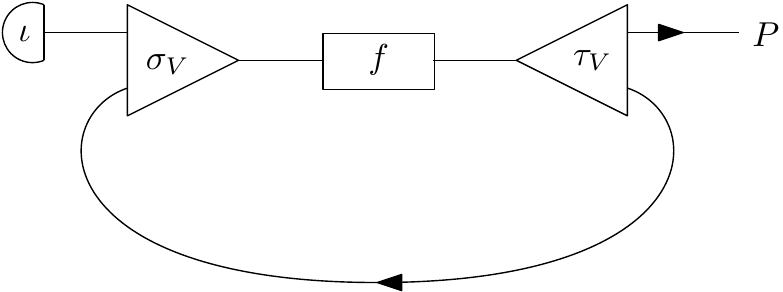}
  \end{center}
  where $\iota$ has no left-hand arrows because it is a map $\tsunit \to P$.
  By Lemma \ref{lemma:brackets-agree}, the above diagram is equal to
  \begin{center}
    \includegraphics{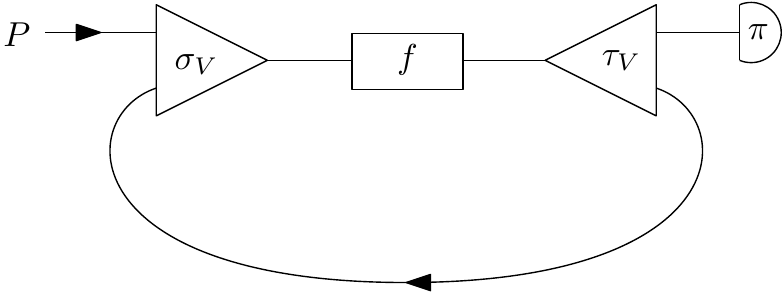}
  \end{center}
  But since $(\pi \otimes \id_V) \tau_V = \id_V$, this is equal to $\tr_V^r(f \sigma_V)$:
  \begin{center}
    \includegraphics{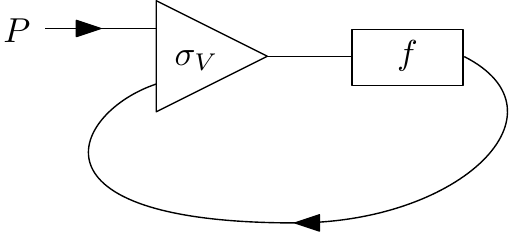}
  \end{center}
  It follows that 
  \[
    \langle \tr_V^r(\tau_V f) \rangle_\iota = \langle \tr_V^r(\sigma_V f) \rangle_\pi
  \]
  as claimed.

  To check the compatibility with the partial trace, let $f : V \otimes W \to V \otimes W$.
  Choose $\tau_V$ with $( \pi \otimes \id_V)\tau_V = \id_V$, and notice that we can set $\tau_{V \otimes W} = \tau_V \otimes \id_W$.
  Then ${\rentr}_{V \otimes W}(f)$ is
  \begin{center}
    \includegraphics{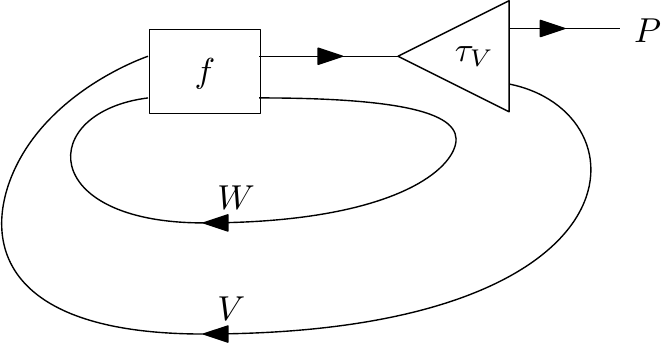}
  \end{center}
  which is clearly equal to ${\rentr}_V(\tr_W^r(f))$.

  Finally, we show cyclicity.
  Suppose $f : V \to W$ and $g : W \to V$.
  Then $\rentr_V(gf)$ is equal to
  \begin{center}
    \includegraphics{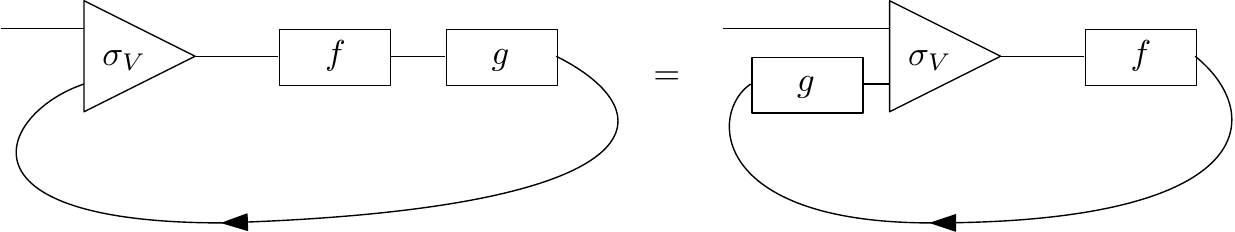}
  \end{center}
  by the cyclicity of the usual trace.
  But by inserting $(\pi \otimes \id_W) \tau_W = \id_W$ and then applying Lemma \ref{lemma:brackets-agree} as before, we can rewrite this as
  \begin{center}
    \includegraphics{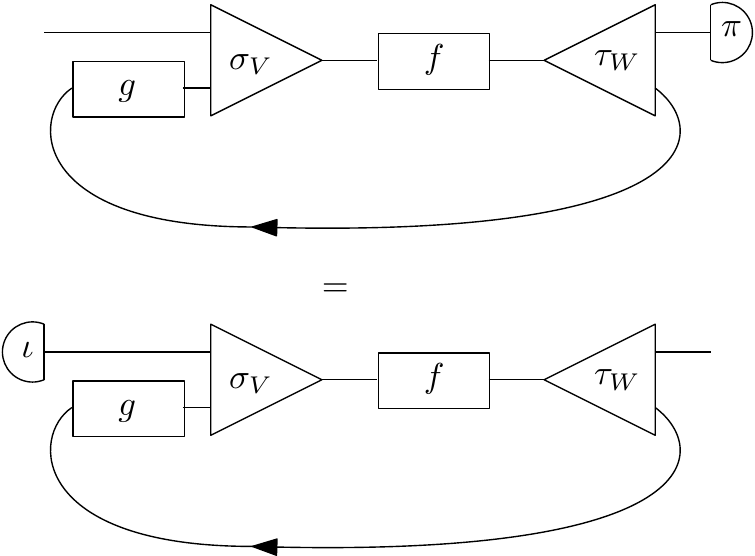}
  \end{center}
  By absorbing $\iota$ into $\sigma_V$, we see that this is equal to
  \begin{center}
    \includegraphics{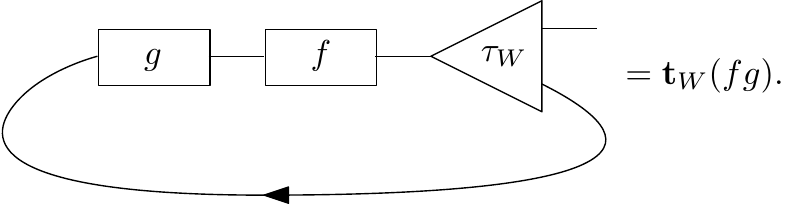}
  \end{center}
\end{proof}

It can be shown that the modified trace on $\proj(\cat)$ is essentially unique; choosing different $\iota$ or $\pi$ will simply change $\rentr$ by an overall scalar.
The paper \cite{Geer2018} proves this and a number of other useful results about these modified traces, such as non-degeneracy and compatibility with the left-hand version of the construction.

\subsection{Application to \texorpdfstring{$\U$}{U}}
Recall the projective $\U$-module $P_0$ defined in \S\ref{appendix:unitmods} with basis $x, y_1, y_2, z$.
As before, we can describe the action of $E$ and $F$ via the diagram
\[
  \begin{tikzcd}
    & x \arrow[dl, swap, "E"] \arrow[dr, "F"] \\
    y_1 \arrow[dr, swap, "F"] & & y_2 \arrow[dl, "E"] \\
                        & z
  \end{tikzcd}
\]
Write $\iota : \tsunit \to P_0$ for the linear map sending $1 \in \C$ to $z$, and $\pi : P_0 \to \tsunit$ for the projection onto the subspace spanned by $x$.
It is not hard to see that these are morphisms of $\U$-modules.

\begin{proposition}
  $(P_0, 2\iota, \pi)$ is a trace tuple defining the modified trace of \cref{prop:C-has-rentrs}.
\end{proposition}
\begin{proof}
  It is clear from Proposition \ref{prop:unit-cover-properties} that it is a trace tuple, so it suffices to check that it gives the same renormalized dimensions as in \cref{prop:C-has-rentrs}.
  Let $V = \irrmod{ \chi, \omega }$ be an irreducible $2$-dimensional module.
  It is not difficult to find a $\U$-module map $\tau_V$ with
  \[
    \begin{tikzcd}
      & P_0 \otimes V \arrow[d, "\pi \otimes \id_V"] \\
      V \arrow[ur, "\tau_V"] \arrow[r, swap, "\id_V"] & V \iso \tsunit \otimes V
    \end{tikzcd}
  \]
  Then we can check that
  \[
    \tr_V^r(\tau_v \id_V) = \frac{2}{\omega} \iota
  \]
  and since we chose $2 \iota$ in our trace tuple the renormalized dimension of $V$ is $\omega^{-1}$ as claimed.
\end{proof}

We conclude by the constructing the modified traces for $\doubledcat$.
\begin{proof}[Proof of Theorem \ref{thm:D-has-rentrs}]
  The modified trace on $\doubledcat$ is constructed using the trace tuple
  \[
    (P_0 \boxtimes P_0, 4 \iota \boxtimes \iota, \pi \boxtimes \pi)
  \]
  obtained as the product of the tuples for $\modcat$ and $\overline{\modcat}$.
  (Recall that $P_0^* \iso P_0$.)
  We show that this trace is compatible with the traces on the factors, in the sense that if $V, \overline{V}$ are objects and $f : V \to V$, $g : \overline{V} \to \overline{V}$ are morphisms in $\modcat$ and $\overline{\modcat}$, respectively, then
  \[
    \rentr( f \boxtimes g ) = \rentr(f) \rentr(g)
  \]
  The computation of the renormalized dimensions for $\doubledcat$ follows immediately.

  Choose lifts $\tau_V, \tau_{\overline{V}}$ as usual.
  Then the diagram
  \[
    \begin{tikzcd}
      & ( P_0 \boxtimes P_0 ) \otimes ( V \boxtimes \overline{V} ) \arrow[d, "(\pi \boxtimes \pi) \otimes (\id_V \boxtimes \id_{\overline{V}})"] \\
      V \boxtimes \overline{V} \arrow[ur, "\tau_V \boxtimes \tau_{\overline{V}}"] \arrow[r, swap, "\id_{V \boxtimes \overline{V}}"] & V \boxtimes \overline{V}
    \end{tikzcd}
  \]
  commutes, so $\tau_V \boxtimes \tau_{\overline{V}}$ is a lift for $V \boxtimes \overline{V}$.
  But then we can use the compatibility of the pivotal structures to write
  \begin{align*}
    \rentr(f \boxtimes g) &= \left\langle \tr_{V \boxtimes \overline{V}}^r( (\tau_V \boxtimes \tau_{\overline{V}}) (f \boxtimes g) ) \right\rangle_{\iota \boxtimes \iota} \\
                          &= \left\langle \tr_V^r(\tau_V f) \boxtimes \tr_{\overline{V}}^r(\tau_{\overline{V}} g) \right\rangle_{\iota \boxtimes \iota} \\
                          &= \left\langle \tr_V^r(\tau_V f) \right\rangle_\iota \left \langle \tr_{\overline{V}}^r(\tau_{\overline{V}} g) \right\rangle_{\iota} \\
                          &= \rentr(f) \rentr(g).
  \end{align*}
\end{proof}

\section{Proof of \cref{lemma:invariant-vector}}
\label{appendix:proof-of-invariant-vector}
The idea is to consider the $\U \otimes_\C \U^{\cop}$-modules $\dirrmod{\hat \chi_i}$ appearing in the image of $\doubledfunctor$ as $\U$-modules.
Then, writing $\otimes$ for the product of $\U$-modules,
\[
  \irrmod{\hat \chi_i} \otimes \irrmod{\hat \chi_i}^* \iso P_0
\]
where $P_0$ is the module of Definition \ref{def:unit-cover}.
$\dirrmod{\hat \chi_i}$ is \emph{not} the same as $P_0$, but we can still exploit this similarity to simplify our computations.

$P_0$ is an indecomposable but reducible module, and can be characterized as the injective hull of the tensor unit; the inclusion is the coevaluation map
\[
  \coev{V_i} : \tsunit \to  \irrmod{\hat \chi_i} \otimes \irrmod{\hat \chi_i}^* \iso P_0 
\]
whose image is (a constant times) the vector
\[
  z \defeq \ket 0 \otimes \bra 0 + \ket 1 \otimes \bra 1,
\]
writing $\ket 0, \ket 1$ for the usual basis of $\irrmod{\hat \chi_i}$ (see \S\ref{subsec:irrmods}) and $\bra 0, \bra 1$ for the dual basis of $\irrmod{\hat \chi_i}^*$.

The choice of inclusion map above fixes a canonical $\U$-module isomorphism $P_0 \to \irrmod{\hat \chi} \otimes \irrmod{\hat \chi}$ for every nonsingular $\hat \chi$.
Forgetting the $\U$-module structure, we have a family of vector space isomorphisms
\[
  f_{\hat \chi} : P_0 \to \irrmod{\hat \chi} \boxtimes \irrmod{\hat \chi}^*
\]
and we define the invariant vectors by
\[
  v_0(\hat \chi_1, \dots, \hat \chi_n) \defeq (f_{\hat \chi_1} \otimes \cdots \otimes f_{\hat \chi_n})(z \otimes \cdots \otimes z).
\]
We now need to prove that they are invariant.

It is enough to prove invariance in the case of a braid generator
\[
  \sigma : \dirrmod{\hat \chi_1} \otimes \dirrmod{\hat \chi_2} \to \dirrmod{\hat \chi_4} \otimes \dirrmod{\hat \chi_3}
\]
where the $\hat \chi_i$ are extended $\Ucentersmall$-characters related as usual by the braiding
\[
  \sigma : (\hat \chi_1, \hat \chi_2) \to (\hat \chi_4, \hat \chi_3).
\]
We need to show that
\[
  (c \boxtimes \overline{c})(v_0(\hat \chi_1, \hat \chi_2)) = \alpha v_0(\hat \chi_4, \hat \chi_3)
\]
for some nonzero $\alpha$, where $c$, $\overline{c}$ are holonomy braidings for $\irrmod{\hat \chi_1} \otimes \irrmod{\hat \chi_2}$ and $\irrmod{\hat \chi_1}^* \overline{\otimes} \irrmod{\hat \chi_2}^*$, respectively, and by $c \boxtimes \overline{c}$ we mean the operator acting on the tensor product
\[
  \irrmod{\hat \chi_1} \boxtimes \irrmod{\hat \chi_1}^* \otimes \irrmod{\hat \chi_2} \boxtimes \irrmod{\hat \chi_2}^*
\]
by $c$ in factors $1, 3$ and by $\overline{c}$ in factors $2, 4$.

In this computation, for elements $X, Y \in \U$ we write $X \boxtimes Y$ to distinguish elements of $\U \boxtimes \U^{\cop} \defeq \U \otimes_\C \U^{\cop}$ from elements of $\U \otimes \U$.
The two tensor products commute, in the sense that
\[
  (X \otimes Y) \boxtimes (Z \otimes W) = X \boxtimes Z \otimes Y \boxtimes W.
\]

Now consider the operators
\begin{align*}
  \gamma_0 &=  K \boxtimes K \otimes K \boxtimes K -1 \\
  \gamma_1 &= K \boxtimes K^{-1} \otimes 1 \boxtimes 1  - K^2 \boxtimes 1 \otimes 1 \boxtimes 1\\
  \gamma_2 &= E \boxtimes K \otimes 1 \boxtimes K + 1 \boxtimes E \otimes 1 \boxtimes K \\
  \gamma_3 &= 1 \boxtimes K^{-1} \otimes 1 \boxtimes F - 1 \boxtimes K^{-1} \otimes F \boxtimes K^{-1} 
\end{align*}
It is not hard to check that the kernel of these operators acting on $W(\hat \chi_1) \otimes W(\hat \chi_2)$ is spanned by $v_0(\hat \chi_1, \hat \chi_2)$.
Hence by the definition of holonomy braiding
\[
  0 = (c \boxtimes \overline{c})(\gamma_k \cdot v_0(\hat \chi_1, \hat \chi_2)) = (\algbraid \boxtimes \algbraidd)(\gamma_k) \cdot (c \boxtimes \overline{c})(v_0(\hat \chi_1, \hat \chi_2))
\]
for $k = 0, 1, 2, 3$.

It is immediate from the defining relations of the braiding operators $\algbraid$ and $\algbraidd$ that the images $\gamma_k' = (\algbraid \boxtimes \algbraidd)(\gamma_k)$ are
\begin{align*}
  \gamma_0' &= K \boxtimes K \otimes K \boxtimes K -1 \\
  \gamma_1' &= (1 \otimes K - i KF \otimes E) \boxtimes (1 \otimes K^{-1} - i F \otimes K^{-1} E) - (K^2 + K^2 F^2 \otimes E^2) \boxtimes (1 \otimes 1) \\
  \gamma_2' &= K \boxtimes K \otimes E \boxtimes K + 1 \boxtimes 1 \otimes 1 \boxtimes E \\
  \gamma_3' &= 1 \boxtimes F \otimes 1 \boxtimes 1 - F \boxtimes K^{-1} \otimes K^{-1} \boxtimes K^{-1} 
\end{align*}

We can compute that the kernel of the operators $\gamma'_k$ acting on $W(\hat \chi_4) \otimes W(\hat \chi_3)$ is spanned by $v_0(\hat \chi_4, \hat \chi_3)$.
Therefore $(c \boxtimes \overline{c})(v_0(\hat \chi_1, \hat \chi_2))$ must be proportional to $v_0(\hat \chi_4, \hat \chi_3)$ as claimed.

\printbibliography
\end{document}